\documentclass[10pt]{amsart}

\usepackage{amsmath,amsfonts,amssymb,amsthm}
\usepackage{xcolor,mhsetup,bm}
\usepackage[extra,safe]{tipa}
\usepackage{mathtools}
\usepackage{hyperref}
\usepackage{float}
\usepackage{tikz}
\usepackage{textcomp}
\usetikzlibrary{shapes}
\usepackage{enumitem}   
\usetikzlibrary{positioning}
\usetikzlibrary{shapes.geometric, arrows}
\usetikzlibrary{calc,decorations.pathmorphing,shapes,arrows,decorations.pathreplacing}
\usepackage[applemac]{inputenc}
\usepackage{mathrsfs}
\usepackage{mathabx}

\newtheorem{theorem}{Theorem}
\newtheorem{definition}{Definition}

\newtheorem{lem}{Lemma}

\newtheorem{corol}{Corollary}
\newtheorem{remark}[theorem]{Remark}

\def \R{\mathbb{R}}

\def\my_c{c_\infty}

\def \ltheta{\overleftarrow{\theta}}
\def \rtheta{\overrightarrow{\theta}}

\newcommand{\mynewtheorem}[2]{
	\newaliascnt{#1}{dummy}
	\newtheorem{#1}[#1]{#2}
	\aliascntresetthe{#1}
	\expandafter\def\csname #1autorefname\endcsname{#2}
}

\newcommand{\be}{\begin{equation}}
\newcommand{\ee}{\end{equation}}
\newcommand{\bde}{\begin{displaymath}}
\newcommand{\ede}{\end{displaymath}}
\newcommand{\beq}{\begin{eqnarray*}}
\newcommand{\eeq}{\end{eqnarray*}}
\newcommand{\beqa}{\begin{eqnarray}}
\newcommand{\eeqa}{\end{eqnarray}}
\newcommand{\bel }{\left\{\begin{array}{ll}}
	\newcommand{\eel}{\cr \end{array} \right.}

\newcommand{\seq}[1]{{\lbrace #1 \rbrace}}

\newcommand{\dcb}{\begin{array}{lll}}
	\newcommand{\dce}{\end{array}}
\newcommand{\ebe}{\begin{enumerate}\setlength{\baselineskip}{13pt}\setlength{\parskip}{0pt}}
	\newcommand{\dbe}{\end{enumerate}}

\newcommand{\E}{\mathcal{E}}

\def\F{{\cal F}}

\def\rr{{\mathbb R}}

\def\P{{\mathbb P}}

\def\I{\mathsf{1}}

\newcommand{\leftrightharpoonup}{%
	\mathrel{\mathpalette\lrhup\relax}%
}
\newcommand{\lrhup}[2]{%
	\ooalign{$#1\leftharpoonup$\cr$#1\rightharpoonup$\cr}%
}
\newcommand\rharp[1]{\mathstrut\mkern2.5mu#1\mkern-11mu\raise1.2ex%
	\hbox{$\scriptscriptstyle\rightharpoonup$}}
\newcommand\lharp[1]{\mathstrut\mkern2.5mu#1\mkern-11mu\raise1.2ex%
	\hbox{$\scriptscriptstyle\leftharpoonup$}}
\newcommand\lrharp[1]{\mathstrut\mkern2.5mu#1\mkern-11mu\raise1.2ex%
	\hbox{$\scriptscriptstyle\leftrightharpoonup$}}

\newcommand \A[1]{{\bf (#1)}}

\def\F{{\mathcal F}}

\def\R{{\mathbb{R}} }
\def\N{{\mathbb{N}} }
\def\E{{\mathbb{E}}  }

\def\P{{\mathbb{P}}  }

\def\I{{\mathbf{1}}}

\def\bint#1^#2{\displaystyle{\int_{#1}^{#2}}}
\def\bsum#1^#2{\displaystyle{\sum_{#1}^{#2}}}

\def\xdt_#1{X_#1(\Delta t)}

\def\0{{\mathbf{0}}}

\begin{document}

\title[IBP formula for killed process]{Integration by parts formula for killed processes: A point of view from approximation theory}

\author{Noufel Frikha}
\address{Noufel Frikha, Universit\'e de Paris, Laboratoire de Probabilit\'es, Statistiques et Mod\'elisation, F-75013 Paris, France}
\email{frikha@math.univ-paris-diderot.fr}


\author{Arturo Kohatsu-Higa}
\address{Arturo Kohatsu-Higa\footnote{This author was supported by grants of the Japanese government KAKENHI 16K05215 and 16H03642},
Department of Mathematical Sciences
Ritsumeikan University
1-1-1 Nojihigashi, Kusatsu, Shiga, 525-8577, Japan}
\email{khts00@fc.ritsumei.ac.jp}

\author{Libo Li}
\address{Libo Li, Department of Mathematics and Statistics, University of New South Wales, Sydney, Australia}
\email{libo.li@unsw.edu.au }


\begin{abstract}
In this paper, we establish a probabilistic representation for two integration by parts formulas, one being of Bismut-Elworthy-Li's type, for the marginal law of a one-dimensional diffusion process killed at a given level. These formulas are established by combining a Markovian perturbation argument with a tailor-made Malliavin calculus for the underlying Markov chain structure involved in the probabilistic representation of the original marginal law. Among other applications, an unbiased Monte Carlo path simulation method for both integration by parts formula stems from the previous probabilistic representations. 
\end{abstract}
\maketitle

\section{Introduction}

In this article, we consider the following one-dimensional stochastic differential equation (SDE in short)
\begin{equation}
\label{sde:dynamics}
X_t = x + \int_0^{t} b(X_s) ds + \int_{0}^t \sigma(X_s) dW_s, \, x\in \rr
\end{equation}

\noindent where the coefficients $b, \, \sigma: \rr \rightarrow \rr $ are smooth and bounded functions and $(W_t)_{t \geq 0}$ stands for a one-dimensional Brownian motion on a given filtered probability space $(\Omega,\F, (\F_t)_{t\ge 0},\P) $. 

The aim of this paper is to provide a probabilistic representation for two integration by parts (IBP) formulas for the marginal law of the process $X$ killed at a fixed given level $L$. To be more specific, for a starting point $x\geq  L$, let $\tau = \inf\seq{t \geq 0:  X_t < L} $ be the first hitting time of the level $L$ by the one-dimensional process $X$. For a given finite horizon $T>0$, we are interested in establishing probabilistic representations for IBP formulae related to the following quantities
\begin{align}
\E[f'(X_T) \I_\seq{\tau >T}] \quad \mbox{ and } \quad \partial_x \E[f(X_T) \I_\seq{\tau >T}] \label{the:two:ibp:formulas}
\end{align}

\noindent where $f$ is a real-valued smooth function defined on $[L, \infty)$. Extensions to non-smooth functions or to the transition density of the killed process are also obtained.

In the recent past, IBP formulae have raised a lot of interest as these explicit formulae can be further analyzed to obtain properties of densities or for Monte Carlo simulation among other applications, see Nualart \cite{Nuabook} or Malliavin and Thalmaier \cite{malliavinT}. The former quantity in \eqref{the:two:ibp:formulas} is commonly considered in the literature on Malliavin calculus, while the latter quantity is referred in the literature to as the Bismut-Elworthy-Li (BEL for short) formula. The BEL formula is of interest for many practical applications such as the numerical computation of price sensitivities in finance for Delta hedging purpose. For a more detailed discussion on this topic, we refer the interested reader to Fourni\'e and al. \cite{fournie:1}, Fourni\'e and al. \cite{fournie:2}, Gobet et al. \cite{Gobet:Kohatsu:1} \cite{Gobet:Kohatsu:2} for a short sample.

In particular, in Section 2.6 of \cite{malliavinT}, the authors propose a continuous time version of the IBP formula for $d$-dimensional diffusion process $X$ killed when it exits an open sub-domain $ D $ of $\rr^d$. Denoting by $\tau$ the first exit time of $X$ from $D$, the formula writes $\partial_x \E[f(X_T) \I_\seq{\tau>T}]=\E[f(X_T)\I_\seq{\tau>T}H]$ where $ H $ has an explicit expression using stochastic integrals. Extensions have also been proposed by Arnaudon and Thalmaier \cite{ADT}, \cite{AT}. At this stage, it is important to observe that, from a numerical perspective, these formulae will inevitably involve a time discretization, thus introducing a bias, in order to devise a Monte Carlo simulation method, as it is already the case for the quantity $\E[f(X_T) \I_\seq{\tau>T}]$, see e.g. Gobet \cite{Gobet:2000}, Gobet and Menozzi \cite{gobe:meno:spa:04}. As observed in \cite{AT}, it may also require to compute the solution of a control problem which can be done off-line. One may thus claim as stated at the beginning of Section 2.6 in \cite{malliavinT}: ``Its Monte-Carlo implementation, which has not yet been done, seems to be relatively expensive in computing time''. 

Our approach is probabilistic and relies on a perturbation argument of Markov semigroups to derive a probabilistic representation for the marginal law of the killed process based on a simple Markov chain approximation scheme for which we develop an appropriate Malliavin calculus machinery. 

The main novelty in comparison with the aforementioned previous works lies in the fact that an unbiased Monte Carlo simulation method directly stems from the integration by parts formulae derived here. One may thus devise an estimator which does not involve any bias but only a statistical error. To the best of our knowledge, this feature appears to be new. As a by product of our analysis, we propose a probabilistic representation for the two derivatives $\partial_x p(T, x, z)$ and $\partial_z p(T, x, z)$ where $(0,\infty) \times [L,\infty)^2 \ni (T, x, z)\mapsto p(T, x, z)$ stands for the transition density evaluated at terminal point $z$ at time $T$ of the process $X$ starting from $x$ and killed at level $L$. We also point out that devising a Monte Carlo estimator for IBP formulae without introducing a bias from the exact simulation methods of Jenkins \cite{Jenkins} or Herrmann and Zucca \cite{Herrmann} does not seem to be apparent. An extension of the exact simulation method introduced by Beskos and al. \cite{beskos2006} to compute the two first derivatives with respect to the starting point $x$ of the quantity $ \E[f(X^{x}_T)]$, $X$ being a one-dimensional diffusion with constant diffusion coefficient, has been recently proposed by Tanr\'e and Reutenauer \cite{reutenauer:tanre}. However, it seems difficult to implement from this method a Monte Carlo estimator for the aforementioned derivatives of the transition density of the diffusion without introducing any bias. 

The first step towards obtaining an IBP formula is to prove a probabilistic representation for the marginal law of the killed diffusion process, in the spirit of Bally and Kohatsu-Higa \cite{BK} which developed such a formula for multi-dimensional diffusion processes (without stopping) and some L\'evy driven SDEs by means of a probabilistic perturbation argument for Markov semigroups. We also refer the reader to Labord\`ere et al. \cite{henry-labordere2017} and Agarwal and Gobet \cite{GobetA} for some recent contributions in that direction for multi-dimensional diffusion processes.

 Such representation involves a simple Markov chain structure evolving along a time grid given by the jump times of an independent renewal process. A similar representation was derived by the same authors in \cite{FKL1} by means of analytic arguments. However, the representation obtained here is different and more amenable for the implementation of Monte Carlo simulation methods or to establish IBP formulae. 

Once such representation is established, we then want to prove suitable IBP formulae using the underlying Markov chain structure and eventually the noise provided by the jump times. For this purpose, we set up a tailor-made Malliavin calculus for this new approximation process and perform a careful propagation argument of spatial derivatives, backward in time for the first quantity in \eqref{the:two:ibp:formulas} and forward in time for the second quantity in \eqref{the:two:ibp:formulas}.

These developments are not free of mathematical hurdles. In fact, the proposed methodology leads to the appearance of boundary terms which have to be treated carefully. The key idea that we develop to deal with this issue consists in using the noise provided by the jump times. Contrary to the IBP formulae developed here, we point out that in most cases the explicitness of the IBP formulae for diffusions killed at a boundary demands a number of simplifications and approximations. A technical argument commonly used consists in performing a localisation of the underlying process in order to ensure that it is not close to the boundary. This technique is successful but has some theoretical and practical limitations as shown in Delarue \cite{Delarue}, \cite{Gobet:Kohatsu:1} and Nakatsu \cite{Nakatsu2}. This is one of the main reasons for the previously quoted statement in Section 2.6 of \cite{malliavinT}.

We finally emphasize that the variance of the Monte Carlo estimators associated to the IBP formulae established here tends to be large and even infinite. This feature is not new and appears to be reminiscent of the probabilistic representation originally obtained in \cite{BK}. Importance sampling or higher order methods have been proposed to circumvent this issue in the case of multi-dimensional diffusions, see Andersson and Kohatsu-Higa \cite{APKA} and more recently Andersson et al. \cite{AKY}. Following the ideas developed in \cite{AKY}, we show how to achieve finite variance for the Monte Carlo estimators obtained from the probabilistic representation formulas of the marginal law of the killed process and of both IBP formulas by employing an importance sampling scheme on the jump times of the renewal process. We finally provide some numerical tests illustrating our previous analysis.

The article is organized as follows. In Section \ref{preliminaries:sec}, we provide some basic definitions, assumptions and present a reflection principle based on a simple one step Markov chain that will play a central role in our probabilistic representations for the marginal law of the killed process and for our IBP formulae. In addition, we also construct the adequate Malliavin calculus machinery related to the underlying Markov chain upon which both IBP formulae are made. In Section \ref{prob:repres:sec}, the probabilistic representation for the marginal law of the killed process, based on the Markov chain of Section \ref{preliminaries:sec}, is established. The change from the process $X$ to the Markov chain coming from the reflection principle of Section \ref{preliminaries:sec} simplifies our analysis as all the irregularities of the process appear as indicator functions. 
Section \ref{sec:ingr} is devoted to the main ingredients to obtain our first IBP formula, that is, the IBP formula for the quantity $\E[f'(X_T) \I_\seq{\tau >T}]$.  

These ingredients are put to work in Section \ref{sec:!IBP}. Theorem \ref{prop:5.2} is the main result of this section. As a by product, we obtain a probabilistic representation for the first derivative of the transition density of the killed process with respect to its terminal point in Corollary \ref{cor:ibp:backward}. In Section \ref{BEL:sec}, we establish the BEL formula for the law of the killed process. The main result of this section is Theorem \ref{thm:forward:ibp} and, as a by product, we obtain a probabilistic representation for the first derivative of the transition density of the killed process with respect to the initial condition $x$ in Corollary \ref{cor:ibp:forward}. Many of the proofs of Sections \ref{prob:repres:sec} and \ref{sec:ingr} are technical and postponed to the appendix in Section \ref{app:sec}. In Section \ref{importance:sampling:section}, we show how to achieve finite variance for our unbiased Monte Carlo estimators by an importance sampling technique that we briefly present. Some numerical results are presented in Section \ref{sec:7}. Clearly, one needs to study numerical issues in more detail and these are left for later studies.

\subsection*{Notations }\label{mal:calculus}
For a fixed given point $z\in \rr$, the Dirac measure is denoted by $\delta_z(dx)$. Derivatives may be denoted by $f'(x) $ in the one-dimensional case or by $ \partial_if (x)$ in the multi-dimensional case for the partial derivative with respect to the $ i $-th variable appearing in the multivariate function $ f $ or also by $ \partial^i_xf(x)\equiv f^{(i) }(x)$ where the latest is used in functions of one variable and the former is used mostly for multivariate functions. 

We will often work with continuous or smooth functions defined on $ [L,\infty) $. In order to shorten notation, we will consider their extensions\footnote{In most cases, unless explicitly said we assume that functions are extended using e.g. Whitney's extension theorem.} on $ \mathbb{R} $. This is done just in order to shorten the length of equations and notation. Therefore all statements can be rewritten using the same assumptions but restricted to the domain $ [L,\infty) $.

The space of functions which are $ k $-times continuously differentiable on a closed domain $ D $  is denoted by $ \mathscr{C}^k(D) $. At the boundary points of $ D $, all derivatives are considered as limits taken from the interior of $ D $ only. In the case that the derivatives are bounded, the space of corresponding functions is denoted by $\mathscr{C}^k_b(D) $, $ k\in\mathbb{N}\cup\{0,\infty\} $. In particular, note that functions in $\mathscr{C}^1_b $ may not be bounded but are at most linearly growing. Finally, $  \mathscr{C}^k_p(D) $  denote the class of $ k $-times differentiable functions with at most polynomial growth at infinity. The space of $ p $-integrable random variables (r.v.'s) is denoted by $ \mathbb{L}^p $ with its extension $ \mathbb{L}^\infty $ for $p=\infty $. 

Given a measure space $S$, the space of real valued Borel measurable functions on $S$ will be denoted by $\mathcal{M}(S)$. We also introduce the simplex $ A_{n}:=\{t\in (0,T]^n;0<t_1<\cdots<t_n\leq T\}$, $ n\in\mathbb{N} $ where $T$ is fixed throughout the paper.

 The transition density function at $ x $ of the standard Brownian motion at time $t$ is denoted by $g(t, x) = (2\pi t )^{-1/2}\exp(-x^2/(2t))$. Its associated Hermite polynomials of order $i$, $i \in \mathbb{N}$, are defined by $\mathcal{H}_i(t, x)= (g( t, x))^{-1} \partial^{i}_x g(t,x)$.

%
%
%
%
%



In order to simplify lengthy equation, we may use the symbol $ \stackrel{\E}{=} $ to mean that two quantities are equal in expectation. Sometimes the same symbol maybe used for equality on conditional expectation. This will be clearly indicated at the point where it is used.

Generic constants are usually denoted by $ C $ and are independent of all variables unless otherwise explicitly stated. As usual they may change value from one line to the next.


As we will be using discrete Markov chains in this article, we will often have indexes whose range may be the set of integers. In order to shorten the length of statements, we will use the following notation for the most common set of indexes $ \mathbb{N}_n\equiv \{1,...,n\}$ or $ \bar{\mathbb{N}}_n \equiv \{0,...,n\}$ with $ n\in\bar{\mathbb{N}}\equiv\mathbb{N}\cup\{0\} $. In the case that $ n \leq 0 $, then $ \mathbb{N}_n:=\emptyset $.

\section{Preliminaries}\label{preliminaries:sec}

\subsection{Assumptions and basic definitions}
Throughout the article, we work on a probability space $(\Omega, \mathcal{F}, \mathbb{P})$ which is assumed to be rich enough to support all r.v's considered in what follows. In addition, we will work under the following assumptions on the coefficients:
\subsection*{Assumption (H)}
\begin{itemize}
	\item[(i)] The coefficients of the SDE \eqref{sde:dynamics} are smooth and bounded, in particular, $b\in \mathscr{C}^{\infty}_b(\rr)$ and $ a \in \mathscr{C}^{\infty}_b(\rr)$. 
	\item[(ii)] The function $\sigma$ is bounded and uniformly elliptic, that is, there exist $\underline{a}, \overline{a}>0$ such that for any $x\in \rr$, $\underline{a} \leq a(x)=\sigma^2(x) \leq \overline{a}$.
	Therefore, without loss of generality, we will assume that $ \sigma(x)>0 $.
\end{itemize}

\subsection{A reflection principle}
Our probabilistic representation involves the following approximation process 
\begin{align}
\label{eq:defbX}
\bar{X}^{s, x}_{t}\equiv \bar{X}^{s, x}_{t}(\rho) = \rho x+(1-\rho)(2L-x)+\sigma(x)(W_t-W_s),
\end{align} 

\noindent where $ \rho $ is a r.v. distributed according to a Bernoulli(1/2) law, independent of $ W $, namely $\P(\rho=1) = \P(\rho=0) = 1/2 $. Under our assumption on the coefficients, the flow derivatives of $\bar{X}^{s, x}_t$ exist. In particular, one has
\begin{align}
\label{eq:fd}	\partial_{x}\bar X^{s,x}_{t}=&2\rho-1+\sigma'(x){(W_t-W_s)},	\quad \partial^2_{{x}}\bar X^{s,x}_{t}=
\sigma''(x){(W_t-W_s)}.
\end{align}

 In the particular case that $ s=0 $, we may use the simplified form $ \bar{X}_t=\bar{X}^{0,x}_{t} $. At this point, we give a brief explanation about how the approximation process $ \bar{X} $ appears in the forthcoming probabilistic representation. The proof of the following lemma is straightforward by using the reflection principle, see e.g. Karatzas and Shreve \cite{Karatzas1991}.
\begin{lem} 
	\label{lem:1}Define the following approximation process:
	\begin{align*}
	\bar{Y}_t=x+\sigma(x)W_t, \quad x\geq L,
	\end{align*}
together with its associated exit time $ \bar{\tau}:=\inf\left\{t,\bar{Y}_t=L\right\} $. Then, for any bounded measurable function $f$, the following property is satisfied:
		\begin{align}
	\E\left[f(\bar{Y}_T)\I_\seq{\bar{\tau} >T}\right]=&
	\E\left[f(\bar{Y}_T)\I_\seq{\bar{Y}_T\geq L}\right]-
	\E\left[f(2L-\bar{Y}_T)\I_\seq{\bar{Y}_T<L}\right] = 2\E\left[(2\rho-1)f(\bar{X}_T)\I_\seq{\bar{X}_T\geq L}\right].\label{eq:refa}
	\end{align}
\end{lem}


\subsection{Basic Markov chain framework}

We also consider a Poisson process with parameter $\lambda>0$, independent of the one-dimensional Brownian motion $W$ with jump times $T_i$, $i\in \mathbb{N}$ and we set $ \zeta_i:= T_i\wedge T$, $ i\in\mathbb{N} $ with the convention that $\zeta_0= T_0 = 0$. 

Define $ \pi $ to be the partition of $ [0,T]$ given by $\pi:=\{0=:\zeta_0<\cdots<\zeta_{N_T} \leq T\} $. Associated with this set, we recall the definition of the simplex $ A_{n}:=\{t\in (0,T]^n;0<t_1<\cdots<t_n\leq T\}$. For instance, on the set $ \{N_T=n\} $, $ n\in\mathbb{N} $, we have $ (\zeta_1,\dots,\zeta_n)\in A_n $ and $\zeta_{n+1} = T$.  In particular, for the set $ \{N_T=0\} $ (i.e. $ n=0 $), we let $ \pi:=\{0,T\} $ and $ A_0=\emptyset $. In this sense, we will use throughout the rest of the paper the index $ n\in\bar{\mathbb{N}} $ without any further mention of its range of values.

As it is the case in the previous observation, many proofs and definitions will be carried out conditioning on the set  $\seq{N_T = n}$, $ n\in\bar{\mathbb{N}} $. 

Let $\bar{X}:= (\bar{X}_i)_{ i \in\bar{\mathbb{N}}} $ be the discrete time Markov chain starting at time $0$ from $\bar{X}_0 = x$ and evolving according to 
\begin{align}
\label{eq:MCa}
\bar{X}_{i+1}:=&\rho_{i+1}\bar{X}_i+(1-\rho_{i+1}) (2L-\bar{X}_i)+\sigma_i Z_{i+1}, \,\quad  i \in\bar{\mathbb{N}}_ {N_T},
\end{align} 

\noindent where for simplicity we set $ \sigma_i:=\sigma(\bar{X}_i) $, $ Z_{i+1}:= W_{\zeta_{i+1}} - W_{\zeta_i}=\sigma_i^{-1}\left(\bar{X}_{i+1}-\rho_{i+1}\bar{X}_i-(1-\rho_{i+1})(2L-\bar{X}_i)\right)$ and $\{\rho_i;i\in\mathbb{N}\} $ is an i.i.d. sequence of Bernoulli$(1/2)$ r.v.'s such that $W, \, N $ and $\{\rho_i;i\in\mathbb{N}\} $ are mutually independent. In what follows we use the notation $ h_i\equiv h(\bar{X}_i)$, $ i\in\bar{\mathbb{N}}_{n+1} $ for any function $ h:\mathbb{R}\rightarrow\mathbb{R} $. In particular, the reader may have noticed that we already used this notation in the above formula for $ h=\sigma $. We also associate to the Markov chain $\bar{X}$ the following sets
\begin{equation}
\label{set:Din:X}
 D_{i,n}:=\{\bar{X}_{i}\geq L, N_T=n\} , \quad \mbox{ for } i\in \bar{\N}_{n+1}.
\end{equation}


 It is important to point out that given $ N_T=n $, the conditional distribution of $\bar X_{n+1}$ given $\bar X_{n}$ is not the same as the conditional distribution of $\bar X_{i+1}$ given $\bar X_{i}$ for $i\in\bar{\mathbb{N} }_{n-1}$. This is due to the fact that the length of the last interval is $T-\zeta_{n}$, rather than the waiting time between two consecutive Poisson jumps. 
 This remark applies to various definitions and results to be stated through the rest of the article.

 We define the filtration $\mathcal{G} := (\mathcal{G}_i)_{i\in \bar{\mathbb{N}}}$ where $\mathcal{G}_i := \sigma(Z^i,\zeta^i,\rho^i) $ with the notation $ a^i :=(a_1,\dots,a_i) $ for $ a=Z,\ \zeta,\ \rho $, $ i \in\mathbb{N}$ and $ \mathcal{G} _0$ defined as the trivial $ \sigma $-field. 
We assume that the filtration $\mathcal{G}$ satisfies the usual conditions.

\subsection{Simplified Malliavin Calculus for the underlying Markov chain}
\label{sec:3a}
In this section we introduce the required material for our Malliavin calculus computations. 
Instead of using an infinite dimensional calculus as it is usually done in the literature, the approach developed below is based on a finite dimensional calculus for which the dimension is given by the number of jumps of the underlying Poisson process involved in the Markov chain $\bar{X}$. In what follows, $ n\in\bar{\mathbb{N}} $ unless stated otherwise.

We start by defining the following space of smooth r.v.'s.

\begin{definition}
	\label{def:1}
	For $ i\in\bar{\mathbb{N}}_n$, we define the set $ {\mathbb{S}}_{i+1,n}(\bar{X}) $ as the subset of r.v.'s $ H\in \mathbb{L}^0 $ such that there exists a measurable function $ h:\mathbb{R}^2\times \{0,1\}\times A_2\rightarrow \mathbb{R} $ satisfying
	\begin{enumerate}
		\item $\displaystyle{ H=
			h(\bar{X}_i,\bar{X}_{i+1},\rho_{i+1},\zeta_i,\zeta_{i+1}) 
			}
			$ on the set $\{N_T=n\} . $ 
		\item For any $ r\in\{0,1\} $ and any $ (s,t)\in A_2 $, $ h(\cdot,\cdot,r,s,t)\in\mathscr{C}_p^\infty(\mathbb{R}^2) $.
	\end{enumerate}
\end{definition} 

For a r.v. $ H\in  {\mathbb{S}}_{i+1,n}(\bar{X})$, $ i\in\bar{\mathbb{N}}_n$, we may sometimes abuse the notation and write 
\begin{equation}
\label{abuse:notation:space}
 H\equiv H(\bar{X}_i,\bar{X}_{i+1},\rho_{i+1}, \zeta_i,\zeta_{i+1}),
\end{equation}
that is the same symbol $ H $ may denote the r.v. or the function in the set $ {\mathbb{S}}_{i+1,n}(\bar{X}) $.  One can easily define the flow derivatives for $ H\in\mathbb{S}_{i+1,n}(\bar{X}) $ as follows:
\begin{align}
\partial_{\bar{X}_{i+1}}H:=&\partial_2h
(\bar{X}_i,\bar{X}_{i+1},\rho_{i+1},\zeta_i,\zeta_{i+1}), \nonumber\\
\partial_{\bar{X}_{i}}H:=&\partial_1h(\bar{X}_i,\bar{X}_{i+1},\rho_{i+1},\zeta_i,\zeta_{i+1})+\partial_2h(\bar{X}_i,\bar{X}_{i+1},\rho_{i+1},\zeta_i,\zeta_{i+1})\partial_{\bar{X}_{i}}\bar{X}_{i+1}, \label{eq:flow} \\
\partial_{\bar{X}_{i}}\bar{X}_{i+1}:=&
(2\rho_{i+1}-1)+\sigma'_i Z_{i+1}.
\nonumber
\end{align}

We now define the derivative and integral operators for $ H\in {\mathbb{S}}_{i+1,n}(\bar{X}) $, $ i\in\bar{\mathbb {N}}_n $, as 
\begin{align}
\label{eq:I1}
\mathcal{I}_{i+1}(H)
:= &H \frac{ Z_{i+1}}{\sigma_i(\zeta_{i+1}-\zeta_i) }- {\mathcal{D}_{i+1} H}, \qquad \mathcal{D}_{i+1}H:=\partial_{\bar{X}_{i+1}}H.
\end{align}

Note that due to the above definitions and Assumption $ \mathbf{(H)} $, we also have that  $ \mathcal{I}_{i+1}(H),\mathcal{D}_{i+1}H \in{\mathbb{S}}_{i+1,n}(\bar{X})$ so that we can define iterations of the above operators, namely $ \mathcal{I}_{i+1}^{\ell+1}(H) = \mathcal{I}_{i+1}(\mathcal{I}^{\ell}_{i+1}(H))$ and similarly $\mathcal{D}^{\ell+1}_{i+1}(H) = \mathcal{D}_{i+1}(\mathcal{D}^{\ell}_{i+1}H)$, $ \ell\in\bar {\mathbb{N}} $, with the convention $\mathcal{I}^{0}_{i+1} (H) = \mathcal{D}^{0}_{i+1}(H) = H$.

Through this article, we will use the following notation for a certain type of conditional expectation that will appear frequently.
 For any $X\in \mathbb{L}^{1}$ and any $ i\in \bar{\mathbb{N}}_{n}$, 
\begin{align*}
	\E_{i,n}[X]:=\E[X\,\vert\,\mathcal{G}_i,T^{n+1},\rho^{n+1}, N_T=n].
\end{align*}

With the above definitions, the following duality\footnote{This duality is obtained using the Gaussian density of $ \bar{X}_{i+1} $ while in classical Malliavin calculus it is based on the density of the Wiener process. Therefore the derivative and integral, $ \mathcal{D}_{i+1} $ and $ \mathcal{I}_{i+1} $, $ i\in\bar{\mathbb{N}}_n $ defined in the  formula \eqref{eq:IBP} are renormalizations of the usual duality principle in Malliavin calculus. In our case this notation simplifies greatly many equations.} is satisfied for any $ f\in\mathscr {C}^1_p(\rr)$ and any $ (i,\ell) \in\bar{\mathbb{N}}_n \times \mathbb{N}$:
\begin{align}
\label{eq:IBP}
\E_{i,n}\left[{\mathcal{D}}^\ell_{i+1}f(\bar{X}_{i+1})H\right]=&\E_{i,n}\left[f(\bar{X}_{i+1}){\mathcal{I}}^\ell_{i+1}(H)\right].
\end{align}
 In order to obtain explicit norm estimates for r.v.'s in $ {\mathbb{S}}_{i+1,n}(\bar{X}) $ it is useful to define for $H\in \mathbb{S}_{i+1,n}(\bar{X})$, $ i\in\bar{\mathbb{N}}_n $ and $p\geq1$
$$
\|H\|_{p, i, n}^p:=\E_{i,n}\left[|H|^p\right].
$$

  Another useful formula that is used at several places later on is the following \emph{extraction formula} for $ H_1,H_2 \in\mathbb{S}_{i+1,n}(\bar{X})$ :
\begin{align}
\label{eq:exta}
\mathcal{I}^\ell_{i+1}(H_1H_2)=\sum_{j=0}^{\ell}(-1)^{j}\binom{\ell}{j}
\mathcal{I}_{i+1}^{\ell-j}(H_1)\mathcal{D}^{j}_{i+1}H_2.
\end{align}
The proof of the above statement is done by induction. Then, by iteration, one obtains that $ {\mathcal{I}}^\ell_{i}(1)\in \mathbb{S}_{i,n}(\bar{X}) $ and it satisfies ${\mathcal{I}}^{\ell+1}_{i}(1)={\mathcal{I}}^\ell_{i}(1){\mathcal{I}}_{i}(1) -\ell {\mathcal{I}}^\ell_{i}(1)$ which, in particular, implies:
\begin{align}
\label{eq:link:integral:hermite:pol}
&{{\mathcal{I}}}_{i}^\ell(1)=(-1)^\ell \mathcal{H}_\ell({a_{i-1}}(\zeta_i-\zeta_{i-1}),\sigma_{i-1} Z_i).
\end{align}

 Using \eqref{eq:exta} for $ H_1=1 $ and $ H_2=H $ as well as \eqref{eq:I1}, the following norm bound for stochastic integrals is clearly satisfied for $ i\in \bar{\mathbb{N}}$ and any measurable set $ A\in\mathcal{F} $
\begin{align}
\label{eq:Hesta}
\left\|\I_{A}{{\mathcal{I}}}_{i+1}^{\ell}(H)\right\|_{p, i, n}\leq C_{\ell,p}\sum_{j=0}^{\ell}(\zeta_{i+1}-\zeta_i)^{\frac{j-\ell}{2}}\|\I_{A}\mathcal{D}^j_{i+1}H\|_{p, i, n}.
\end{align}

The following H\"older like inequality for smooth r.v.'s $ H_1,H_2 \in\mathbb{S}_{i+1,n}(\bar{X})$ is also frequently used in our computations without any further mentioning for any $ i\in \bar{\mathbb{N}} $, for any $ p,p_1,p_2\geq 1 $ satisfying $ p^{-1}=p_1^{-1}+p_2^{-1} $ and any $ A\in\mathcal{F} $:
\begin{align}
\label{eq:Hest1}
\|\I_{A}H_1H_2\|_{p, i, n}\leq& C\|\I_{A}H_1\|_{p_1, i, n}\|\I_{A}H_2\|_{p_2,i,n}.
\end{align}

We will also frequently manipulate quantities such as $ \E_{i,n}[\delta_L(\bar{X}_{i+1})H] $ for $ H\in\mathbb{S}_{i+1,n}(\bar{X}) $. The previous expression has a clear meaning due to the IBP formula \eqref{eq:IBP} (see the theory in \cite{IW}, Chapter V.9. for a much more general framework) or in the sense of conditional laws
\begin{align*}
\E_{i,n}[\delta_L(\bar{X}_{i+1})H] =&\E_{i,n}[\I_{D_{i+1,n}}{\mathcal{I}}_{i+1}(H)] = \E_{i,n}[H\vert \bar{X}_{i+1}=L]g(a_i(\zeta_{i+1}-\zeta_i),L-\bar{X}_i).\nonumber
\end{align*}

We finally introduce the following space of r.v.'s with certain time (ir)regularity estimates.
\begin{definition}
	 For $\ell\in\mathbb{Z}$, $ i\in \bar{\mathbb{N}}_n$, we define the space $ {\mathbb{M}}_{i+1,n}(\bar{X},\ell/2) $ as the set of r.v.'s $ H\in\mathbb{L}^0 $, satisfying the property {\it (1)} in Definition \ref{def:1} and such that 
	\begin{align*}
	\I_{D_{i,n}}\|\I_{D_{i+1,n}}H\|_{p, i, n}\leq C(\zeta_{i+1}-\zeta_i)^{\ell/2}
	\end{align*}
	 {for some deterministic constant $ C $ independent of $ (p,i,n) $.}
\end{definition}

  We again remark that since the definition of the space $ {\mathbb{M}}_{i+1,n}(\bar{X},\ell/2)  $ uses the conditional norm $\E_{i, n}[.]$ and property {\it (1)} in Definition \ref{def:1}, when we say that a r.v. $ H\in  {\mathbb{M}}_{i+1,n}(\bar{X},\ell/2)$, this statement is understood on the set $\left\{N_T = n \right\}$, $n\in \N$.


A straightforward consequence of equation \eqref{eq:exta} and \eqref{eq:Hesta} is the following property.
%

\begin{lem}
	\label{lem:6}
	For $ j\in\{0,1\} $, $ i\in \bar{\mathbb{N}}, $ and $ k\in\mathbb{N} $, if $ H_1\in{\mathbb{M}}_{i+1,n}(\bar{X},j/2)\cap\mathbb{S}_{i+1,n}(\bar{X}) $ with $ \|\mathcal{D}_{i+1}^kH_1\|_{p, i, n}\leq C $ then $ {\mathcal{I}}_{i+1}^k(H_1)\in{\mathbb{M}}_{i+1,n} (\bar{X},(j-k)/2)$. Furthermore, if $ H_2\in  {\mathbb{M}}_{i+1,n}(\bar{X},k/2) $ then the product $ H_1H_2\in {\mathbb{M}}_{i+1,n}(\bar{X},(j+k)/2)  $.
\end{lem}

\begin{lem}\label{chain:rule}
Let $h\equiv h(\bar X_i, \bar X_{i+1}, \rho_{i+1}, \zeta_i,\zeta_{i+1}) \in \mathbb{S}_{i+1,n}(\bar X)$ with 
$ i\in \bar{\mathbb{N}}, $ then the following chain rule type formula holds
\begin{gather*}
\partial_{\bar X_i}\mathcal{I}_{i+1}(h) = \mathcal{I}_{i+1}(\partial_{\bar X_i}h) - \frac{\sigma'_i}{\sigma_i}\mathcal{I}_{i+1}(h) .
\end{gather*}
\end{lem}
\begin{proof}
From the extraction formula \eqref{eq:exta}, the usual chain rule and the fact that $\partial_{\bar X_i}\mathcal{I}_{i+1}(1) = -\frac{\sigma'_i}{\sigma_i}\mathcal{I}_{i+1}(1)$, 
\begin{align*}
\partial_{\bar X_i} \mathcal{I}_{i+1}(h)
& = -\frac{\sigma'_i}{\sigma_i}\mathcal{I}_{i+1}(1) h + \mathcal{I}_{i+1}(1) \partial_{\bar X_{i}} h - \partial_{\bar X_{i}} {\mathcal{D}}_{i+1}h\\
\mathcal{I}_{i+1}(\partial_{\bar X_i} h)
& = \mathcal{I}_{i+1}(1) \partial_{\bar X_i} h - {\mathcal{D}}_{i+1}\partial_{\bar X_i} h.
\end{align*}
Note that $\partial_{\bar X_i}$ and ${\mathcal{D}}_{i+1}$ do not commute. Indeed, by the usual chain rule, one has
\begin{align*}
\partial_{\bar X_{i}} {\mathcal{D}}_{i+1}h -  {\mathcal{D}}_{i+1}\partial_{\bar X_i} h & = \partial_1\partial_2 h  + \partial^2_2 h \frac{\partial  \bar X_{i+1}}{\partial \bar X_i} - {\mathcal{D}}_{i+1}(\partial_1h + \partial_2 h\frac{\partial  \bar X_{i+1}}{\partial \bar X_i} )= -\partial_2 h  \frac{\sigma'_i}{\sigma_i}
\end{align*}
where in the last equality we used the fact that ${\mathcal{D}}_{i+1}\frac{\partial  \bar X_{i+1}}{\partial \bar X_i}  = \frac{\sigma'_i}{\sigma_i}$. By combining the above computations we obtain
\begin{align*}
\partial_{\bar X_i} \mathcal{I}_{i+1}(h) - \mathcal{I}_{i+1}(\partial_{\bar X_i} h) & = - \frac{\sigma'_i}{\sigma_i}\left(\mathcal{I}_{i+1}(1) - {\mathcal{D}}_{i+1} h\right) = -\frac{\sigma'_i}{\sigma_i}\mathcal{I}_{i+1}(h).
\end{align*}
\end{proof}
\begin{remark}
Heuristically, the result of Lemma \ref{chain:rule} can be viewed as a chain rule formula of the type 
\begin{align*}
\partial_{\bar X_i} (\mathcal{I}_{i+1}(h)) \equiv(\mathcal{I}_{i+1})(\partial_{\bar X_i} h) + (\partial_{\bar X_i} \mathcal{I}_{i+1})(h).
\end{align*}
\end{remark}

\section{Markov chain representation for killed processes}
\label{prob:repres:sec}
In this section, we establish a probabilistic representation for the marginal law of the killed process, that is, for the law of $ X_T $ on the set $ \{\tau> T\} $ based on the Markov chain $\bar{X}$ introduced in the previous section. The proof of the following result is postponed to Appendix \ref{app:sec}. 
\begin{theorem}
	\label{th:3.2a}Let $ f:\mathbb{R} \rightarrow\mathbb{R}$ be a measurable function with at most polynomial growth at infinity such that $ f(L)=0 $. Then, one has
	\begin{align}
	\label{eq:PR}
	\E\left[f(X_{T})\I_\seq{\tau> T}\right ]  &= \mathbb{E}\Big[  f(  
	\bar{X}
	_{N_T+1}) \prod_{{i}=1}^{N_T+1}  \I_{D_{i,N_T}}\bar{\theta}_i  \Big].
	\end{align}
	Here, for $ i\in\mathbb{N}_{n+1}$, we let 
		\begin{align}
	\label{eq:13aa}
	\bar{\theta}_i :=&\I_\seq{N_T>i-1}
	2(2\rho_i-1)\lambda^{-1}\left({{\mathcal{I}}}_{{{i}}}(c^i_1)+{{\mathcal{I}}}^2_{ {{i}}}(c^i_2)\right)
	+\I_\seq{N_T=i-1}2e^{\lambda T}(2\rho_{N_T+1}-1)
	.
	\end{align}
	 In the above definitions, we have used $ c^i_j\in\mathbb{S}_{i,n} (\bar{X})$, $ j=1,2 $, $ i\in \mathbb{N}_{n+1} $, $ n\in\bar{\mathbb{N} }$ where 
	\begin{align}
	c^i_1\equiv c^{i}_1(\bar{X}_{{i-1}}, \bar{X}_{i}) & {:= }b_i = b(\bar{X}_{i}),\label{def:c1:c2a}\\
	c_2^i \equiv c^{i}_2(\bar{X}_{{i-1}}, \bar{X}_{i}) & := \frac12 (a_i-a_{i-1}) = \frac12 (a(\bar{X}_{{i}}) - a(\bar{X}_{{i-1}})). \nonumber
	\end{align}
	
	Furthermore, we have 
	$\bar{\theta}_i \in\mathbb{S}_{i,n}(\bar{X}) $, $ i\in \mathbb{N}_{n+1} $, $ n\in\bar{\mathbb{N} }$ and the following time degeneracy estimate is satisfied: for all $p\geq 1$, there exists some positive constant $C:=C(T, a,  b, p)$ such that for $ i\in\mathbb{N}_n $,
		\begin{equation}
		\label{moment:estimates:prob:representation} 
		  \I_\seq{i \leq  n}(\zeta_{i}-\zeta_{i-1})^{\frac{p}{2}}\I_{D_{i-1,n}}\E_{i-1,n}\left[\I_{D_{i,n}}|{\bar\theta}_i |^p\right]+\I_\seq{i=n+1}\I_{D_{n,n}}\E_{n,n}\left[\I_{D_{n+1,n}}|{\bar\theta}_{n+1}|^p\right ]\leq C. 
		\end{equation}
		As a consequence, for all $ p\in [0,2) $, one has $ \E\left[\big|\prod_{{i}=1}^{N_T+1} \I_{D_{i,N_T}}\bar {\theta}_i \big|^p\right ]<\infty $ .
\end{theorem}

We importantly observe that the above probabilistic representation allows to implement an unbiased Monte-Carlo simulation method since one just has to simulate the Poisson process $N$, then the Markov chain $ \bar{X}$ along the jump times of $N$ and finally to compute the explicit product of weights. Contrary to our previous work \cite{FKL1}, the Markov chain $ \bar{X}$ is defined from the reflection principle introduced in Lemma \ref{lem:1} thus reducing further the number of variables in the problem. This probabilistic representation is thus more workable for the forthcoming IBP formulae.

\begin{remark}\label{rem:3}
We make several remarks before moving on:\hfill\break
(i) The assumption $ f(L)=0 $ can be avoided at the cost of longer formulae as one can obtain a probabilistic representation based on the same Markov chain for the probability $ \P(\tau\geq T) $ using the results in Section 3 of \cite{FKL1}, but we do not pursue this goal here.\hfill\break
(ii) In the proof of the above result, one uses the following crucial property: For $ i\in\bar{\mathbb{N}}_n $, the random variables $\I_{D_{i+1,n}} c^{i+1}_1$, $\I_{D_{i+1,n}}\partial_{\bar{X}_i}c^{i+1}_2\in{\mathbb{M}}_{i+1,n}(\bar{X},0) $ and $ \I_{D_{i+1,n}}c^{i+1}_2\in{\mathbb{M}}_{i+1,n}(\bar{X},1/2) $ so that $\I_{D_{i+1,n}}{\mathcal{I}}_{i+1}(c^{i+1}_1)$, $ \I_{D_{i+1,n}}{\mathcal{I}}_{i+1}^2(c^{i+1}_2)\in {\mathbb{M}}_{i+1,n}(\bar{X},-1/2)  $ and therefore $ \I_{D_{i+1,n}}\bar{\theta}_{i+1}\in  {\mathbb{M}}_{i+1,n}(\bar{X},-1/2)$. 
\hfill\break	
(iii) The power $ p/2 $ appearing in the time degeneracy estimates in \eqref{moment:estimates:prob:representation} is crucial in order to determine the integrability of the r.v. appearing on the right hand side of \eqref{eq:PR}. This motivates the definition below.

\end{remark}

\begin{definition}
	\label{def:td}
	We say that a weight r.v. $ H\in \mathbb{S}_{i,n} $ satisfies the time degeneracy estimate if for all $ p\geq 1 $
	\begin{align}
	\label{eq:td}
	 \I_{D_{i-1,n}}\left\|\I_{D_{i,n}}H\right\|_{p,i-1,n}\leq C(\zeta_i-\zeta_{i-1})^{-\frac 12} 
	\end{align}
	  in the case that $ i\in\mathbb{N}_n $ and 
	$ \I_{D_{n,n}}\left\|\I_{D_{n+1,n}}H\right \|_{p,n,n}\leq C $ in the case that $ i=n+1 $.	
\end{definition}

At this stage, we find it useful to show graphically the dynamic structure of the Markov chain and the random {\it weights} $ \bar\theta_i $, $ i\in\mathbb{N}_{N_T+1}  $ for the probabilistic representation \eqref{eq:PR}. This will be important in order to understand the structure of the IBP formula.


\begin{figure}[t]
\vspace{1cm}
\hspace{-16cm}
\begin{tikzpicture}(500,400)
\put(32,3){$x$}
\put(32,12){$ \bar{X}_0 $}
\put(40,5){\vector(1,0){30}}
\put(50,10){$\bar\theta_{1}$}
\put(71,3){$\bullet$}
\put(70,12){$ \bar{X}_1$}
\put(78,5){\vector(1,0){30}}
\put(85,10){$\bar\theta_{2}$}
\put(110,3){$\bullet$}
\put(109,12){$ \bar{X}_2$}
\put(120,3){$\cdot$}
\put(125,3){$\cdot$}
\put(130,3){$\cdot$}
\put(135,3){$\cdot$}
\put(142,3){$\bullet$}
\put(140,12){$ \bar{X}_{n-1} $}
\put(150,5){$\vector(1,0){43}$}
\put(170,10){$\bar\theta_{n}$}
\put(197,3){$\bullet$}
\put(197,12){$ \bar{X}_{n} $}
\put(230,3){$=:\bigstar_{n}$}
\put(256,5){$\vector(1,0){30}$}
\put(264,10){$\bar\theta_{n+1}$}
\put(296,3){$f(\bar{X}_{n+1})\,\,\,\,\mathrm{ Space }$}
\put(365,3){$ \,\,\mathrm{ evolution}$}
\put(32,-17){$\zeta_0$}
\put(41,-15){\vector(1,0){28}}
\put(71,-17){$\zeta_1$}
\put(81,-15){\vector(1,0){28}}
\put(110,-17){$\zeta_2$}
\put(120,-17){$\cdot$}
\put(125,-17){$\cdot$}
\put(130,-17){$\cdot$}
\put(135,-17){$\cdot$}
\put(140,-18){$\zeta_{n-1}$}
\put(150,-14){$\vector(1,0){40}$}
\put(192,-17){$\zeta_{n}$}
\put(256,-15){$\vector(1,0){28}$}
\put(286,-15){ \hspace*{.3cm}$ T $}
\put(306,-17){$\,\,\mathrm{  Continuous}$}
	\put(358,-17){$\,\, \mathrm{time} $}
	\put(383,-17){$\,\, \mathrm{evolution} $}
\end{tikzpicture}
\vspace*{10mm}
\caption{The time evolution of the Markov chain and its weights} \label{fig:M1}
\end{figure}
\vspace{0.1cm}

In the above figure, one observes the evolution of the Markov chain $ \bar{X}_i $, $ i\in\bar{\mathbb{N}}_{n+1} $, together with its associated weight sequence appearing in the probabilistic representation. We also note that at the end of the Markov chain evolution which always happens at time $ T $, the test function $ f $ is evaluated at $ \bar{X}_{n+1}$. Furthermore, the right hand side of \eqref{eq:PR} is the product of all elements on top of the arrows on the first line of the above figure and the corresponding indicator functions of the sets $ D_{i,n} $ defined by \eqref{set:Din:X}, with $ f(\bar{X}_{n+1}) $. The second line gives the time evolution followed by the jump times of the Poisson process $ N $ which coincide with the times at which transitions of the Markov chain $ \bar{X}$ happen.


 In other figures that appear later on, we will use the general symbol $ \bigstar_k$, $ k\in\mathbb{N}_{n+1} $ defined in Figure \ref{fig:M1} which stands for product $ \bigstar_k:=\prod_{i=1}^k\I_{D_{i,n}}\bar{\theta}_i  $. We remark here that as stated in \eqref{eq:13aa}, on the set $\left\{ N_T=n\right\}$, the definition of $(\bar{X}_{n+1}, \bar\theta_{n+1} )$ differs from all other $(\bar{X}_i,\bar{\theta}_i)  $, $ i\in\mathbb{N}_n $.

\section{Ingredients for an IBP formula}
\label{sec:ingr}
In this section we give the main ingredients in order to establish an integration by parts formula for the marginal law taken at time $T$ of the killed process. 

In order to understand the main ingredients to be introduced in this section, one starts by supposing that the problem is to find an IBP formula for 
$ \E\left[f'(X_{T})\I_\seq{\tau\geq T}\right ] $. The BEL formula will be easier to handle and is tackled in Section \ref{BEL:sec}.

\subsection{Some Heuristics}
The following heuristic arguments may help the reader to understand the strategy developed in the next sections to prove our IBP formulae. 

{\it Step 1:}
 The first step was performed with the probabilistic representation established in Theorem \ref{th:3.2a} involving the Markov chain $\bar{X}$ evolving on a time grid governed by the set of jump times of the Poisson process $N$.

{\it Step 2:} As already explained in the introduction, the central idea is to reduce the infinite dimensional problem of finding an IBP formula for $\E[f'(X_T) \I_\seq{\tau >T}]$ to a finite dimensional problem, namely finding an IBP formula for the quantity $\mathbb{E}\Big[  f'(\bar{X}_{N_T+1}) \prod_{{i}=1}^{N_T+1}  \I_{D_{i,N_T}}\bar{\theta}_i  \Big]$, for which the dimension is random and given by the number of jumps of the Poisson process at time $T$. 

At this stage, unfortunately, one cannot perform a standard integration by parts formula as in \cite{Nuabook} on the whole time interval $[0,T]$ for various reasons. For example, the Skorokhod integral of the product of weights $\prod_{{i}=1}^{N_T+1}  \I_{D_{i,N_T}}\bar{\theta}_i$ will inevitably involve the Malliavin derivatives of $\bar{\theta}_i$ and the indicator function $\I_{D_{i,N_T}}$, which in turn will raise integrability problems of the resulting Malliavin weight. 

The key idea that we use consists in performing IBP formulae locally on each random intervals $[\zeta_i, \zeta_{i+1}]$, for $ i\in\bar{\mathbb{N}}_n $ on the set $ \{N_T=n\} $, that is, by using the noise of the Markov chain on this specific time interval and then by combining them in an suitable way. The case $ i=n $ is easy to consider because, by taking the conditional expectation $\E_{n,n}[.]$ in the original probabilistic representation on the set $\left\{N_T=n\right\}$, the IBP reduces to $\E_{n,n}[f'(\bar{X}_{N_T+1}) \I_{D_{n+1,n}} \bar{\theta}_n] = \E_{n,n}[f(\bar{X}_{N_T+1}) \I_{D_{n+1,n}} \mathcal{I}_{n+1}(\bar{\theta}_{n+1})]$ by \eqref{eq:IBP} and the fact that $f(L) = 0$. However, performing the IBP formula on a random interval $[\zeta_{i}, \zeta_{i+1}]$ for $0\leq i<n$ is more challenging. 

In order to do it, our first ingredient consists in transferring the derivative operator appearing on the test function $f$ backward in time from the last interval to the interval on which we perform the local IBP, say $[\zeta_{i},\zeta_{i+1}]$. This operation will unfortunately generate a boundary term each time a transfer is performed, i.e. one for each time interval $[\zeta_{j}, \zeta_{j+1}]$, $j=i,\cdots, n$. As previously mentioned, the boundary terms will induce integrability problems. In order to circumvent this issue, we then introduce our second ingredient which consists in performing a \emph{boundary merging procedure} of this term using the time randomness provided by the Poisson process. We refer the reader to Section \ref{sec:conv} for a more detailed discussion of this issue. Then, an additional ingredient that we use is a time merging lemma about the reduction of jump times of the Poisson process. It is described in Lemma \ref{lem:Unif} and directly employed in the proof of our main result, namely Theorem \ref{prop:5.2}. 
   
  As already explained above, the last ingredient consists in combining these various local IBP formulae in a suitable way. Roughly speaking, we consider a weighted sum of each integral operator, the weight being the length of the corresponding time interval. 

\subsection{The transfer of derivative lemma}
\label{sec:td}
\begin{lem}
	\label{lem:5.1}
	Let $f\in \mathscr{C}^1_p(\mathbb{R} )$ and $n\in \bar{\N}$. The following transfer of derivative formula holds for $ i\in\bar{\mathbb{N}}_{n-1} $ :
\begin{align}
\label{eq:5.1a}
 \E_{i,n}\big[\partial_{\bar{X}_{i+1}}f(\bar X_{i+1})\I_{D_{i+1,n}}
		\bar{\theta}_{i+1}\big]   =& \partial_{\bar{X}_i}\mathbb{E}_{i,n}\big[f(\bar X_{i+1}) 
		\I_{D_{i+1,n}} \ltheta^{e}_{i+1}\big] \\
		& \quad + \mathbb{E}_{i,n}\big[f(\bar X_{i+1})
		\big(\I_{D_{i+1,n}}
		\ltheta^c_{i+1}
		 +\delta_L(\bar{X}_{i+1})\ltheta^{\partial}_{i+1}\big)\big].\nonumber 
\end{align}
	
	\noindent where the r.v.'s $(\ltheta^e_{i+1},\ltheta^c_{i+1},\ltheta^\partial_{i+1}) \in {\mathbb{S}}_{i+1,n}(\bar{X}) $ are defined by
	\begin{align}
		\nonumber
		\ltheta^e_{i+1}:=&
		2\lambda^{-1}\left ({\mathcal{I}}_{i+1}^2(d_2^{i+1})
		+{\mathcal{I}}_{i+1}(d^{i+1}_1)
		\right ),
		\\
		\label{eq:defDc}
		{{\ltheta}}_{i+1}^c:=&
		{{\mathcal{I}}}_{i+1}\left(
		\bar\theta_{i+1}-(2\rho_{i+1}-1)
		\ltheta^e_{i+1}
		\right) -\partial_{\bar{X}_i}
		\ltheta^e_{i+1}
		-
		\sigma'_i{{\mathcal{I}}}_{i+1}\left(Z_{i+1}
		\ltheta^e_{i+1}
		\right) ,\\
			\nonumber
		{{\ltheta}}_{i+1}^\partial:=&
		2(2\rho_{i+1}-1)\lambda^{-1}(a'(L)-b(L)){\mathcal{I}}_{i+1}(1),\\
		d_1^{i+1}:=&c_1^{i+1}-(2\rho_{i+1}-1)\partial_{\bar{X}_{i}}c_2^{i+1},\label{d1}\\
		d_2^{i+1}:=&c_2^{i+1}.\label{d2}
	\end{align}

	Assume additionally that $f(L)=0$. Then, one has 
	\begin{align}
	\label{eq:sdf}
		 \E_{n,n}\big[\partial_{\bar{X}_{n+1}}f(\bar X_{n+1})\I_{D_{n+1,n}} \bar {\theta}_{n+1}\big]   
		& =  \partial_{\bar{X}_{n}} \E_{n,n}\big[f(\bar X_{n+1}) 
		\I_{D_{n+1,n}} {\ltheta}^e_{n+1}\big] \\
		& \quad + \E_{n,n}\big[f(\bar X_{n+1})
		\I_{D_{n+1,n}} {\ltheta}^c_{n+1}\big] \nonumber  
	\end{align}
	with $ \ltheta^\partial_{n+1}:=0 $,
	 ${\ltheta}^e_{n+1}:=2e^{\lambda T}$ and ${\ltheta}^c_{n+1}:=-2e^{\lambda T} (\sigma'\sigma)_{n} (T-\zeta_{n}){\mathcal{I}}^2_{n+1}(1)\in \mathbb{S}_{n+1,n}$.	 
	 With the above definitions, 
		 $x\mapsto \E_{i,n}[f(\bar X_{i+1}) \I_{D_{i+1,n}} \ltheta^e_{i+1} \, | \bar{X}_i = x] \in \mathscr{C}^1_p(\mathbb{R} )$, a.s. for $i\in \bar{\mathbb{N}}_n$. Moreover, one has that 
		 $ \ltheta^a_{i+1} $, for $ a\in\{e,c,\partial\} $ satisfies the time degeneracy estimates for $ i\in \bar{\mathbb{N}}_n $.
		
\end{lem}
The proof of Lemma \ref{lem:5.1} is postponed to Appendix \ref{app:tl}. The transfer of derivatives procedure starts on the last time interval $ [\zeta_{n},T] $ according to formula \eqref{eq:sdf}. It expresses the fact that the derivative operator $ \partial_{\bar{X}_{n+1}} $ on the left hand side of the equation is transferred to a flow derivative $ \partial_{\bar{X}_n} $ of the conditional expectation on the right hand side of \eqref{eq:sdf}. Remark that the derivative of $ f $ has been written ubiquitously as $ \partial_{\bar X_{n+1}}f(\bar{X}_{n+1}) $ and that exceptionally in the last time interval there is no boundary term due to the assumption $ f(L)=0 $.

 Then, by the Markov property, the first conditional expectation appearing on the right hand side of \eqref{eq:5.1a} can be expressed as a function of $ \bar{X}_n $ which will be the new test function that will be used in \eqref{eq:5.1a} for the case $ i=n-1>0 $ and so on.

  The transfer of derivatives for other time intervals is obtained in \eqref{eq:5.1a}. This formula in comparison with \eqref{eq:sdf} has a boundary term which is denoted by $ \ltheta^\partial_{i+1}$. In this fashion, various transfer of derivatives formulae can be obtained by transferring successively the derivative operator through all intervals backward in time. The left pointing arrow appearing on the top of the notation $\ltheta^a_{i+1}$, for $a\in\{e,c,\partial\}$ expresses the fact that we are performing a transfer of derivatives argument backward in time. 

\hspace{2cm}
\begin{figure}[H]
\begin{tikzpicture}
\node (1) at (0,0) {};
\node (2)  at (3,0)  {$ \partial_{\bar{X}_{i}} $};
\node (3) at (6,0) {$ \partial_{\bar{X}_{i+1} }$};
\node (4) at (9,0)  {};
\draw[->] (1)--(2);
\draw[->] (2)--(3)
node[pos=.5,above] {$\ltheta^e_{i+1}$}; 
\draw[->] (3)--(4);
\node (5) at (3,-1) {$ \bigstar_{i} $};
\node (6) at (3,-2) {$ \bigstar_{i} $};
\node (7) at (1.5,0.3) {$ \theta_i $};
\node (8) at (7.5,0.3) {$ \ltheta^e_{i+2} $};
\draw[->] (5) -- (3)
node[pos=.3,sloped,above] {$\ltheta^c_{i+1}$};
\draw[->,blue] (6) -- (3)
node[pos=.5,sloped,below] {$\ltheta^\partial_{i+1}$};
\draw[->,very thick, red] (6.0,0.6) to [out=125,in=45] (3,0.65);
\end{tikzpicture}
\caption{The dynamics of the transfer of derivatives formula}
\label{fig:2}
\end{figure}
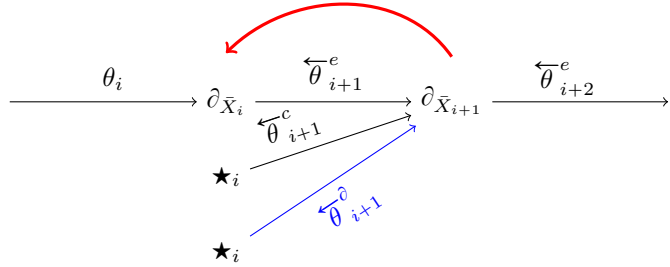

%
%
%
As explained with the transfer of derivatives for the last time interval, we also see that the derivative $ \partial_{\bar{X}_{i+1}} $ is transferred (this is the meaning of the left pointing red arrow in the above figure) to become the derivative $ \partial_{\bar{X}_i} $ changing the weight $ \bar\theta_{i+1} $ into $\ltheta^e_{i+1}$ ($e$ for exchange) but at the expense of creating extra terms denoted by $ \ltheta_{i+1}^c $ ($ c $ for correction) and a boundary term $ \ltheta^\partial_{i+1} $ coming from the boundary (denoted by $\partial$ in what follows). All these weights satisfy the same moment estimates as \eqref{moment:estimates:prob:representation}.

In the above figure, we can see the structure of the transfer of derivatives will generate two branches of a tree at each time Lemma \ref{lem:5.1} is applied. The blue arrow means that the corresponding term generated by the transfer of derivative is a degenerate boundary term that needs to be addressed because the weight r.v. $ \ltheta^\partial_{i+1} $ is multiplied by a Dirac delta distribution. This will be solved using the second ingredient to be introduced in the next section: {\it the boundary merging lemmas}. The term $ \bigstar_i $ standing for the product $ \bigstar_i:=\prod_{\ell=1}^i\I_{D_{\ell,n}}\bar{\theta}_\ell $ means that no more transfer of derivatives is required for these two terms because the derivative operator no longer appears on the test function $f$ for the terms associated to $ \ltheta^c_{i+1} $ and $ \ltheta^\partial_{i+1} $ in \eqref{eq:5.1a} .


We remark that the last term in \eqref{eq:5.1a} is singular but well defined if the time variables are fixed which is the case here as there is a conditioning with respect to the Poisson process expressed with the conditional expectations $ \E_{i,n} $ in \eqref{eq:5.1a}. Since the boundary weight $ \ltheta^\partial_{i+1}$ is multiplied by $\delta_L(\bar{X}_{i+1})$ the time degeneracy estimate for the product $\delta_L(\bar{X}_{i+1})  \ltheta^\partial_{i+1}$ deteriorates so that the resulting probabilistic representation will not be exploitable without the extra ingredient of the boundary merging lemmas. A further study is thus required before proving moment estimates.
We will discuss this matter in detail in the next section.


\subsection{The boundary merging lemmas}
\label{sec:conv}
The second ingredient that we need in order to establish our first IBP formula is the boundary merging lemma. Though the sequence of weights $\big\{ \ltheta^\partial_{i},  i \in\mathbb{N}_{N_T}  \big\}$ satisfies the time degeneracy estimate stated in Lemma \ref{lem:5.1}, which is similar to the one satisfied by the sequence $\left\{\bar\theta_i, 
i \in\mathbb{N}_{N_T} \right\}$ in \eqref{moment:estimates:prob:representation}, the fact that it is multiplied by $ \delta_L(\bar{X}_{i})$ increases its time degeneracy by a factor of $(\zeta_{i}-\zeta_{i-1})^{-1/2}$.

 The key idea in order to circumvent this issue consists in using the time randomness provided by the Poisson process in order to smooth this singularity. This will eventually allow us to retrieve a time degeneracy estimate similar to \eqref{moment:estimates:prob:representation}. We call this ingredient {\it the boundary merging lemmas}.

To be more specific, in order to circumvent the time degeneracy problem related to the boundary term on the right hand side of \eqref{eq:5.1a}, one needs to consider two successive random time intervals, say $[\zeta_{i-1}, \zeta_i]$ and $[\zeta_i, \zeta_{i+1}]$, and take expectations with respect to the intermediate time $\zeta_i$ in order to regularize the term $ \delta_L(\bar{X}_{i}) $ that appears when one transfers the derivative using \eqref{eq:5.1a} on the time interval $[\zeta_{i-1}, \zeta_{i}]$. This operation will remove the time singularity induced by the Dirac delta distribution but will also reduce the number of jumps by one unit.

An important remark is that the boundary merging lemmas are not needed in the case that there are no jumps in the interval $ [0,T] $, that is, on the set $\left\{N_T =0\right\}$, because of the condition $ f(L)=0 $.

From the application of the boundary merging lemmas, a new Markov chain structure appears which we call the {\it merged boundary (one-step) Markov chain }$\bar{X}^\partial$ whose transition on the time interval
 $[\zeta_{j},\zeta_{i+1}]$, for $j=i-1, \, i$, is given by
\begin{align}
\label{eq:dbp}
\bar{X}^\partial_{j, i+1}\equiv \bar{X}^{\partial}_{j,i+1}(\bar{X}_{j}) := &L\left(1-\mu(\bar{X}_{j})\right)+\bar{X}_{j}
\mu(\bar{X}_{j})
+\sigma(L) Z_{j, i+1}, \quad Z_{j, i+1} := \sum_{k=j+1}^{i+1}Z_{k},
\end{align}


\noindent with $\mu(x):=- \sigma(L)/\sigma(x)$. 

In the above one-step dynamics and in what follows, the index $j= i-1$ will be used to indicate the boundary merging of the underlying Markov chain $\bar{X}$ on the two consecutive time intervals $[\zeta_{i-1}, \zeta_i]$ and $[\zeta_i, \zeta_{i+1}]$ while the index $j=i$ will be used in a second stage once the reduction of jumps (a.k.a. {\it time merging}) is performed in Section \ref{sec4.3}. As the proof of time degeneration estimates are similar and in order to simplify the presentation, we will deal with both cases at the same time in the next two Lemmas.

In the same way as it was done in Section \ref{sec:3a}, we define the associated space of smooth r.v.'s.
\begin{definition}
		For $ i\in\bar{\mathbb{N}}_n$ and $ j\in\{i-1,i\} $, we define the set $ {\mathbb{S}}_{j,i+1,n}(\bar{X}^\partial) $ as the subset of r.v.'s $ H\in \mathbb{L}^0 $ such that there exists a measurable functions $ h:\mathbb{R}^2\times \{0,1\}\times A_{i+1-j}\rightarrow \mathbb{R} $ satisfying
	\begin{enumerate}
		\item the random variable $H$ can be written as $\displaystyle{ H=
			h(\bar{X}_{j},\bar{X}^\partial_{j,i+1},\rho_{i+1},\zeta_{j},\zeta_i,\zeta_{i+1}) 
		}
		$ on the set $ \{N_T=n\} $. 
		\item For any $ r\in\{0,1\} $ and any $ s\in A_{i+1-j}$, one has $ h(\cdot,\cdot,r,s)\in\mathscr{C}_p^\infty(\mathbb{R}^2) $.
	\end{enumerate}
\end{definition}

Similarly to the operators ${\mathcal{D}}$ and ${\mathcal{I}}$, defined in Section \ref{sec:3a}, related to the Markov chain $\bar{X}$, we introduce the operators $ \bar{{\mathcal{D}}}$ and $\bar{{\mathcal{I}}} $ associated to the merged boundary process $ \bar{X}^\partial $ defined above. For any (smooth) r.v. $ H\in  {\mathbb{S}}_{j,i+1,n}(\bar{X}^\partial)$ and $\ell\geq 1$, we let
\begin{align*}
\bar{{\mathcal{I}}}_{j,i+1}(H):=&
H \frac{Z_{j,i+1}}{\sigma(L)(\zeta_{i+1}-\zeta_{j}) } - {\bar{\mathcal{D}}_{i+1} H}, \quad \bar{{\mathcal{I}}}^{\ell+1}_{j,i+1}(H)  := \bar{{\mathcal{I}}}^{\ell}_{j,i+1}(\bar{{\mathcal{I}}}_{j,i+1}(H))\\
\bar{\mathcal{D}}_{i+1} H:=&\partial_2h(\bar{X}_{j},\bar{X}^\partial_{j,i+1}, \rho_{i+1}, \zeta_j,\zeta_i,\zeta_{i+1}),  \quad \bar{\mathcal{D}}^{\ell+1}_{i+1} H :=  \bar{\mathcal{D}}^{\ell}_{i+1} (\bar{\mathcal{D}}_{i+1} H).
\end{align*}

Observe that since $Z_{j,i+1} = \sigma^{-1}(L)( \bar{X}^{\partial}_{j,i+1}- (L\left(1-\mu(\bar{X}_{j})\right)+\bar{X}_{j}\mu(\bar{X}_{j})) )\in {\mathbb{S}}_{j,i+1,n}(\bar{X}^{\partial})$, it is clear that the r.v.'s $\bar{{\mathcal{I}}}_{j,i+1}(1)$ and $\bar{{\mathcal{I}}}^2_{j,i+1}(1)$ also belong to ${\mathbb{S}}_{j,i+1,n}(\bar{X}^{\partial})$ so that they are explicit functions of the variables $\bar{X}_{j}, \, \bar{X}^{\partial}_{j,i+1}$, $\rho_{i+1}$, $\zeta_{j}, \, \zeta_{i}$ and $ \zeta_{i+1} $. 

In the following result, we will also use new weights $ \ltheta^{\partial *e}_{j,i+1}$, for $j= i -1 , i$, obtained by a merging procedure of the two weights $\ltheta^{\partial}_{i}$ and $\ltheta^{e}_{i+1}$ on the boundary set $\left\{ \bar{X}_i = L\right\}$
\begin{align*}
\ltheta^{\partial *e}_{j,i+1}:=&
4\lambda^{-1}\frac{a'(L)-b(L)}{a_{j}} \left(\bar{{\mathcal{I}}}_{j,i+1}^2(\bar{d}^{i+1}_2) + \bar{{\mathcal{I}}}_{j,i+1}(\bar{d}^{i+1}_1) \right)
\end{align*}
\noindent with 
\begin{align}
\bar d^{i+1}_1 & :=(b-a')(\bar{X}^{\partial}_{j,i+1}), \quad \bar d^{i+1}_2   := \frac12 (a(\bar{X}^{\partial}_{j,i+1}) - a(L)) \label{coefficients:2:weights:d:bar}
\end{align}
\noindent for $i=0, \cdots, n-1$. For the last interval, on $ \{N_T=n\} $, we have ${\ltheta}^{\partial*e}_{j,n+1} := 4e^{\lambda T}\frac{a'(L)-b(L)}{a_{j}}$, for $j=n-1, n$.

The cases $ j=i-1 $ for \eqref{coefficients:2:weights:d:bar} and $ j=n-1 $ for $ {\ltheta}^{\partial*e}_{n-1,n+1} $ are used in the following result where we perform the boundary merging procedure. The remaining cases $ j=i $ for $ \ltheta^{\partial *e}_{i,i+1} $ and $ j=n $ for  ${\ltheta}^{\partial*e}_{n,n+1} $ will be used once time merging is performed in Section \ref{sec:!IBP}. For this reason, some properties that we will employ later appear in the next lemma for these cases.

The definition of the time degeneracy estimate in the sense of Definition \ref{def:td} is naturally extended to this case if we define
\begin{align*}
D^\partial_{j,i+1,n}:=\{\bar{X}^\partial_{j,i+1}\geq L, N_T=n\}, \quad j=i-1, i
\end{align*} 

\noindent and replace \eqref{eq:td} for $ H\in \mathbb{S}_{j,i+1,n}(\bar{X}^\partial) $ by 
\begin{align}
\label{time:degeneracy:estimate:boundary:process}
\forall p\geq1, \quad \I_{D_{j,n}}\E\left[\I_{D^\partial_{j,i+1,n}}|H|^p\Big| \mathcal{G}_{j},\zeta_{i+1},N_T=n\right]\leq C(\zeta_{i+1}-\zeta_j)^{-\frac p2}, \, i\in\mathbb{N}_{n-1}, \, j=i-1,i 
\end{align}

\noindent and $ \I_{D_{j,n}}\E\left[\I_{D^\partial_{j,n+1,n}}|H|^p \Big|  \mathcal{G}_{j}, N_T=n\right] \leq C$ for $j=n-1,\, n$. 

We are now in position to provide the second ingredient in order to establish our first IBP formula, namely the boundary merging lemma. Its proof is postponed to Section \ref{proof:lemma:merging}.


%


\begin{lem}
	\label{lem:4.3}
	Let $ f\in \mathscr{C}^0_p(\mathbb{R})$ and $n\in \bar{\N}$. The following property is satisfied for any $ i\in\mathbb{N}_{n-1} $ 
	\begin{align}
	\label{eq:faim}
	 & \E\big[f(\bar{X}_{{i+1}})\I_{D_{i+1,n}} \ltheta^e_{{i+1}}\delta_L(
	\bar{X}_{i})\ltheta^{\partial}_{{i}} 
	\,\vert\, \mathcal{G}_{{i-1}}, \zeta_{i+1},N_T = n\big]\\
	\quad & = 
	\frac{\lambda^{-1}}{\zeta_{i+1}-\zeta_{i-1}}  \E\big[f(\bar{X}^{\partial}_{i-1,{i+1}})\I_
	{D_{i-1,i+1,n}^\partial}\ltheta^{\partial *e}_{{i-1},{i+1}}
	\,\vert\, \mathcal{G}_{{i-1}}, \zeta_{i+1},N_T= n\big].\nonumber
	\end{align}

Similarly, for the last time interval, one has 
	\begin{align}
	\label{eq:faim:last:interval1}
	&  \E\big[f(\bar{X}_{n+1})\I_{D_{n+1,n}} {\ltheta}^e_{n+1}\delta_L(
	\bar{X}_{n})
	\ltheta^{\partial}_{{n}} 
	\,\vert\,\mathcal{G}_{{n-1}},N_T=n\big] \\\nonumber
	 & \quad = 
	\frac{\lambda^{-1}}{T-\zeta_{n-1}}
	\E\big[f(\bar{X}^{\partial}_{n-1,n+1})\I_{D_{n-1,n+1,n}^\partial} {\ltheta}^{\partial *e}_{n-1,n+1 }\,\vert\,\mathcal{G}_{{n-1}},N_T=n\big].
	\end{align}
	
	\noindent Here, $ |{\ltheta}^{\partial*e}_{j,n+1}|\leq C $ a.s. for $ j=n-1,n $. Moreover, with the above definitions $ \ltheta^{\partial *e}_{j,i+1} \in {\mathbb{S}}_{j,i+1,n}(\bar{X}^{\partial}) $, $ j=i-1,i $ and 
	$ \ltheta^{\partial *e}_{j,i+1}  $ satisfies the time degeneracy estimates in the sense of \eqref{time:degeneracy:estimate:boundary:process}.
	
	\end{lem}

The above lemma can be illustrated using Figure \ref{fig:3} below. The boundary weight r.v. $ \ltheta^\partial_i $ marked in blue is merged together with the weight $ \ltheta^e_{i+1} $ appearing in the next time interval, by taking expectations with respect to $ \zeta_i $ in \eqref{eq:faim}. This operation leads to a new transition on the time interval $[\zeta_{i-1}, \zeta_{i+1}]$ for the underlying Markov chain. More precisely, the new r.v. $ \bar{X}_{i-1,i+1}^\partial $, and the new weight, $ \ltheta^{\partial*e}_{i-1,i+1} $, both marked in red in the above figure, will replace the symbols in blue: $ \ltheta^\partial_i $ and $ \ltheta_{i+1}^e $.  In this sense, we call $ \ltheta_{i-1,i+1}^{\partial*e} $ the boundary merging of weights $ \ltheta^\partial_i $ and $ \ltheta_{i+1}^e $. 

 Remark that for the branch corresponding to the pair $ (\ltheta_i^e,\ltheta_{i+1}^e) $, no merging procedure is required so that the respective Markov chain elements $ \bar{X}_i $ and $ \bar{X}_{i+1} $ remain unchanged. For this reason, we have chosen to leave the weights $\ltheta^{e}_{i}$ and $\ltheta_{i+1}^e $ in black on the left hand side.

On the right hand side of Figure \ref{fig:3}, we see the result of the boundary merging procedure. Clearly the transition for the merged term has changed from $ \bar{X}_{i-1} $ to $ \bar{X}^\partial_{i-1,i+1} $ and then, for the next time interval $[\zeta_{i+1}, \zeta_{i+2}]$, the transition of the underlying Markov chain remains unchanged and is given by
\begin{align*}
Y_{i+2}:=& \bar{X}_{i+2}(i+1,{{\bar{X}^\partial}_{i-1,i+1}}) \\
=&
\rho_{i+2}{{\bar{X}^\partial}_{i-1,i+1}}+(1-\rho_{i+2}) (2L-{{\bar{X}^\partial}_{i-1,i+1}})+\sigma({{\bar{X}^\partial}_{i-1,i+1}}) Z_{i+2},
\end{align*}  

\noindent and where $ \overleftarrow{\beta}_{i+2}^e:= \ltheta_{i+2}^{e}({{\bar{X}^\partial}_{i-1,i+1}}, Y_{i+2}, \rho_{i+2}, \zeta_{i+1}, \zeta_{i+2})$ stands for the same weight formula as $ \ltheta_{i+2}^e $ but whose starting point is ${{\bar{X}^\partial}_{i-1,i+1}}  $ and its end point is $ Y_{i+2} $. We remind the reader that we freely use function notation for smooth r.v.'s on the spaces $ \mathbb{S}_{i+2,n}(\bar{X}) $ as in \eqref{abuse:notation:space}.

\hspace{3cm}
%

\begin{figure}[H]
	
	\begin{tikzpicture}[scale=0.53]
	\node (a) at (0,0) {$ {\bar{X}}_{i-1}$};
	\node (b) at (4,0) {$ \textcolor{blue}{{\bar{X}}_i} $};
	\node (c) at (8,0) {
		$ \textcolor{blue}{{\bar{X}}_{i+1}} $};
	\node (b1) at (4.4,-0.3){};
	\node (c1) at (7.3,-0.3){};
	\node (f) at (12,0) {$ {\bar{X}}_{i+2} $};
	\node (e) at (8,-1) {$ \textcolor{red}{{\bar{X}^\partial}_{i-1,i+1}} $	};
	\node (d) at (0,-4) {$ \textcolor{red}{{\bar{X}}_{i-1}} $};
	\draw[->] (a) -- (b);
	\draw[->] (c) -- (f)
	node[pos=.5,sloped,above] {$\ltheta_{i+2}^e$};
	\draw[->] (b) -- (c);
	\draw[->,blue] (b1) -- (c1);
	\draw [->, blue]  (d) -- (b) node[pos=.6,sloped,below] {$\ltheta^\partial_i$}; 
	\draw[->,thick, red] (d) -- (c)
	node[pos=.5,sloped,below] {$\ltheta^{\partial*e}_{i-1,i+1}$}; 
	\node at (2,0.5) {$ \ltheta_i^e $};
	\node at (6,0.5) {$ \ltheta^e_{i+1} $};
	\node (a1) at (16,0) {$ {\bar{X}}_{i-1}$};
	\node (b1) at (20,0) {$ {{\bar{X}}_i} $};
	\node (c1) at (24,0) {
		$ {{\bar{X}}_{i+1}} $};
	\node (f1) at (28,0) {$ {\bar{X}}_{i+2} $};
	\node (e3) at (24,-1) {$ {{\bar{X}^\partial}_{i-1,i+1}} $	};
	\node (e1) at (22.8,-0.63){};
	\node (e2) at (28,-1){$ 
Y_{i+2}	 $};
\draw[->,thick] (e3) -- (e2)
node[pos=.5,below] 
{$ \overleftarrow{\beta}_{i+2}^e $};
	\node (d1) at (16,-4) {$ {{\bar{X}}_{i-1}} $};
	\draw[->] (a1) -- (b1);
	\draw[->] (c1) -- (f1)
	node[pos=.5,sloped,above] {$\ltheta_{i+2}^e$};
	\draw[->] (b1) -- (c1);
	\draw[->,thick] (d1) -- (e1)
	node[pos=.5,sloped,below] {$\ltheta^{\partial*e}_{i-1,i+1}$}; 
	\node at (18,0.5) {$ \ltheta_i^e $};
	\node at (22,0.5) {$ \ltheta^e_{i+1} $};
	\end{tikzpicture}
	\caption{Markov chain structure before merging of boundary terms on the left and after merging on the right. }\label{fig:3}
\end{figure}
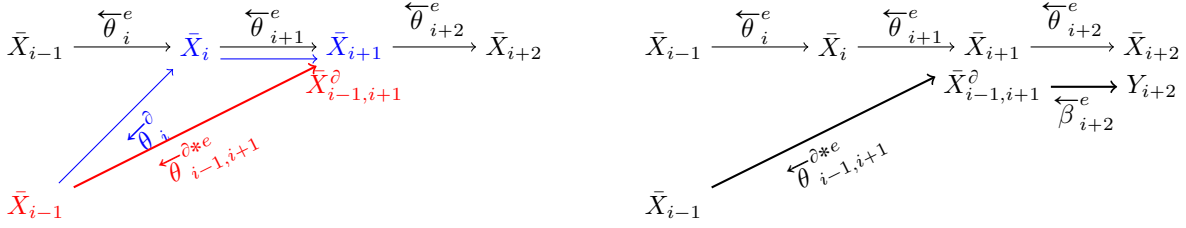

So far, we have explained how to transfer the derivatives and how to deal with boundary terms. The last step consists in performing a local IBP formula on a fixed time interval, say $[\zeta_{i}, \zeta_{i+1}]$. This operation will involve the integral operator applied to  corresponding weight, namely $\mathcal{I}_{i}(\I_{D_{i,n}} \bar{\theta}_{i})$, and thus will inevitably increase the time singularity in the estimate as stated in \eqref{eq:td}. Moreover, by the extraction formula, the Malliavin derivative of $\I_{D_{i,n}}$ will also generate a boundary term that has to be carefully treated by a merging procedure. This is the purpose of the following second boundary merging Lemma that we now describe. 

The proof of the following result is postponed to Section \ref{proof:lemma:merging:ibp}. As the resulting boundary merging weight is different, the notation for the boundary merged weight changes from $ \partial*e $ to $ \partial\circledast e $. The merged boundary process does not change. 

	\begin{lem}
		\label{lem:4.3a}
		Let $ f\in \mathscr{C}^0_p(\mathbb{R})$ and $n \in \bar{\N}$. The following property is satisfied for any $ i\in\mathbb{N}_{n-1} $ 
		\begin{align}
		\label{eq:faima}
	&  \E\left[f(\bar{X}_{{i+1}}) \I_{D_{i+1,n}} \ltheta^e_{{i+1}}\delta_L(\bar{X}_{i})(\zeta_{i} - \zeta_{i-1}) \bar{\theta}_i 
		\,\vert \,\mathcal{G}_{{i-1}}, \zeta_{i+1}, N_T = n  \right]\\
		\quad & = 
		\frac{\lambda^{-1}}{\zeta_{i+1}-\zeta_{i-1}}  \E\big[f(\bar{X}^{\partial}_{i-1,{i+1}})\I_{D^\partial_{i-1,i+1,n}}\ltheta^{\partial  \circledast e}_{{i-1},{i+1}}\,\vert \,\mathcal{G}_{{i-1}}, \zeta_{i+1}, N_T = n \big]\nonumber
		\end{align}
		\noindent where we define the boundary merged weight $ \ltheta^{\partial  \circledast e}_{j, i+1} $ for $j= i-1, i$ as
		\begin{align*}
		\ltheta^{\partial  \circledast e}_{j,i+1}:=&
		4\lambda^{-1}\frac{a'(L)-b(L)}
		{ a_{j}^{3/2}\sigma(L)}(\bar{X}_{j}-L) 
		 \left(\bar{{\mathcal{I}}}_{j,i+1}^2(\hat{d}^{i+1}_2)
		+
		\bar{{\mathcal{I}}}_{j,i+1}(\hat{d}^{i+1}_1)
		\right)
		\end{align*}
\noindent with coefficients given by
		\begin{align*}
		\hat{d}^{i+1}_k:=&\bar{d}^{i+1}_k\times(\bar{\Phi}g^{-1})(a(L)(\zeta_{i+1}-\zeta_{j}),Z_{j,i+1}),\quad k=1,2,\\
		\bar{\Phi}(t,z):=&\int_{|z|}^\infty g(t,y)dy.
		\end{align*}

		Similarly, for the last time interval, one has 
		\begin{align}
		\label{eq:faim:last:interval}
		& \E\big[f(\bar{X}_{n+1}) \I_{D_{n+1,n}}{\ltheta}^e_{n+1}\delta_L(\bar{X}_{n})(\zeta_{n}-\zeta_{n-1})\bar \theta_{n}
		\,\vert \,\mathcal{G}_{{n-1}},N_T=n\big] \\
		&= \frac{\lambda^{-1}}{T-\zeta_{n-1}}  \E\big[f(\bar{X}^{\partial}_{n-1, n+1})
		\I_{D_{n-1,n+1,n}^\partial}
		 {\ltheta}^{\partial  \circledast e}_{n-1,n+1} \,\vert \,\mathcal{G}_{n-1},N_T=n\big]\nonumber
\end{align}
		
\noindent where $  {\ltheta}^{\partial  \circledast e}_{j,n+1 } :=4e^{\lambda T}\frac{2a'(L)-b(L)} { a_{j}^{3/2}\sigma(L)}(\bar{X}_{j}-L) \hat{d}^{n+1}$, $\hat{d}^{n+1}:=(\bar{\Phi}g^{-1})(a(L)(\zeta_{n+1}-\zeta_{j}),Z_{j,n+1} )$ so that $ |{\ltheta}^{\partial  \circledast e}_{j,n+1}|\leq C $ a.s., for $j=n-1,n$.
	
		Moreover, for any $i \in \bar{\N}_{n}$, for any $ j=i-1,i $, $ \ltheta^{\partial  \circledast e}_{j,i+1} \in {\mathbb{S}}_{j,i+1 ,n}(\bar{X}^{\partial}) $ and it satisfies the time degeneracy estimates in the sense of \eqref{time:degeneracy:estimate:boundary:process}.
	\end{lem}

	\begin{remark}
		\label{rem:5}
(i) We again emphasize the role of the two indexes $j=i-1, i$, for $i\in {\mathbb{N}}_n$, in the definition of the boundary merged weights $\ltheta^{\partial *e}_{j,i+1}$ and $\ltheta^{\partial  \circledast e}_{j,i+1}$ of Lemmas \ref{lem:4.3} and \ref{lem:4.3a}. The index $j=i-1$ is used for the boundary merging operation while the index $j=i$ will be used once the reduction of jumps operation (a.k.a. time merging) is performed in Section \ref{sec4.3}. \\
(ii) The reason why the terms $ \bar{\Phi}g^{-1} $ appear in the definitions of the coefficients $ \hat{d}^{i+1} $, contrary to the previous Lemma \ref{lem:4.3}, is due to the time increments $ \zeta_{i}-\zeta_{i-1} $ and $ \zeta_n-\zeta_{n-1} $ on the left hand side of equalities \eqref{eq:faima} and \eqref{eq:faim:last:interval}.
	\hfill\break 
	(iii)
	An important technical remark which follows from the definitions of the merged weights $\ltheta^{\partial *e}_{j,i+1}$ and $\ltheta^{\partial \circledast e}_{j,i+1}$, $ j=i-1,i $  in Lemmas \ref{lem:4.3} and \ref{lem:4.3a} is that they
	do not depend neither on $ \rho_i $ or $ \rho_{i+1} $. 
	\end{remark}

\section{A first IBP formula: Putting the ingredients to work}
\label{sec:!IBP}
As explained at the beginning of Section \ref{sec:ingr},
 in order to obtain an IBP formula, one first has to take the conditional expectation w.r.t the Poisson process $ N $ inside $ \E\left[\partial_{\bar{X}
 	_{N_T+1}}f(  
 \bar{X}
 _{N_T+1}) \prod_{{i}=1}^{N_T+1} \I_{D_{i,N_T}}\bar {\theta}_i \right] $. Once the jump times are fixed, one observes that there is a Markov chain structure to which one may apply the transfer of derivatives procedure described in Lemma \ref{lem:5.1}. This leads to a tree structure of terms which are combined up to the time interval where we decide to stop the transfer of derivatives and to perform a local IBP formula using only the noise on this specific interval. 
 Throughout this section we will often denote this interval by the general index $ k $. For example, the tree structure is illustrated in Figure \ref{fig:4} where the IBP is performed on the time interval $ [\zeta_{k-1},\zeta_k] $ for $ k=2 $. The application of the IBP formula \eqref{eq:IBP} on the time interval
  $ [\zeta_{k-1},\zeta_k] $ will give after using the extraction formula \eqref{eq:exta} the new weight
  \begin{align*}
  \mathcal{I}_k(\I_{D_{k,n}}\bar{\theta}_k) =
  	\delta_L(\bar{X}_k)\bar{\theta}_k+\I_{D_{k,n}}\mathcal{I}_k(\bar{\theta}_k).
  \end{align*}
  We thus see that a boundary term appears which is treated by the boundary merging procedure described in Lemma \ref{lem:4.3a} and thus gives rise to the new weight associated to the symbol $ \partial\circledast e $.

 The terms which contain boundary weights due to the successive application of the transfer of derivatives formula \eqref{eq:5.1a} are treated by the boundary merging procedure of Lemma \ref{lem:4.3}. Figure \ref{fig:4} shows these terms before the boundary merging procedure in a simplified algebraic notation that will be introduced in this section. 
 
 This boundary merging operation will lead to the recovery of an integrable time degeneracy estimate and to the reduction of one jump unit (after an application of the so-called time merging procedure) in the underlying Poisson process. Finally, a local IBP formula is performed on the interval $ [\zeta_{k-1},\zeta_k] $ where we have stopped the transfer of derivatives. As described above, after using the extraction formula, the new weight given by $\mathcal{I}_k(\I_{D_{k, n}} \bar{\theta}_k)$ will also generate a boundary term to which we apply the boundary merging procedure described by Lemma \ref{lem:4.3a}. As before, this will in turn reduce the number of jump times by one unit. 
 
 We importantly note that when one stops the (backward) transfer of derivatives procedure and decide to perform a local IBP formula on the time interval $ [\zeta_{k-1},\zeta_k] $, the weights $\I_{D_{j, n}} \bar{\theta}_{j}$ for the preceding time intervals, namely $[\zeta_{j-1}, \zeta_{j}]$, for $j=1, \cdots, k-1$, correspond to the original probabilistic representation in Theorem \ref{th:3.2a} and thus remain unchanged. A similar remark applies when a time merging of weights or a correction weight appears on one branch of the tree structure below.
 
 The overall tree diagram before carrying out the boundary merging procedure can be schematically described as follows in the case of $N_T=4$ jumps of the Poisson process
 \begin{figure}[H]
 	\begin{tikzpicture}
 	\node (0) at (0,-1)  {$ 0=\zeta_0 $};
 	\node (00) at (0,-3.1){$ 0=\zeta_0  $};
 	\node at (1,-0.7){$0$};
 	\node at (1,-3.3){$0$};
 	\node (11) at (1.8,-1){$\zeta_{1}$};
 	\node (011) at (1.8,-3.1){$\zeta_{1}$};
 	\node at (2.55,-0.7){$\mathcal{I}$};
 	\node (12) at (3.3,-1){$\zeta_{2}$};
 	\node at (4.05,-0.7){$e$};
 	\node (13) at (4.8,-1){$\zeta_{3}$};
 	\node (013) at (4.8,-3.1){$\zeta_{3}$};
 	\node (30) at (3.3,-3.1){$ \zeta_2 $};
 	\node (31) at (1.8,-3.1){};	\node(300) at (3.3,-3.1){};
 	\node (41) at (1.8,-3){};
 	\node (73) at (8.9,-2.1){};
 	\node (50) at (9,-2.5){};
 	\draw[->, blue] (011)--(12)
 	node[pos=0.5,above] {$ \partial $};
 	\draw[->] (31)--(30) 
 	node[pos=.4,sloped,below] {$ 0 $};
 	\draw[-latex,bend right]   (300) edge (13);
 	\node at (4.3,-1.9) {$ c $};
 	\node at (5.7,-1.9) {$ c $};
 	\draw[-latex,blue, bend left]   (300) edge (13);
 	\node at (3.4,-1.9) {$ \textcolor{blue}{ \partial   }$};
 	\node at (4.9,-1.9) {$ \textcolor{blue}{ \partial  } $};
 	\node (36) at (7.7,-2.2){};
 	\node (46) at (7.8,-3){};
 	
 	\node (16) at (6.3,-1){$\zeta_{4}$};
 	\node (016) at (6.3,-3.1){$\zeta_{4}$};
 	\node at (5.5,-0.7){$e$};
 	
 	\node at (7,-0.7){$ e$};
 	\node (17) at (7.8,-1){$T$};
 	\draw[->] (0) -- (11);
 	\draw[->] (00) -- (011);
 	\draw[-latex,blue,bend left]  
 	(013) edge (16);
 	\draw[-latex,bend right]  
 	(013) edge (16);

 	\draw[->] (30) -- (013)
 	node[pos=.6,below] {$ 0 $}	;

 	\draw[->] (013) -- (016)
 	node[pos=.6,below] {$ 0 $}	;
 	\draw[->] (016) -- (17)
 	node[pos=.4, above] {$ c$}	;
 	\draw[->] (11) -- (12);
 	\draw[->] (12) -- (13);
 	\draw[->] (13) -- (16);
 	\draw[->] (16) -- (17);
 	\end{tikzpicture}
 	\caption{A tree in the case of $N_T=4$ jump times with an IBP performed on the time interval $ [\zeta_{1},\zeta_2] $ before any merging.}
 		\label{fig:4}
 \end{figure}

  An important remark in order to understand the tree notation to be introduced in the next section is that the blue curved arrows will lead to the application of the boundary merging Lemma \ref{lem:4.3}. Therefore, as the number of jumps will be reduced by the time merging procedure, those time merged weights will no longer be considered as weights associated to the set $ \{N_T=4\} $ but $ \{N_T=3\} $. A similar remark applies when applying Lemma \ref{lem:4.3a} to the boundary term generated by $ \mathcal{I}_k(\I_{D_{k, n }} \bar{\theta}_k) $ which is denoted by the straight blue arrow in Figure \ref{fig:4}.
  
 We believe that the above explanations given before the proof of our main result are important for the reader because they are necessary to understand the following section where the various symbols for the above tree structure are introduced.

\subsection{The tree structure in the IBP formula}\label{sec4.3}


In this section, we start by describing all the branches of the tree that are generated after performing transfer of derivatives and the two types of boundary merging previously described. 

 We denote by $ S_{n+1} $, the set of all symbol sequences of length $ n+1 $ described by the following vectors of length $ n+1 $. 
 More explicitly, when $n=0$, we define $S_1 :=\{I_1^1,\mathcal{B}_1^1\}$ with $ I_1^1:=(\mathcal{I}) $ and $ \mathcal{B}_1^1:=(\partial\circledast e)
  $ and for $n\in \bar \N$, 
 \begin{align*}
 S_{n+1} & := \left(\bigcup_{1\leq k \leq n+1}\{I_k^{n+1},\mathcal{B}_k^{n+1} \}\right) \bigcup \left( \bigcup_{2\leq k \leq n+1}\{C_k^{n+1},B_k^{n+1}\} \right).
 \end{align*}
 Therefore in $ S_{n+1}$ there are in total $2(n+1) + 2n$ vectors. The above vectors give the description of a branch in the IBP formula and are defined as follows. On the one hand, we let $C^{n+1}_k:=(0,...,0, c, e,...,e) $ and $ I^{n+1}_k:= (0,...,0, \mathcal{I}, e,...,e)$ where the component $ c $ or $\mathcal{I}$ appears in the $ k$-th coordinate for $ k=2,...,n+1 $ and we allow $ k=1 $ only for the vectors $ I^{n+1}_k$. On the other hand, we denote by $ {B}^{n+1}_{k}:=(0,...,0,\partial *e,e,...,e) $, $ k=2,\ldots,n+1$ and $ \mathcal{B} ^{n+1}_{k}:=(0,...,0,\partial \circledast e, e,...,e)$, $ k=1,\ldots, n+1$, where the element $ \partial* e$ or $ \partial  \circledast e $ appears in the $k$-th coordinate for $ k=2,\ldots,n $ and $ k=1,...,n$ respectively\footnote{In order to keep index notation short, so as not to have to always consider two cases when $ k $ takes as lowest value $ 1 $ or $ 2 $, we include $ C^{n+1}_1 $ and  $ B^{n+1}_1$ as symbols but any statement in this case should be taken as an empty statement or that the symbol corresponds to the zero (empty) element.}. 

 As mentioned previously, the index $ k $ corresponds to the time interval on which we perform the local IBP formula. In other words, if the symbol $0$ appears at the $j$-th coordinate of a vector, for $1\leq j<k$, this means that the weight corresponding to that time interval, namely $ [\zeta_{j-1},\zeta_j] $, remains the same as given in the probabilistic representation of Theorem \ref{th:3.2a}, i.e. $\I_{D_{j,n}}\bar{\theta}_j$. The symbol $c$ appearing at the $k$-th coordinate of the vector $C^{n+1}_k$ means that the new weight (associated to this vector and to the interval $ [\zeta_{k-1},\zeta_k] $) is $\I_{D_{k,n}}\ltheta^{c}_k$. The symbol $\mathcal{I}$ appearing at the $k$-th coordinate of $I^{n+1}_k$ stands for the IBP weight, i.e. $\I_{D_{k,n}}{\mathcal{I}}_{k} (\bar\theta_k)$. The symbol $ \partial *e$ corresponds to the merging between exchange (denoted by $e$) and boundary (denoted by $\partial$) r.v.'s according to Lemma \ref{lem:4.3} while the symbol $\partial  \circledast e$ corresponds to the same merging but occurring on the same interval as the one where we perform the IBP formula and is computed according to Lemma \ref{lem:4.3a}.

We remark here that in the last interval there is no boundary term because we always assume that the test function $f$ vanishes at the boundary.

If we fix our attention on all the branches that are generated with the objective of carrying out the local IBP formula on the interval $ [\zeta_{k-1},\zeta_k] $, we find the two following subsets of $ S_{n+1} $ which correspond to the branches that finish with a correction weight $ \ltheta^c $ and the branches that finish with a time merging of two time intervals, before leaving unchanged the remaining weights on remaining intervals. For $k = 1,\dots, n+1$, we define 
{
$$
\bar S_{n+1}^k := \bigcup_{k< j\leq n+1} \left\{ C^{n+1}_j \right\}, \quad \dot{S}_{n+1}^k  := \left\{\mathcal{B}^{n+1}_k\right\}  \bigcup \left( \bigcup_{k< j\leq n+1} \left\{ B^{n+1}_j \right\} \right)
$$ 
}
\noindent and in the special case $n=0$, we let $\dot{S}_{1}^1 := \left\{ \mathcal{B}^{1}_1\right\}$. 

 For example, the following figure describes the situation in the case when originally there were $n= 4 $ jump times. Observe that, on the one hand, due to the merging procedure, all sequences of arrows starting at time $ 0 $ and finishing at time $ T $ which contain time merged weights are thus considered on the set $ \dot{S}_{4}^2$, i.e. for $n=3$. On the other hand, the sequences of arrows that contain a correction term $ c $ are part of the set $\dot{S}_{5}^2  $. 
 
 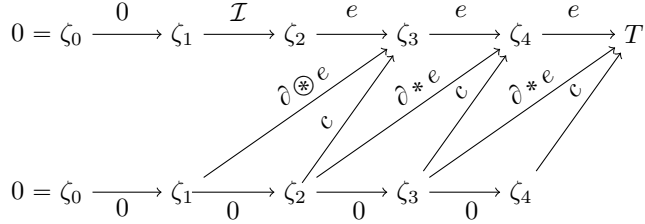
\begin{figure}[H]
 	\begin{tikzpicture}
 	\node (0) at (0,-1)  {$ 0=\zeta_0 $};
 	\node (00) at (0,-3.1){$ 0=\zeta_0  $};
 	\node at (1,-0.7){$0$};
 	\node at (1,-3.3){$0$};
 	\node (11) at (1.8,-1){$\zeta_{1}$};
 	\node (011) at (1.8,-3.1){$\zeta_{1}$};
 	\node at (2.55,-0.7){$\mathcal{I}$};
 	\node (12) at (3.3,-1){$\zeta_{2}$};
 	\node at (4.05,-0.7){$e$};
 	\node (13) at (4.8,-1){$\zeta_{3}$};
 	\node (013) at (4.8,-3.1){$\zeta_{3}$};
 	
 	\node (30) at (3.3,-3.1){$ \zeta_2 $};
 	
 	\node (31) at (1.8,-3.1){};
 	\node(300) at (3.3,-3.1){};
 	\node (41) at (1.8,-3){};

 	\node (73) at (8.9,-2.1){};
 	\node (50) at (9,-2.5){};
 	\draw[->] (31)--(30) 
 	node[pos=.5,sloped,below] {$ 0 $};
 	\draw[->]  (300) -- (13) 
 	node[pos=.4,sloped,above] {$c$};
 	\node (36) at (7.7,-2.2){};
 	\node (46) at (7.8,-3){};

 	\node (16) at (6.3,-1){$\zeta_{4}$};
 	\node (016) at (6.3,-3.1){$\zeta_{4}$};
 	\node at (5.5,-0.7){$e$};
 	
 	\node at (7,-0.7){$ e$};
 	\node (17) at (7.8,-1){$T$};
 	

 	\draw[->] (0) -- (11);
 	\draw[->] (00) -- (011);
 	\draw[->] (011) -- (13)
 	node[pos=.6,sloped,above] {$ \partial\circledast e $}	;
 	\draw[->] (30) -- (16)
 	node[pos=.6,sloped,above] {$ \partial * e $}	;
 	\draw[->] (30) -- (013)
 	node[pos=.6,below] {$ 0 $}	;
 	\draw[->] (013) -- (16)
 	node[pos=.6,sloped,above] {$ c $}	;
 	\draw[->] (013) -- (17)
 	node[pos=.6,sloped,above] {$ \partial *e $}	;
 	\draw[->] (013) -- (016)
 	node[pos=.6,below] {$ 0 $}	;
 	\draw[->] (016) -- (17)
 	node[pos=.6,sloped, above] {$ c$}	;
 	\draw[->] (11) -- (12);
 	\draw[->] (12) -- (13);
 	\draw[->] (13) -- (16);
 	\draw[->] (16) -- (17);
 	\end{tikzpicture}
 	\caption{A tree in the case of $N_T=4 $ jump times with an IBP on the interval $ [\zeta_{1},\zeta_2] $ after time merging. All branches that contain a merged terms are associated to the set $ \{N_T=3\} $.}
 	\label{fig:5}
 \end{figure}

\subsection{The Markov chain and weights associated to the IBP tree branches}\label{tree:structure:ibp}

For the corresponding time partition $ \pi:=\{0=\zeta_0<...<\zeta_{n+1}=T\} $ of the underlying Poisson process on the set $\left\{N_T= n\right\}$, we now need to define the underlying Markov chain $ \bar{X}^{\mathbf{s}} $, for $ \bold{s}\in S_{n+1}$ or $ \bold{s}\in \dot{S}^k_{n+1}$, that will be used in the probabilistic representation of our first IBP formula. 

It will be defined as the same process as the original Markov chain $ \bar{X} $ except that on a certain time interval it may use the one-step transition of the merged boundary Markov chain $\bar{X}^{\partial}$ given by \eqref{eq:dbp}. To be more specific, in the case that $ {\bold{s}}= B^{n+1}_{k}$ or $\mathcal{B}^{n+1}_{k}$, $ k=1,...,n+1 $, the Markov chain $ \bar{X}^{\mathbf{s}} $ is defined by
\begin{align*}
\bar{X}^{\mathbf{s}}_j:=
\begin{cases}
\bar{X}_j,\ \text{ if } 0\leq j\leq k-1,\\
\bar{X}^{\partial}_{k-1,k}(\bar{X}_{k-1})\ \text{ for }
j=k,\\ 
\bar{X} _j({k},\bar{X}^{\partial}_{k-1,k})\text{ for }j=k+1,...,n+1.
\end{cases} 
\end{align*}

Here $ \bar{X}^{\partial}_{k-1,k}(\bar{X}_{k-1}) $ stands for the one-step transition defined in \eqref{eq:dbp}. Similarly, for $j = k+1, \cdots, n+1$, $\bar{X} _j({k},\bar{X}^{\partial}_{k-1,k})  $ is the flow notation for the scheme \eqref{eq:MCa} taken at step $j$ starting from the point $ \bar{X}^{\partial}_{k-1,k} $ at step $k$. For any other $ \bold{s} \in S_{n+1} $ of length $ n+1 $, the process $\bar{X}^{\bold{s}}$ corresponds to the original Markov chain dynamics, that is, we let $ \bar{X}^{\bold{s}}\equiv\bar{X} $. We also define the associated set
$$
D_{i,n}^{\bold{s}}:=\{ \bar{X}^{\bold{s}}_{i}\geq L,N_T=n\}.  
$$



We now introduce the weights corresponding to the  Markov chain $ \bar{X}^{\bold{s}} $ to be used in the IBP formula.
For each $ {\bold{s}}\in S_{n+1} $, we will define its associated weights $ (\ltheta^{\bold{s}}_1,...,\ltheta^{\bold{s}}_{n+1}) $ and its product\footnote{Recall the standard convention $\prod_{\emptyset} =1$. } as 
\begin{align*}
\ltheta^{\bold{s}}:=\prod_{j=1}^{n+1}\ltheta^{\bold{s}}_j.
\end{align*}

We now proceed to define the weights for each element $ {\bold{s}}=(s_1,...,s_{n+1})\in S_{n+1} $ as follows for $ i\in \N_{n+1}$:
\begin{enumerate}
	\item If $ s_i=0 $, then $ \ltheta^{\bold{s}}_i := \I_{D_{i,n}^{\bold{s}}}\bar\theta_i(\bar{X}^{\bold{s}}_{i-1},\bar{X}^{\bold{s}}_{i},\rho_i,\zeta_{i-1},\zeta_i) = \I_{D_{i,n}} \bar\theta_i(\bar{X}_{i-1},\bar{X}_{i},\rho_i,\zeta_{i-1},\zeta_i)$.
	\item If $ s_i\in\{c,e\} $, with the notation introduced in \eqref{eq:defDc}
	\begin{align*}
	\ltheta^{\bold{s}}_{{i}}:= \I_{D_{i,n}^{\bold{s}}}
	\ltheta^{s_i}_{{{i}}}(\bar{X}^{\bold{s}}_{i-1},\bar{X}^{\bold{s}}_{i},\rho_i,\zeta_{i-1},\zeta_i).
	\end{align*} 
	As remarked at the beginning of Section \ref{sec:3a}, the weight $ \ltheta^{\bold{s}}_{{{i}}} $ is an explicit function of the underlying Markov chain and we will make use of such property in what follows.
	\item Similarly, if $ s_i\in\{\partial*e,\partial\circledast e,\mathcal I\} $, then using \eqref{eq:defDc}, Lemmas \ref{lem:4.3} and \ref{lem:4.3a}
	\begin{align*}
	\ltheta^{\bold{s}}_{{i}}:=
	\begin{cases}
	\ltheta^{\partial*e}_{i-1,{i}}(\bar{X}^{\bold{s}}_{i-1},\bar{X}^{\bold{s}}_{i},\rho_i,\zeta_{i-1},\zeta_i)\I_{D_{i,n}^{\bold{s}}},\ \text{ if } s_i=\partial*e,\\
	(\zeta_i-\zeta_{i-1})^{-1}\ltheta^{\partial  \circledast e}_{i-1,{i}}(\bar{X}^{\bold{s}}_{i-1},\bar{X}^{\bold{s}}_{i},\rho_i,\zeta_{i-1},\zeta_i)
	\I_{D_{i,n}^{\bold{s}}}
	,\ \text{ if } s_i=\partial\circledast e,\\
	\mathcal{I}_i(\bar{\theta}_i)\I_{D^{\bold{s}}_{i,n}}
	,\ \text{ if } s_i = \mathcal I.
		\end{cases}
	\end{align*} 
	In the case of time merging, we importantly refer the reader to Remark \ref{rem:5}. Due to this, we know that the reduction of one jump does not affect the sequence $ \rho_i$, $ i\in\mathbb{N}_{n+1} $ as the boundary merged weight does not depend on it. 
\end{enumerate}

\subsection{The IBP formula}

\begin{theorem}
	\label{prop:5.2}
	Let $ f\in \mathscr{C}^1_b(\mathbb{R})$ such that $ f(L)=0 $. Under assumption \textbf{(H}), the following IBP formula is satisfied:
	\begin{align*}
	{T} 	\E[f'(X_{T})\I_\seq{\tau\geq T}]  =&
	\E\left [\sum_{k=1}^{N_T+1} (\zeta_k - \zeta_{k-1})
	\left\{f(		
	\bar{X}_{N_T+1})
	\ltheta^{I^{N_T+1}_k}+\sum_{\mathbf{s}\in \bar{S}^k_{N_T+1}\cup \dot{S}^k_{N_T+1}}f(\bar{X}^{\mathbf{s}}_{N_T+1})
	\ltheta^{\mathbf{s}}\right\}
	\right].
	\end{align*}
	Moreover, the r.v. appearing inside the expectation of the right-hand side of the above equality belongs to $\mathbb{L}^p(\mathbb{P} ) $ for $ p\in [0,2) $.
\end{theorem}

\begin{proof} We start from the Markov chain representation given by Theorem \ref{th:3.2a}, namely 
	$$ 
	\E[f'(X_{T})\I_\seq{\tau\geq T}]  = \E[  f'(  \bar{X}_{N_T+1}) \prod_{i=1}^{N_T+1} \I_{D_{i,N_T}} \bar{\theta}_{i} ] = \sum_{n\geq0} \E\Big[\E[f'(  \bar{X}_{n+1})  \prod_{i=1}^{n+1} \I_{D_{i,n}}\bar\theta_{i} \vert T^{n+1} ] \,\I_{\{N_T = n\}} \Big]  .
	$$
	where we notice that $\seq{N_T = n} = \seq{T_{n+1} > T}\cap \seq{T_{n}\leq T}$.
	We remind the reader that $\mathbb{E}_{i,n}[X] $ is the expectation of $X$ conditional on $\{\mathcal{G}_i,T^{n+1},\rho^{n+1}, N_T=n\}$ and therefore for the rest of the proof we will work on the set $ \{N_T=n\} $.

	{\it Step 1: IBP on the last interval.}
		We start by proving the IBP for the last time interval. That is, by the tower property of conditional expectation and the integration by parts formula, \eqref{eq:IBP} on the (deterministic) time interval $[\zeta_{n}, T]$, noting that $ f(L)=0 $, one has
	\begin{align}
	\E[f'(  \bar{X}_{n+1}) \I_{D_{i, N_T}}\theta_{n+1} \prod_{i=1}^{n} \I_{D_{i,N_T}} \bar{\theta}_{i} \,|\, T^{n+1}] 
	& = \E[ \E_{n,n}[\mathcal{D}_{n+1}f(  \bar{X}_{n+1}) \I_{D_{n+1, n}} \theta_{n+1}]  \prod_{i=1}^{n} \I_{D_{i,n}} \bar{\theta}_{i}|\, T^{n+1}] \nonumber \\
	&= \E[  f(  \bar{X}_{n+1})  \I_{D_{n+1,n}} {\mathcal{I}}_{n+1}(\bar\theta_{n+1}) \prod_{i=1}^{n} \I_{D_{i,n}} \bar{\theta}_{i} |\, T^{n+1}]. \label{ibp:last:interval}
	\end{align}
		
{\it Step 2: The transfer of derivatives.}  In this step, we will perform the transfer of derivatives from the last time interval $[\zeta_{N_T}, T]$ to the time interval $ [\zeta_{k-1},\zeta_k] $.
In order to carry this step, using the Markov property of the process $\bar{X}$, we define for $k\in \N_n$ the functions:
\begin{align*}
F_{k}(\bar X_k)
:= \E_{k,n}\big[f(\bar{X}_{n+1})\prod_{i=k+1}^{n+1} \I_{D_{i,n}} \ltheta^{e}_{i}\big] =  \E\big[f(\bar{X}_{n+1})\prod_{i=k+1}^{n+1} \I_{D_{i,n}} \ltheta^{e}_{i}| \bar{X}_k, T^{n+1}, \rho^{n+1}, N_T=n\big], \quad 
\end{align*}
\noindent with the convention that $F_{n+1}(\bar{X}_{n+1}) = f(\bar{X}_{n+1})$ and $\prod_{\emptyset} \cdots = 1$. Note that the following recursive relation is satisfied for $ k\in\mathbb{N}_n $
\begin{align}
\label{eq:rec}
F_k(\bar X_k)= \mathbb{E}_{k,n}[F_{k+1}(\bar{X}_{k+1}) \I_{D_{k+1,n}}\ltheta^{e}_{k+1} ].
\end{align}

From the transfer of derivatives formula \eqref{eq:sdf} of Lemma \ref{lem:5.1}, we obtain
	\begin{align*}
	\E[  \partial_{\bar{X}_{n+1}}f(  \bar{X}_{n+1})  \prod_{i=1}^{n+1} \I_{D_{i,n}} \bar \theta_{i} |\, T^{n+1}] & =
	\E[  \partial_{\bar{X}_{n+1}}F_{n+1}(  \bar{X}_{n+1})  \prod_{i=1}^{n+1} \I_{D_{i,n}} \bar \theta_{i} |\, T^{n+1}]
	\\
	&= \E[\partial_{\bar{X}_{{n}}}F_n(\bar{X}_{{n}}) \prod_{i=1}^{n} \I_{D_{i,n}} \bar \theta_{i} |\, T^{n+1}] \\
	&\quad + \E[f(
	\bar{X}_n) \I_{D_{n+1,n}}{\ltheta}^{c}_{n+1} \prod_{i=1}^{n} \I_{D_{i,n}} \bar \theta_{i}\, |\, T^{n+1}].
	\end{align*}

	Next, we proceed using a backward induction argument by combining successive applications of the transfer of derivative formula of Lemma \ref{lem:5.1} with the tower property of conditional expectation. To be more specific, from Lemma \ref{lem:5.1}, one has for $ k\in\mathbb{N}_n $
	\begin{align}
	\E_{k,n}[ \partial_{\bar{X}_{{k+1}}}F_{k+1}(
	\bar{X}_{{k+1}}) \I_{D_{k+1, n}}\bar \theta_{k+1}] =& \partial_{\bar{X}_{{k}}}F_{k}(\bar{X}_{{k}}) + \E_{k,n}[F_{k+1}(\bar{X}_{{k+1}}) \I_{D_{k+1,n}} \ltheta^{c}_{{k+1}}] \\
	&+  \E_{k,n}[F_{k+1}(
	\bar{X}_{{k+1}}) \delta_{L}(\bar{X}_{{k+1}}) \ltheta^{\partial}_{{k+1}}] \nonumber
	\end{align}
	\noindent which directly implies by iteration 
	\begin{align}
	\E[  f'(\bar{X}_{n+1}) \prod_{i=1}^{n+1} \I_{D_{i,n}} \bar \theta_{i} |\, T^{n+1}] & = \E[\partial_{\bar{X}_{{k}}}F_{k}(\bar{X}_{{k}}) \prod_{i=1}^{k} \I_{D_{i,n}}\bar \theta_{i} |\, T^{n+1}]  + \sum_{j=k+1}^{n}\E[F_{j}(\bar{X}_{{j}})\I_{D_{j,n}} \ltheta_{j}^c \prod_{i=1}^{{j}-1} \theta_i |\, T^{n+1}]  \nonumber \\
	&  \quad +  \sum_{j=k+1}^{n}\E[F_{j}(\bar{X}_{{j}}) \delta_L(\bar{X}_{{j}})  \ltheta_{j}^{\partial} \prod_{i = 1}^{{j}-1} \I_{D_{i,n}} \bar \theta_i |\, T^{n+1}]\label{temp:transfer:lem:ibp}\\
	&  \quad + \E[f(
	\bar{X}_T)\I_{D_{n+1,n}} {\ltheta}^{c}_{n+1} \prod_{i=1}^{n} \I_{D_{i,n}}\bar \theta_{i} |\, T^{n+1}] \nonumber
	\end{align}
	\noindent for $k=1, \dots, n$, with the convention $\sum_{\emptyset} \cdots = 0$. 
	
	{\it Step 3: The local IBP on the interval $ [\zeta_{k-1},\zeta_k] $.}	
To proceed, we first notice that $\bar\theta_{k}\in {\mathbb{S}}_{k,n}(\bar{X})$, for $k\in \N_{n+1}$, is a smooth r.v.. Then from the tower property of conditional expectation (using $\mathbb{E}_{{ k-1  },n}[\cdot]$), the integration by parts formula \eqref{eq:IBP}, the extraction formula \eqref{eq:exta} and \eqref{eq:coeff} one obtains
	\begin{align*}
	\E[ \partial_{\bar{X}_{{k}}}F_{k}(
	\bar{X}_{{k}}) \theta_{k} \,\vert\, \mathcal{G}_{{k-1}},T^{n+1}] & = \E[F_{k}(
	\bar{X}_{{k}}) 
	\I_{D_{k,n}} {\mathcal{I}}_{ {k}}(\bar\theta_{k}) \,\vert\, \mathcal{G}_{{k-1}},T^{n+1}] - \E[F_{k}(
	\bar{X}_{{k}}) \delta_L(
	\bar{X}_{{k}}) \bar\theta_{k} \,\vert\, \mathcal{G}_{{k-1}},T^{n+1}], \\
	& =  \E[F_{k}(
	\bar{X}_{{k}}) 
	\I_{D_{k,n}}
	 {\mathcal{I}}_{ {k}}(\bar\theta_{k}) \,\vert\, \mathcal{G}_{{k-1}},T^{n+1}] + \E[F_{k}(
	\bar{X}_{{k}}) \delta_L(
	\bar{X}_{{k}}) \ltheta^{\partial}_{{k}} \,\vert\, \mathcal{G}_{{k-1}},T^{n+1}].
	\end{align*}
	
	Now, plugging the previous identity into \eqref{temp:transfer:lem:ibp} and using the recursive formula \eqref{eq:rec} yield
	for $k\in\mathbb{N}_n$:
	\begin{align}
	& \nonumber\E[  f'(  
	\bar{X}_{n+1})  \prod_{i=1}^{n+1} \I_{D_{i,n}} \bar \theta_{i} |\, T^{n+1}] \\
	& = \E[ F_{k}(\bar{X}_{{k}}) \I_{D_{k,n}}
 {\mathcal{I}}_{{k}}(\bar\theta_{k}) \prod_{i=1}^{k-1} \I_{D_{i,n}} \bar {\theta}_i  |\, T^{n+1}] + \sum_{j=k+1}^{n}\E[F_{j}(
	\bar{X}_{{j}}) \I_{D_{j,n}}\ltheta_{j}^c \prod_{i=1}^{{j}-1} \I_{D_{i,n}} \bar \theta_i |\, T^{n+1}] \nonumber \\
	& \quad + \sum_{j=k}^{n}\E[F_{j+1}(\bar{X}_{{j+1}})\I_{D_{j+1,n}}\ltheta^e_{j+1} { \delta_L(\bar{X}_{{j}})  }\ltheta_{j}^{\partial} \prod_{i=1}^{{j}-1} \I_{D_{i,n}} \bar \theta_i |\, T^{n+1}]\label{ibp:k:interval} \\
	& \quad + \E[f(\bar{X}_{n+1})\I_{D_{n+1,n}} {\ltheta}^{c}_{n+1} \prod_{i=1}^{n} \I_{D_{i,n}} \bar {\theta}_i |\, T^{n+1}]. \nonumber  
	\end{align}

	{\it Step 4: Combining all the local IBP formulae.}
	Multiplying \eqref{ibp:last:interval} by $(T-\zeta_n)$, \eqref{ibp:k:interval} by $(\zeta_k - \zeta_{k-1})$ and summing from $k=1$ to $n$ the resulting equalities, we obtain
	\begin{align*}
	T \E[  f'(\bar{X}_{n+1})  \prod_{i=1}^{n+1} \I_{D_{i,n}} \bar \theta_{i} |\, T^{n+1}]  & = \sum_{k=1}^{n+1} (\zeta_{k} - \zeta_{k-1}) \E[  \partial_{\bar{X}_{n+1}}f(\bar{X}_{n+1}) \prod_{i=1}^{n+1} \I_{D_{i,n} }\bar \theta_{i} |\, T^{n+1}] ,\\
	 = & \E[  f(\bar{X}_{n+1})  \I_{D_{n+1,n}} (T-\zeta_n) {\mathcal{I}}_{n+1}(\bar\theta_{n+1}) \prod_{i=1}^{n} \I_{D_{i,n}} \bar {\theta}_i |\, T^{n+1} ] \\
	&  + \sum_{k=1}^n \E[ f(\bar{X}_{n+1}) \prod_{i =k+1}^{n+1}  \I_{D_{i,n}}\ltheta^{e}_{i} \times \I_{D_{k,n}} (\zeta_k-\zeta_{k-1}) {\mathcal{I}}_{{k}}(\bar\theta_{k}) \times  \prod_{i=1}^{k-1} \I_{D_{i,n}}\bar {\theta}_i|\, T^{n+1}  ] \\
	&  + \sum_{k=1}^n (\zeta_{k} - \zeta_{k-1}) \sum_{j=k+1}^{n}\E[f(\bar{X}_{n+1})\prod_{i=j+1}^{n+1} \I_{D_{i,n}}\ltheta^{e}_{i} \times  \I_{D_{j,n}}\ltheta_{j}^c \times \prod_{i=1}^{{j}-1} \I_{D_{i,n}} \bar \theta_i|\, T^{n+1} ] \\
	& + \sum_{k=1}^n (\zeta_k-\zeta_{k-1}) \sum_{j=k}^{n} \E[f(\bar{X}_{n+1}) \prod_{i=j+1}^{n+1} { \I_{D_{i,n}}  }\ltheta^e_{i} \times \delta_L(\bar{X}_{{j}}) \,  \ltheta_{j}^{\partial} \times\prod_{i=1}^{{j}-1} \I_{D_{i,n}}\bar \theta_i|\, T^{n+1} ] \\
	& +  \sum_{k=1}^{n} (\zeta_k -\zeta_{k-1}) \E[f(\bar{X}_{n+1})  \I_{D_{n+1,n}}{\ltheta}^{c}_{n+1} \prod_{i=1}^{n} \I_{D_{i,n}} \bar {\theta}_i |\, T^{n+1}]
	\end{align*}

\noindent where we used the fact that $T-\zeta_n + \sum_{k=1}^n (\zeta_k - \zeta_{k-1}) = T-\zeta_0=T$ in the first equality. The final argument starts by multiplying the above identity by $\I_\seq{N_T = n}$ and taking the expectation of both hand sides with respect to the Poisson process. 

{\it Step 5: The { merging  } procedure.} Now, we perform the boundary and time merging procedures for the fourth term appearing in the right-hand side of the above equality. This is done by first conditioning with respect to $\mathcal{G}_{j-1}, \zeta_{j+1}, N_T = n$ in the inside sum and then by applying Lemma \ref{lem:4.3}, for $ j>k $ and Lemma \ref{lem:4.3a}. More specifically, for any $ n\geq 1 $, $k \in \N_{n}$ and $ k < j \leq n$, 
\begin{align*}
&   \E\Big[ f(\bar{X}_{n+1})\prod_{i=j+1}^{n+1}   \I_{D_{i,n}}\ltheta^{e}_{i} \times \delta_L(
	\bar{X}_{{j}}) \ltheta_{j}^{\partial} \, | \mathcal{G}_{j-1}, \zeta_{j+1}, N_T=n \Big] \\
	& =  \E\Big[ F_{j+1}(\bar{X}_{j+1})   \I_{D_{j+1,n}}\ltheta^{e}_{j+1} \times \delta_L(
	\bar{X}_{{j}}) \ltheta_{j}^{\partial} \, | \mathcal{G}_{j-1}, \zeta_{j+1}, N_T=n \Big] \\
	& =  \frac{\lambda^{-1}}{\zeta_{j+1}-\zeta_{j-1}} \E\Big[ F_{j+1}(\bar{X}^{\partial}_{j-1,j+1})   \I_{D^{\partial}_{j-1, j+1,n}}\ltheta^{\partial * e}_{j-1, j+1} \, | \mathcal{G}_{j-1}, \zeta_{j+1}, N_T=n \Big]
\end{align*}

\noindent and for $j=k$,
\begin{align*}
 &   \E\Big[ f(\bar{X}_{n+1})\prod_{i  = k+1}^{n+1}   \I_{D_{i,n}}\ltheta^{e}_{i} \times \delta_L(\bar{X}_{{k}}) (\zeta_{k}-\zeta_{k-1})\ltheta_{k}^{\partial} \, | \mathcal{G}_{k-1}, \zeta_{k+1}, N_T=n \Big] \\
	& =  \E\Big[ F_{k+1}(\bar{X}_{k+1})   \I_{D_{k+1,n}}\ltheta^{e}_{k+1} \times \delta_L(
	\bar{X}_{{k}}) \ltheta_{k}^{\partial} \, | \mathcal{G}_{k-1}, \zeta_{k+1}, N_T=n \Big] \\
	& =  \frac{\lambda^{-1}}{\zeta_{k+1}-\zeta_{k-1}} \E\Big[ F_{k+1}(\bar{X}^{\partial}_{k-1, k+1})   \I_{D^{\partial}_{k-1, k+1,n}}\ltheta^{\partial \circledast e}_{k-1, k+1} \, | \mathcal{G}_{k-1}, \zeta_{k+1}, N_T=n \Big]
\end{align*}

\noindent so that
\begin{align*}
& \E\Big[\sum_{k=1}^n (\zeta_k-\zeta_{k-1}) \sum_{j=k}^{n} \E[f(\bar{X}_{n+1}) \prod_{i=j+1}^{n+1} { \I_{D_{i,n}}  }\ltheta^e_{i} \times \delta_L(\bar{X}_{{j}}) \,  \ltheta_{j}^{\partial} \times\prod_{i=1}^{{j}-1} \I_{D_{i,n}}\bar \theta_i|\, T^{n+1} ] \I_\seq{N_T=n} \Big] \\
& = \E\Big[\sum_{k=1}^n (\zeta_k-\zeta_{k-1}) \times \sum_{j=k+1}^{n} F_{j+1}(\bar{X}^{\partial}_{j-1,j+1}) \lambda^{-1}(\zeta_{j+1}-\zeta_{j-1})^{-1} \I_{D^{\partial}_{j-1, j+1, n}}  \ltheta^{\partial * e}_{j-1, j+1} \times\prod_{i=1}^{{j}-1} \I_{D_{i,n}} \bar \theta_i   \I_\seq{N_T=n} \Big]\\
&   + \E\Big[\sum_{k=1}^n (\zeta_k-\zeta_{k-1})  F_{k+1}(\bar{X}^{\partial}_{k-1,k+1}) \lambda^{-1} (\zeta_{k+1}-\zeta_{k-1})^{-1}\I_{D^{\partial}_{k-1, k+1, n}} \ltheta^{\partial \circledast e}_{k-1, k+1} \times\prod_{i=1}^{k-1} \I_{D_{i,n}} \bar \theta_i\I_\seq{N_T=n} \Big]
\end{align*}

 We then apply Lemma \ref{lem:Unif} to the above identity. In order to do it, we make use of the following decomposition
\begin{align*}
& \E\Big[\sum_{k=1}^n (\zeta_k-\zeta_{k-1}) \sum_{j=k}^{n} \E[f(\bar{X}_{n+1}) \prod_{i=j+1}^{n+1} { \I_{D_{i,n}}  }\ltheta^e_{i} \times \delta_L(\bar{X}_{{j}}) \,  \ltheta_{j}^{\partial} \times\prod_{i=1}^{{j}-1} \I_{D_{i,n}}\bar \theta_i|\, T^{n+1} ] \I_\seq{N_T=n} \Big] \\
& = \E\Big[\sum_{k=1}^{n}\sum_{j=k+1}^{n} (\zeta_{j+1}-\zeta_{j-1})^{-1} G_1(\zeta_{1}, \cdots, \zeta_{j-1}, \zeta_{j+1}, \cdots, \zeta_{N_T+1}) \I_\seq{N_T=n} \Big] \\
& \quad +  \E\Big[\sum_{k=1}^{n} (\zeta_{k+1}-\zeta_{k-1})^{-1} G_2(\zeta_{1}, \cdots, \zeta_{k-1}, \zeta_{k+1}, \cdots, \zeta_{N_T+1}) \I_\seq{N_T=n} \Big]
\end{align*}

\noindent where $G_1$ and $G_2$ are the measurable functions defined for $j\geq k+1$ by
\begin{align*}
& G_1(\zeta_1, \cdots, \zeta_{j-1}, \zeta_{j+1}, \cdots, \zeta_{N_T+1})  :=  \lambda^{-1} (\zeta_{k}-\zeta_{k-1}) \E\Big[F_{j+1}(\bar{X}^{\partial}_{j-1, j+1}) \I_{D^{\partial}_{j-1, j+1, n}}  \ltheta^{\partial * e}_{j-1, j+1} \times\prod_{i=1}^{{j}-1} \I_{D_{i,n}} \bar \theta_i |\, T^{n+1} \Big], \\
& G_2(\zeta_1, \cdots, \zeta_{k-1}, \zeta_{k+1}, \cdots, \zeta_{N_T+1})  :=  \lambda^{-1}  (\zeta_{k}-\zeta_{k-1})\E\Big[F_{k+1}(\bar{X}^{\partial}_{k-1, k+1}) \I_{D^{\partial}_{k-1, k+1, n}}  \ltheta^{\partial \circledast e}_{k-1, k+1} \times\prod_{i=1}^{k-1} \I_{D_{i,n}} \bar \theta_i |\, T^{n+1} \Big].
\end{align*}

Now observe that the weights $ \ltheta^{\partial *e}_{j-1,j+1}, \ltheta^{\partial  \circledast e}_{k-1, k+1}$ satisfy the time degeneracy estimates in the sense of \eqref{time:degeneracy:estimate:boundary:process} and so do the weights $\ltheta^{e}_i$, $\bar \theta_{i}$, so that from the tower property of conditional expectation and Lemma \ref{lem:7} with $p=1$, the following estimates hold
\begin{align*}
 |G_1|& \leq  C^{n} \lambda^{-1} (\zeta_{k}-\zeta_{k-1}) \prod_{i=1, i\neq j, j+1}^{n+1} (\zeta_{i}-\zeta_{i-1})^{-\frac12} (\zeta_{j+1}-\zeta_{j-1})^{-\frac12}, \\
   |G_2|& \leq C^{n} \lambda^{-1} (\zeta_{k}-\zeta_{k-1}) \prod_{i=1, i \neq k, k+1}^{n+1} (\zeta_{i}-\zeta_{i-1})^{-\frac12} (\zeta_{k+1}-\zeta_{k-1})^{-\frac12}
 \end{align*}

\noindent with the convention $\zeta_0=0$, $\zeta_{n+1}=T$ on $\left\{N_T=n\right\}$. 

Using that conditional on the event $\left\{ N_T=n , \zeta_1,\cdots,\zeta_{j-1},\zeta_{j+1},\cdots,\zeta_{N_T} \right\}$, the distribution of $ \zeta_j $ is uniform on $ [\zeta_{j-1},\zeta_{j+1}] $, we get
\begin{align*}
& \E[(\zeta_{j+1}-\zeta_{j-1})^{-1} |G_1| \I_\seq{N_T=n}] \\
& \leq C^{n} \int_0^T \cdots \int_0^{s_2}  (s_{k}-s_{k-1}) (s_{j+1}-s_{j-1})^{-\frac32}  \prod_{i=1, \neq j, j+1}^{n+1} (s_{i}-s_{i-1})^{-\frac12}\, ds_1 \cdots ds_{n}\\
& \leq C^{n} \int_0^T \cdots \int_0^{s_2} (s_{k}-s_{k-1})  (s_{j+1}-s_{j-1})^{-\frac12}  \prod_{i=1, i\neq j, j+1}^{n+1} (s_{i}-s_{i-1})^{-\frac12}\, ds_1 \cdots ds_{j-1}ds_{j+1}\cdots ds_{n}\\
& < \infty.
\end{align*}

The same argument yields $\E[(\zeta_{k+1}-\zeta_{k-1})^{-1} |G_2| \I_\seq{N_T=n}] < \infty$. From Lemma \ref{lem:Unif}, we thus obtain
	\begin{align*}
	\E\Big[(\zeta_k-\zeta_{k-1}) & \sum_{j=k}^{n}f(\bar{X}_{n+1})\prod_{i=j+1}^{n+1}   \I_{D_{i,n}}\ltheta^{e}_{i} \times \delta_L(
	\bar{X}_{{j}}) \ltheta_{j}^{\partial} \times \prod_{i=1}^{{j}-1} \theta_i \, \I_\seq{N_T=n} \Big] \\
	= &\E\Big[(\zeta_k-\zeta_{k-1})  \sum_{j=k+1}^{n}f(\bar{X}^{B^{n}_{j}}_{n})
	\prod_{i=j+1}^{n}  \I_{D^{B^n_j}_{i,n}} \ltheta^{e,B_{j}^{n}}_i \times  \I_{D^{B^n_j}_{j,n}}\ltheta^{\partial * e,B_{j}^{n}}_{j}  \times \prod_{i=1}^{{j}-1} \theta_i \,  \I_\seq{N_T=n-1}\Big]\\
	&+\E\Big[ f(\bar{X}^{\mathcal{B}^{n}_{k}}_{n})
	\prod_{i=k+1}^{n} \I_{D^{\mathcal{B}^{n}_{k}}_{i,n}} \ltheta^{e,\mathcal{B}_{k}^{n}}_i \times  
	\I_{D^{\mathcal{B}^{n}_{k}}_{k,n}}\ltheta^{\partial \circledast e, \mathcal{B}_{k}^{n}}_{k}  \times \prod_{i=1}^{{k}-1} \theta_i  \, \I_\seq{N_T=n-1}\Big].
	\end{align*}
	
	{\it Step 6: Integrability properties.}
The proof of the $\mathbb{L}^{p}(\P)$-integrability, $p\in [0,2)$, follows from the time degeneracy estimates for each weight in the Lemmas \ref{lem:5.1}, \ref{lem:4.3}, \ref{lem:4.3a} and \ref{lem:7} combined with a similar argument to the one employed at the end of the proof in Section \ref{app:sec} (see the discussion following Lemma \ref{lem:7}). This also implies that the infinite sum over $n$ converges absolutely and therefore after a re-ordering of the different terms one obtains the claimed formula. This concludes the proof. 
\end{proof}
\begin{remark}\label{other:formulation}
(i) The right hand side of the IBP formula may alternatively be written in a longer but maybe more appealing format as 
\begin{align*}
	{T} \E[f'(X_{T})\I_\seq{\tau\geq T}]  & = \E\left[f(\bar{X}_{N_T+1})\sum_{k=1}^{N_T+1} (\zeta_k - \zeta_{k-1}) \left\{ \ltheta^{I^{N_T+1}_k} + \sum_{j=k+1}^{N_T+1} \ltheta^{C^{N_T+1}_j}  \right\}\right]\\
		& \quad + \E\left[\sum_{k=1}^{N_T+1} (\zeta_k - \zeta_{k-1}) \sum_{j=k+1}^{N_T+1}f(\bar{X}_{N_T+1}^{B^{N_T+1}_{j}})  \ltheta^{B^{N_T+1}_{j}}\right]  + \E\left[\sum_{k=1}^{N_T+1}f(\bar{X}^{\mathcal{B}^{N_T+1}_{k}}_{N_T+1}) \ltheta^{\mathcal{B}^{N_T+1}_{k}}\right].
\end{align*}

\noindent (ii) The restriction $f(L)=0$, can be easily removed if one considers the test function $ \tilde{f}(x)=f(x)-f(L) $ instead of $f$.
\end{remark}

Note that, the previous theorem not only yields an IBP formula which is suitable for Monte Carlo simulation, it also provides a probabilistic representation for the derivative, with respect to the terminal point, of the transition density of the killed process at time $T$. To be more specific, under assumption \textbf{(H)}, for any bounded measurable $f$ defined on $[L,\infty)$, one deduces $\E[f(X_T) \I_\seq{\tau > T}] = \int^{\infty}_{L} f(z) p(T, x, z) \, dz$ where $(0,\infty)\times [L,\infty)^2 \ni (T, x ,z) \mapsto p(T, x, z)$ is the transition density of the killed process at time $T$ starting from $x$ at time $0$. Moreover, from Theorem \ref{prop:5.2}, by a standard approximation argument that we omit, one deduces that $z\mapsto p(T, x, z)$ is differentiable on $[L,\infty)$.\footnote{We refer the interested reader to \cite{FKL1} for an alternative proof based on analytic arguments.} The next result provides a probabilistic representation of this derivative from which directly stems an unbiased Monte Carlo simulation method.
\begin{corol}
	\label{cor:ibp:backward}
	Under assumption \textbf{(H)}, the transition density of the killed process at time $T$ is differentiable with respect to its terminal point. Moreover, for all $ (T, x ,z) \in (0,\infty)\times [L,\infty)^2$ the following probabilistic representation holds
	\begin{align*}
	T  \partial_z p(T, x,  z)&=
	\E\left [\sum_{k=1}^{N_T+1} (\zeta_k - \zeta_{k-1}) g(a(\bar{X}_{N_T})(T-\zeta_{N_T}),z-\bar{X}_{N_T}) \ltheta^{I^{N_T+1}_k}\right] \\
	&\qquad  + \E\left[\sum_{k=1}^{N_T+1} (\zeta_k - \zeta_{k-1})\sum_{\mathbf{s}\in \bar{S}^k_{N_T+1}\cup \dot{S}^k_{N_T+1}} p_{\mathbf{s}}(T-\zeta_{N_T},\bar{X}^{\mathbf{s}}_{N_T}, z)\ltheta^{\mathbf{s}}\right].
	\end{align*}
	Here $p_{\mathbf{s}}$ denotes the transition density of the Markov chain $\bar{X}^{\bold{s}}$, that is, $ z\mapsto p_{\mathbf{s}}(T-\zeta_{N_T}, x, z)$ is the density of the r.v. $\bar{X}^{\bold{s}}_{N_T+1}$ conditional on $\left\{\bar{X}^{\bold{s}}_{N_T} =x\right\}$. In particular, for $ {\bold{s}}= B^{N_T+1}_{N_T+1}$ or $\mathcal{B}^{N_T+1}_{N_T+1}$, the dynamics \eqref{eq:dbp} readily gives $p_{\mathbf{s}}(T-\zeta_{N_T}, \bar{X}^{\mathbf{s}}_{N_T}, z) = g(a(L) (\zeta_{N_T+1}-\zeta_{N_T}), z - (L (1-\mu(\bar{X}_{N_T})) + \bar{X}_{N_T} \mu(\bar{X}_{N_T}))$ since $\bar{X}^{\mathbf{s}}_{N_T} = \bar{X}_{N_T}$.
\end{corol}

\section{Bismut-Elworthy-Li type formula}\label{BEL:sec}
In this section, we briefly derive the IBP formula for $ \partial_x\E[f(X_{T})\I_\seq{\tau\geq T}]  $ commonly referred in the literature as the Bismut-Elworthy-Li formula. As we will see, obtaining this formula is simpler than the IBP formula derived in Theorem \ref{prop:5.2} as it does not involve boundary weights. Therefore the proof of this formula will not require any merging procedure contrary to the one of Theorem \ref{prop:5.2}. 

First, we give the transfer of derivatives lemma which is carried out forward in time in comparison with Lemma \ref{lem:5.1} where it is done backward in time. The proof is similar to the one of Lemma  \ref{lem:5.1} given in Section \ref{app:tl} and therefore we omit it. Note that due to the change of time direction some changes of notation and sign occur among other changes in the formulae for the weights. 
\begin{lem}\label{lem:tran2}
	Let $f\in \mathscr{C}^1_p(\mathbb{R} )$ and $n \in \bar \N$. Then, the following transfer of derivative formula holds for $ i\in\bar{\mathbb{N}}_{n-1} $:
\begin{align*}
 \partial_{\bar X_{i}}\E_{i,n}[f(\bar X_{i+1})\I_{D_{i+1,n}}\bar \theta_{i+1}] 		& =  \E_{i,n}[\partial_{\bar X_{i+1}}f(\bar X_{i+1})\I_{D_{i+1,n}} \rtheta^e_{i+1}]\\
&\quad  + \E_{i,n}[f(\bar X_{i+1})\left(\I_{D_{i+1,n}} \rtheta^c_{i+1}+\delta_L(\bar X_{i+1}) \rtheta^\partial_{i+1}\right)],
\end{align*}
	\noindent where the r.v.'s $(\rtheta^e_{i+1},\rtheta^c_{i+1},\rtheta^\partial_{i+1}) \in {\mathbb{S}}_{i+1,n}(\bar{X}) $ are defined by
	\begin{align*}
		\nonumber
		 \rtheta^e_{i+1}:=&
		2\lambda^{-1}\left ({\mathcal{I}}_{i+1}^2(d_2^{i+1})
		+{\mathcal{I}}_{i+1}(d^{i+1}_1)
		\right ),
		\\
		{ \rtheta}_{i+1}^c:=& 
		{{\mathcal{I}}}_{i+1}\left(
		(2\rho_{i+1}-1)\bar \theta_{i+1} - \rtheta^e_{i+1}
		\right) 
		+ \partial_{\bar{X}_i}
		 \bar \theta_{i+1}
		+
		\sigma'_i{{\mathcal{I}}}_{i+1}\left(Z_{i+1}
		\bar \theta_{i+1}
		\right) ,\\
			\nonumber
		{{\rtheta}}_{i+1}^\partial:=&
		\rtheta^e_{i+1}, \\
		d_1^{i+1}:=&c_1^{i+1}+ (2\rho_{i+1}-1)\partial_{\bar{X}_{i}}c_2^{i+1},\\
		d_2^{i+1}:=&c_2^{i+1}.
	\end{align*}
	
	In a similar way, let $ f\in \mathscr{C}^1_p(\mathbb{R})$ such that $f(L) = 0$. Then, the following transfer of derivative formula is satisfied
	\begin{align*}
		 \partial_{\bar{X}_{n}}\mathbb{E}_{n,n}\big[f(\bar X_{n+1})\I_{D_{n+1,n}} \bar {\theta}_{n+1}\big]  & =  \mathbb{E}_{n,n}\big[\partial_{\bar X_{n+1}}f(\bar X_{n+1}) \I_{D_{n+1,n}} {\rtheta}^e_{n+1}\big] \nonumber  
	\end{align*}
	 with ${\rtheta}^e_{n+1}:=2e^{\lambda T} (1+(2\rho_{n+1}-1)\sigma'_n Z_{n+1})$ so that $\forall p\geq1$, $\| \rtheta^e_{n+1}\|_{p, n,n}\leq C$. We also set $\rtheta^{c}_{n+1}:= 0$ for notational convenience.
	
	 With the above definitions, for any $ i\in \bar{\mathbb{N}}_{n-1}$, we have 
		 $ \mathbb{E}_{i,n}[f(\bar X_{i+1}) \I_{D_{i+1,n}} \bar \theta_{i+1}]\big |_{\bar{X}_i=\cdot} \in \mathscr{C}^1_p(\mathbb{R} ) $ a.s. Moreover, the weights $ \rtheta^a_{i+1} $, $ a\in\{e,c,\partial\} $, $ i\in\bar{\mathbb{N}}_n $, satisfy the time degeneracy estimates. 
\end{lem}
As we did before the statement of Theorem \ref{prop:5.2}, we need to define the weights that will be used in the BEL formula. It will be apparent in what follows that no boundary terms will appear so that no merging operation is needed. Similarly to Section \ref{sec4.3}, we consider the vectors of length $n+1$ defined by $\widehat C^{n+1}_{k} = (e,\dots,e, c,0,\dots, 0)$ where the symbol $c$ appears in the $k$-th coordinate for $k \in \mathbb{N}_{n+1}$ and $\widehat I^{n+1}_{k} = (e,\dots,e, \mathcal{I},0,\dots, 0)$ where the symbol $\mathcal{I}$ appears in the $k$-th coordinate for $k\in\mathbb{N}_{n+1}$. Again in order to keep index notation short, we included $\widehat C^{n+1}_{n+1} $ as a symbol but any statement in this case should be taken as an empty statement or that the symbol corresponds to an empty element. We then define the set $ \widehat{S}^k_{n+1}= \bigcup_{1\leq j \leq k}\{ \widehat C^{n+1}_j\}$ for $k\in\mathbb{N}_{n+1}$. Note that since there are no merging terms, the weights $\rtheta^e_i$ and $\rtheta^c_i$ are given by Lemma \ref{lem:tran2} applied on the set $ \{N_T=n\} $. Again, we consider the mapping, which to a vector $\bold{s} \in \widehat S^k_{n+1}$, provides the associated product of weights denoted by $\rtheta^{\bold{s}}$, namely for $i \in\mathbb{N}_{n+1}$ 
\begin{align*}
\rtheta^{C^{n+1}_{i}}:=&\prod_{\ell=i+1}^{n+1} \theta_\ell \times
\I_{D_{i,n}} \rtheta_{i}^{c} \times
\prod_{j=1}^{i-1}	\I_{D_{j,n}} \rtheta_{j}^{e},
\\
\rtheta^{{I}^{n+1}_{i}}:=&	\prod_{\ell={i}+1}^{n+1}
\theta_\ell \times \I_{D_{i,n}}
{\mathcal{I}}_{{{i}}}(
\rtheta^e_i) \times
\prod_{j=1}^{i-1} \I_{D_{j,n}} \rtheta^e_j.
\end{align*}
\begin{theorem}\label{thm:forward:ibp}
		Let $ f\in \mathscr{C}^1_b(\mathbb{R})$ such that $ f(L)=0 $.  {Using the weights defined above, the following Bismut-Elworthy-Li formula is satisfied for any initial point $x\in [L,\infty)$:}
		\begin{align*}
		{T} \partial_x	\E[f(X_{T})\I_\seq{\tau\geq T}]  =&
		\E\left [f(\bar{X}_{N_T+1})\sum_{k=1}^{N_T+1} (\zeta_k - \zeta_{k-1})\left\{ \rtheta^{I^{N_T+1}_k}+\sum_{\mathbf{s}\in \widehat{S}^k_{N_T+1}} \rtheta^{\mathbf{s}}\right\}\right].
		\end{align*}
		Moreover, the r.v. appearing inside the expectation in the right-hand side of the above equality belongs to $\mathbb{L}^p(\mathbb{P} ) $, for any $ p\in [0,2) $.
\end{theorem}

\begin{proof}  As in the proof of Theorem \ref{prop:5.2}, we have  
\begin{align*}
	\E[f(X_{T})\I_\seq{\tau\geq T}]  & = \sum_{n\geq0} \E\Big[\E[f(  \bar{X}_{n+1})  \prod_{i=1}^{n+1} \I_{D_{i,n}}\bar \theta_{i} \vert T^{n+1} ] \,\I_{\{N_T = n\}} \Big]  .
\end{align*}
 In most of the arguments below, we will work on the set $\{N_T=n\} $. In order to perform a forward induction argument through the  Markov chain structure, we define for $ k\in\mathbb{N}_{n+1}$ the functions 
\begin{align*}
	 \widehat F_{k}(\bar X_k)
	:= \E_{k,n}\big[f(\bar{X}_{n+1})\prod_{i=k+1}^{n+1} \I_{D_{i,n}} \bar \theta_{i}\big] = \E\big[f(\bar{X}_{n+1})\prod_{i=k+1}^{n+1}  \I_{D_{i,n}} \bar \theta_{i}| \bar{X}_k, T^{n+1}, \rho^{n+1}, N_T=n\big].
\end{align*}

We let $\widehat F_{n+1}(\bar X_{n+1}):=f(\bar{X}_{n+1})$ and the following recursive relation is satisfied for $ k\in\mathbb{N}_n $
	\begin{align}
	\widehat F_k(\bar X_k)= \mathbb{E}_{k,n}[ \widehat F_{k+1}(\bar X_{k+1}) \I_{D_{k+1,n}} \theta_{k+1} ].\label{recur2}
	\end{align}

Then, iterating the transfer of derivative formula in Lemma \ref{lem:tran2}
for $k\in\mathbb{N}_n$,  we obtain\footnote{As before, we use the convention $\sum_{\emptyset} \cdots = 0$, $\prod_{\emptyset} \cdots =1$.}
	\begin{align}
	\label{eq:29}
\partial_x\E[  f(  
	\bar{X}_{n+1}) \prod_{i=1}^{n+1}\I_{D_{i,n}} \bar \theta_{i} |\, T^{n+1}] 
	&= \E[\mathcal{D}_k\widehat F_{k}(\bar{X}_{{k}}) \prod_{i=1}^{k} \I_{D_{i,n}}\rtheta^e_{i} |\, T^{n+1}]+ \sum_{j=1}^{k}\E[\widehat F_{j}(\bar{X}_{{j}})\I_{D_{j,n}} \rtheta_{j}^c \prod_{i=1}^{{j}-1} \I_{D_{i,n}}  \rtheta^e_i |\, T^{n+1}]  \\
	& \qquad + \sum_{j=1}^{k}\E[\widehat F_{j}(\bar{X}_{{j}}) \delta_L(\bar{X}_{{j}})\rtheta_{j}^{\partial} \prod_{i=1}^{{j}-1} \I_{D_{i,n}}  \rtheta^e_i |\, T^{n+1}]. \nonumber
	\end{align}

To further simplify the first term on the right-hand side of the above equation, we use the tower property of conditional expectation, the integration by parts formula \eqref{eq:IBP} and the extraction formula \eqref{eq:exta} to obtain
	\begin{align*}
	\E[ \mathcal{D}_k\widehat F_{k}(\bar{X}_{{k}}) \I_{D_{k,n}}\rtheta^e_{k} \,\vert\, \mathcal{G}_{{k-1}},T^{n+1}] & = \E[\widehat F_{k}(
	\bar{X}_{{k}})\I_{D_{k,n}}
	 {\mathcal{I}}_{ {k}}(\rtheta^e_{k}) \,\vert\, \mathcal{G}_{{k-1}},T^{n+1}] \\
	& \qquad - \E[\widehat F_{k}(\bar{X}_{{k}}) \delta_L(\bar{X}_{{k}})  \rtheta^e_{k} \,\vert\, \mathcal{G}_{{k-1}},T^{n+1}].
			\end{align*}

	In the case $k = n$, using the transfer of derivative formula of Lemma \ref{lem:tran2} on the last time interval and then performing the IBP formula \eqref{eq:IBP}, noting that $f(L) = 0$, we obtain the representation
	\begin{align}
	 \partial_x\E[  f(\bar{X}_{n+1}) \prod_{i=1}^{n+1} \I_{D_{i,n}} \bar \theta_{i} |\, T^{n+1}] &=  \E[\mathcal{D}_{n+1}f(\bar{X}_{n+1})\prod_{i=1}^{n+1} \I_{D_{i,n}} \rtheta^{e}_i |\, T^{n+1}] \nonumber\\
	& \qquad + \sum_{j=1}^{n}\E[\widehat F_{j}(\bar{X}_{{j}}) \I_{D_{j,n}} \rtheta_{j}^c \prod_{i=1}^{{j}-1} \I_{D_{i,n}} \rtheta^e_i |\, T^{n+1}] \nonumber \\
	& \qquad + \sum_{j=1}^{n}\E[\widehat F_{j}(\bar{X}_{{j}}) \delta_L(\bar{X}_{{j}})\rtheta_{j}^{\partial} \prod_{i=1}^{{j}-1} \I_{D_{i,n}}  \rtheta^e_i |\, T^{n+1}]. \nonumber \\
	& = \E[f(\bar{X}_{n+1}) \I_{D_{n+1,n}} \mathcal{I}_{n+1}(\rtheta^{e}_{n+1}) \prod_{i=1}^{n} \I_{D_{i,n}}\rtheta^{e}_i |\, T^{n+1}]\nonumber \\
	& \qquad + \sum_{j=1}^{n+1}\E[\widehat F_{j}(\bar{X}_{{j}}) \I_{D_{j,n}} \rtheta_{j}^c \prod_{i=1}^{{j}-1} \I_{D_{i,n}}  \rtheta^e_i |\, T^{n+1}] \label{nn} \\
	&  \qquad + \sum_{j=1}^{n}\E[\widehat F_{j}(\bar{X}_{{j}}) \delta_L(\bar{X}_{{j}})\rtheta_{j}^{\partial} \prod_{i=1}^{{j}-1} \I_{D_{i,n}}  \rtheta^e_i |\, T^{n+1}] \nonumber 
	\end{align}	

\noindent where we remind the reader that we previously set $\rtheta^{e}_{n+1} =0$ in Lemma \ref{lem:tran2} for notational convenience.
 At this stage we emphasize that all the boundary terms in \eqref{nn} vanish. In fact, as $ f(L) =0$, the boundary term in the last interval vanishes as stated in Lemma \ref{lem:tran2}. For the other intervals, from  \eqref{recur2}, we claim that for $ j\in\mathbb{N}_n $,
\begin{equation}\label{claim:boundary:term:zero}
\E[\widehat F_{j+1}(\bar X_{j+1})\I_{D_{j+1,n}} \bar \theta_{j+1}\delta_L(\bar X_{j})\rtheta^\partial_{j}\,\vert\, \mathcal{G}_{j-1},T^{n+1}] = 0.
\end{equation}
To see this, we note that for $\bar X_j = L$, one has $\bar X_{j+1} = L+\sigma(L)Z_{j+1}$, which is independent of $\rho_{j+1}$, and $\bar \theta_{j+1}$ given by \eqref{eq:13aa} reduces to $\bar \theta_{j+1} = 2 (2\rho_{j+1}-1)\lambda^{-1}( \frac12(a(\bar{X}_{j+1})-a(L)) \mathcal{I}^2_{j+1}(1) + (b(L)-a'(\bar{X}_{j+1})) \mathcal{I}_{j+1}(1) + \frac{a''(\bar{X}_{j+1})}{2}-b'(\bar{X}_{j+1}))$. Therefore the conclusion follows by conditioning with respect to $\left\{ \bar{X}_{j} = L\right\}$ in \eqref{claim:boundary:term:zero} and by noting that $ \E[2\rho_{j+1}-1|\mathcal{G}_{j-1}, T^{n+1}, \bar{X}_{j} = L]=0 $.

From this property the identity \eqref{eq:29} becomes
\begin{align*}
\partial_x\E[  f(\bar{X}_{n+1}) \prod_{i=1}^{n+1} \I_{D_{i,n}} \bar \theta_{i} |\, T^{n+1}] 
=& \E[\widehat F_{k}(\bar{X}_{{k}})\I_{D_{k,n}} {\mathcal{I}}_{ {k}}(\rtheta^e_{k}) \prod_{i=1}^{k-1} \I_{D_{i,n}} \rtheta^e_{i} |\, T^{n+1}] \\
&  + \sum_{j=1}^{k}\E[\widehat F_{j}(\bar{X}_{{j}}) \I_{D_{j,n}} \rtheta_{j}^c \prod_{i=1}^{{j}-1} \I_{D_{i,n}}\rtheta^e_i |\, T^{n+1}].
\end{align*}

Now, for each $k\in\mathbb{N}_n$, one can multiply the above equality by the length of the interval on which the local IBP formula is performed, namely $\zeta_{k}- \zeta_{k-1}$ and sum them over all $k$. For the last interval, we multiply \eqref{nn} by $T-\zeta_{n}$. This gives 
\begin{align*}
&  T\partial_x\E[  f(  
	\bar{X}_{n+1}) \prod_{i=1}^{n+1} \I_{D_{i,n}} \bar \theta_{i} |\, T^{n+1}]\\
	=&\sum_{k = 1}^{n+1} (\zeta_{k}-\zeta_{k-1})\E[f(\bar{X}_{n+1}) \prod_{i=k+1}^{n+1} \I_{D_{i,n}} \bar \theta_{i} \times \I_{D_{k,n}} \mathcal{I}_{k}(\rtheta^e_k) \times \prod_{i=1}^{k-1}\I_{D_{i,n}} \rtheta^e_{i} |\, T^{n+1}]\\
	 & +\sum_{k = 1}^{n+1}(\zeta_{k}-\zeta_{k-1})\sum_{j=1}^{k}\E[f(\bar{X}_{n+1}) \prod_{i=j+1}^{n+1} \I_{D_{i,n}} \bar \theta_i \times  \I_{D_{j,n}}\rtheta_{j}^c \times \prod_{i=1}^{{j}-1}\I_{D_{i,n}} \rtheta^e_i |\, T^{n+1}].
\end{align*}
	From the above formula, the $\mathbb{L}^p$-moment estimate, $p\in [0,2)$, follows by similar arguments as described at the end of the proof of Theorem \ref{th:3.2a}. Finally, one concludes by using the Lebesgue differentiation which yields
$$ 
	T \partial_x \E[f(X_{T})\I_\seq{\tau\geq T}]  = T \partial_x \E[  f(  \bar{X}_{N_T+1}) \prod_{i=1}^{N_T+1} \theta_{i} ] = \sum_{n\geq0} \E\Big[T \partial_x\E[f(  \bar{X}_{n+1})  \prod_{i=1}^{n+1} \theta_{i} \vert T^{n+1} ] \,\I_{\{N_T = n\}} \Big]
$$

\noindent and summing the previous formula over $n$.
\end{proof}

\begin{remark}\label{rem:9}(i) The right hand side of the IBP formula may alternatively be written as 
\begin{align*}
	{T} \partial_x \E[f(X_{T})\I_\seq{\tau\geq T}]  & = \E\left[f(\bar{X}_{N_T+1})\sum_{k=1}^{N_T+1} (\zeta_k - \zeta_{k-1}) \left\{ \rtheta^{I^{N_T+1}_k} + \sum_{j=1}^{k} \rtheta^{C^{N_T+1}_j}  \right\}\right].
\end{align*}

\noindent (ii) We note that the above formula does not involve any merging procedure. This is due to the fact that only the first derivative is being considered here. In fact, the key property \eqref{claim:boundary:term:zero} would not be satisfied if one considers second order derivatives. Therefore, a merging procedure similar to the one described in Section \ref{sec:conv} would be necessary. 

\noindent (iii) Similarly to the previous section, the above theorem yields a probabilistic representation for the derivative of the transition density of the killed process at time $T$ from which stems an unbiased Monte Carlo simulation method. In contrast with Corollary \ref{cor:ibp:backward}, the derivative is taken with respect to the starting point. This can be seen by formally taking the Dirac mass at point $z$ as a test function in Theorem \ref{thm:forward:ibp}.\footnote{We again refer the interested reader to \cite{Men1} or to \cite{FKL1} for an analytical proof of the differentiability of the map $[L,\infty) \ni x\mapsto p(T, x, z)$.} 
\end{remark}

\begin{corol}
	\label{cor:ibp:forward}
	Under assumption \textbf{(H)}, the transition density of the killed process at time $T$ is differentiable with respect to its starting point. Moreover,  for all $ (T, x ,z) \in (0,\infty)\times [L,\infty)^2$ the following probabilistic representation holds
	\begin{align*}
	T  \partial_x p(T, x,  z) = &\E\left [g(a(\bar{X}_{N_T})(T-\zeta_{N_T}), z-\bar{X}_{N_T})\sum_{k=1}^{N_T+1} (\zeta_k - \zeta_{k-1})\left\{ \rtheta^{I^{N_T+1}_k}+\sum_{\mathbf{s}\in \widehat{S}^k_{N_T+1}} \rtheta^{\mathbf{s}}\right\}\right].
		\end{align*}
\end{corol}

\section{Achieving finite variance by importance sampling}\label{importance:sampling:section}
The previous probabilistic representations of Theorems \ref{th:3.2a}, \ref{prop:5.2} and \ref{thm:forward:ibp} as well as Corollaries \ref{cor:ibp:backward} and \ref{cor:ibp:forward} allow to devise an unbiased Monte Carlo simulation. However, in general, its use is hampered by the fact that the variance is infinite as suggested for instance by the moment estimate of Lemma \ref{lem:4.1}. The main tool that we develop here in order to circumvent this issue consists in employing an importance sampling scheme on the jump times of the Poisson process $N$ as originally proposed by Andersson and Kohatsu-Higa \cite{APKA}. Since the arguments developed below follow similar lines of reasonings as those employed in \cite{APKA}, we will omit some technical details.
 
Let us first introduce a renewal process in the following sense:
\begin{definition}\label{counting:process}
Let $(T_n)_{n\geq1}$ be a sequence of random variables such that $(T_{n} - T_{n-1})_{n\geq 1}$, with the convention $T_0=0$, are i.i.d. with density $f$ and c.d.f: $t\mapsto F(t) = \int_{-\infty}^{t} f(s) \, ds$. Then, the renewal process $J:=(J_t)_{t\geq0}$ with jump times $(T_n)_{n\geq1}$ is defined by $J_t:=\sum_{n\geq1} \textbf{1}_{\left\{ T_n \leq t\right\}}$.
\end{definition}

As previously done, we assume that $J$ is independent of the Brownian motion $W$. It is readily seen that $\left\{ J_t = n \right\} = \left\{T_n \leq t < T_{n+1} \right\}$ and by an induction argument that we omit, one may prove that the joint distribution of $(T_1, \cdots, T_n)$ is given by
$$
\mathbb{P}(T_1 \in ds_1, \cdots, T_n \in ds_n) = \prod_{j=0}^{n-1} f(s_{i+1}-s_i) \textbf{1}_{\left\{ 0 < s_1 < \cdots < s_n\right\}}
$$

\noindent which in turn implies
\begin{align}
\E[\textbf{1}_{\left\{ J_t =n \right\}} \Phi(T_1, \cdots, T_n)] & = \E[\textbf{1}_{\left\{T_n \leq t < T_{n+1}\right\}} \Phi(T_1, \cdots, T_n)] \\
& = \int_t^{\infty} \int_{\Delta_n(t)} \Phi(s_1, \cdots, s_n) \prod_{j=0}^{n} f(s_{j+1}-s_j) \, d\bold{s}_{n+1} 
\end{align}

\noindent so that, by Fubini's theorem
\begin{align}
\E[\textbf{1}_{\left\{ J_t =n \right\}} \Phi(T_1, \cdots, T_n)] & = \int_{\Delta_n(t)} \Phi(s_1, \cdots, s_n) (1- F(t-s_{n})) \prod_{j=0}^{n-1} f(s_{j+1}-s_j) \, d\bold{s}_{n} \label{probabilistic:representation:time:integrals}
\end{align}

\noindent for any map $\Phi: \Delta_n(t) \rightarrow \R$ satisfying $\E[\textbf{1}_{\left\{ J_t =n \right\}} |\Phi(T_1, \cdots, T_n)|] < \infty$. Usual choices that we will consider are the followings:

\noindent \textbf{Examples:}
\begin{enumerate}
\item If the density function $f$ is given by $f(t) = \lambda e^{-\lambda t} \textbf{1}_{[0,\infty)}(t)$ so that $F(t)= 1- e^{-\lambda t}$, $t\geq0$, for some positive parameter $\lambda$, then $J$ is a Poisson process with intensity $\lambda$.
  
\item If the density function $f$ is given by $f(t) = \frac{1-\alpha}{\bar{\tau}^{1-\alpha}}\frac{1}{t^{\alpha}} \textbf{1}_{[0, \bar{\tau}]}(t)$, so that $F(t)= (t/\bar \tau)^{1-\alpha}$, $t\in [0,\bar \tau]$, for some parameters $(\alpha, \bar{\tau})\in (0,1) \times (0,\infty)$, then $J$ is a renewal process with $[0,\bar{\tau}]$-valued $Beta(1-\alpha, 1)$ jump times. 

\item More generally, if the density function $f$ is given by $f(t) = \frac{\bar{\tau}^{1-\alpha-\beta}}{B(\alpha,\beta)}\frac{1}{t^{1-\alpha}(\bar{\tau}-t)^{1-\beta}} \textbf{1}_{[0, \bar{\tau}]}(t)$, so that $F(t) = B(t/\bar \tau, \alpha, \beta)/ B(\alpha, \beta)$, $[0,1] \ni x \mapsto B(x, \alpha, \beta)$ being the incomplete Beta function, for some parameters $(\alpha, \beta, \bar{\tau})\in (0,1)^2 \times (0,\infty)$, then $J$ is a renewal process with $[0,\bar{\tau}]$-valued $Beta(\alpha, \beta)$ jump times.
 \end{enumerate}

Having these definitions at hand, as done before, we define the partition $\pi$ of $ [0,T]$ given by $\pi:=\{0=:\tau_0<\cdots<\tau_{J_T} \leq T\} $ with $\tau_i := T_i \wedge T$. The probabilistic representation of Theorem \ref{th:3.2a} then becomes
\begin{align}
	\label{eq:PR:modified}
	\E\left[f(X_{T})\I_\seq{\tau> T}\right ]  &= \mathbb{E}\Big[  f(  
	\bar{X}
	_{J_T+1}) \prod_{{i}=1}^{J_T+1}  \I_{D_{i,J_T}}\bar{\theta}_i  \Big].
\end{align}
\noindent where the dynamics of the Markov chain $\bar X$ is given by \eqref{eq:MCa} with innovations $\left\{ Z_{i+1} = W_{\tau_{i+1}} - W_{\tau_i}, i=0, \cdots, J_T\right\}$, with the set $ D_{i,n}:=\{\bar{X}_{i}\geq L, J_T=n\}$ , $i\in \bar{\N}_{n+1}$, and for $i \in \bar \N_{J_T}$
\begin{align}
	\label{eq:new:weights:theta}
	\bar{\theta}_i :=&\I_\seq{J_T>i-1}
	2(2\rho_i-1) (f(\tau_{i} - \tau_{i-1}))^{-1}\left({{\mathcal{I}}}_{{{i}}}(c^i_1)+{{\mathcal{I}}}^2_{ {{i}}}(c^i_2)\right)
	+\I_\seq{J_T=i-1} (1-F(T-\tau_{J_T}))^{-1} (2\rho_{J_T+1}-1).
\end{align}
	
Moreover, the following estimates holds on the set $\left\{ J_T= n \right\}$
\begin{equation}
		\label{moment:estimates:prob:representation:new:weight} 
		  \I_\seq{i \leq  n}(\tau_{i}-\tau_{i-1})^{\frac{p}{2}} f(\tau_i - \tau_{i-1})^{p} \I_{D_{i-1,n}}\E_{i-1,n}\left[\I_{D_{i,n}}|{\bar\theta}_i |^p\right]+\I_\seq{i=n+1}\I_{D_{n,n}} (1-F(T-\tau_{n}))^{p}\E_{n,n}\left[\I_{D_{n+1,n}}|{\bar\theta}_{n+1}|^p\right ]\leq C
\end{equation}

\noindent where $C$ is positive constant independent of $n$. Hence, using \eqref{moment:estimates:prob:representation:new:weight} and then \eqref{probabilistic:representation:time:integrals}, for all $p\geq1$, one gets
\begin{align*}
\E\left[\Big|\prod_{i=1}^{J_T+1}\I_{D_{i,J_T}}\bar{\theta}_i\, \Big|^p\right]& = \sum_{n\geq0} \E\left[\Big|\prod_{i=1}^{n+1}\I_{D_{i,n}}\bar{\theta}_i\, \Big|^p \I_\seq{J_T=n}\right] \\
&  \leq C \sum_{n\geq0} C^{n+1} \int_{A_n} (1-F(T-t_n))^{-p+1} \prod_{i=1}^{n} (t_{i}-t_{i-1})^{-\frac{p}{2}} f(t_i-t_{i-1})^{-p+1} \, dt_1 \cdots dt_n.
\end{align*}

As already mentioned before, the previous estimate is actually quite sharp and the series appearing in the right-hand side is finite for $p\in [0,2)$ if $J$ is a Poisson process. In order to achieve a finite variance, one has to select the law of the jump times suitably. Indeed, if for instance $J$ is a renewal process with $[0,\bar{\tau}]$-valued $Beta(1-\alpha, 1)$ jump times, $\alpha \in (0,1)$, $\bar \tau >T$, a simple computation shows that the above series is finite as soon as $-\frac{p}{2} + \alpha (p-1) > -1$, that is, $p(\frac12-\alpha) < 1-\alpha$. In particular, taking $\alpha=1/2$, it is readily seen that the moment of order $p$ of the random variable appearing inside the expectation in the right-hand side of \eqref{eq:PR:modified} is finite for all $p\geq1$. Similarly, if $J$ is a renewal process with $[0,\bar{\tau}]$-valued $Beta(1-\alpha, 1-\beta)$ jump times, $(\alpha, \beta) \in (0,1)^2$, $\bar \tau >T$, the integral appearing in the right-hand side of the above inequality is finite as soon as $p(\frac12-\alpha) < 1-\alpha$ and $-p \beta< 1-\beta$, the later condition being always satisfied. In particular, taking $\alpha=1/2$ and any $\beta \in (0,1)$, any moment of order $p$ for $p\geq1$ is finite.

We now provide the probabilistic representation for the two IBP formulas using the above importance sampling technique. We thus redefine the weights appearing in the first IBP formula keeping in mind that the sequence $(\bar{\theta}_i)_{1\leq i \leq J_T+1}$ is now given by \eqref{eq:new:weights:theta}. The new random variables $(\ltheta^e_{i+1},\ltheta^c_{i+1},\ltheta^\partial_{i+1}) \in {\mathbb{S}}_{i+1,n}(\bar{X}) $ of Lemma \ref{lem:5.1} are now defined by: for $i \in \bar \N_{J_T-1} $
	\begin{align}
		\nonumber
		\ltheta^e_{i+1}:=&
		2 (f(\tau_{i+1}-\tau_i))^{-1} \left ({\mathcal{I}}_{i+1}^2(d_2^{i+1})
		+{\mathcal{I}}_{i+1}(d^{i+1}_1)
		\right ),
		\\
		\label{eq:defDc:new}
		{{\ltheta}}_{i+1}^c:=&
		{{\mathcal{I}}}_{i+1}\left(
		\bar\theta_{i+1}-(2\rho_{i+1}-1)
		\ltheta^e_{i+1}
		\right) -\partial_{\bar{X}_i}
		\ltheta^e_{i+1}
		-
		\sigma'_i{{\mathcal{I}}}_{i+1}\left(Z_{i+1}
		\ltheta^e_{i+1}
		\right) ,\\
			\nonumber
		{{\ltheta}}_{i+1}^\partial:=&
		2(2\rho_{i+1}-1) (f(\tau_{i+1}-\tau_i))^{-1} (a'(L)-b(L)){\mathcal{I}}_{i+1}(1)
	\end{align}
\noindent and $ \ltheta^\partial_{J_T+1}:=0 $, ${\ltheta}^e_{J_T+1}:=2 (1-F(T-\tau_{J_T}))^{-1}$ and ${\ltheta}^c_{J_T+1}:=-2 (1-F(T-\tau_{J_T}))^{-1} (\sigma'\sigma)_{J_T} (T-\tau_{J_T}){\mathcal{I}}^2_{J_T+1}(1)$. For the boundary merging weights of Lemmas \ref{lem:4.3} and \ref{lem:4.3a}, we first redefine the Markov chain $\bar{X}^{\partial}$ with dynamics \eqref{eq:dbp} by modifying the corresponding increments $Z_{j, i+1}$ as done previously. We then set $D^\partial_{j, i+1, n}:=\{\bar{X}^\partial_{j, i+1}\geq L, J_T=n\}, \, j=i-1, i$. Finally, on the set $\left\{ J_T =n \right\}$, 
\begin{align}
\ltheta^{\partial *e}_{j,i+1}:=&
4 (f(\tau_{i+1}-\tau_j))^{-1} \frac{a'(L)-b(L)}{a_{j}} \left(\bar{{\mathcal{I}}}_{j,i+1}^2(\bar{d}^{i+1}_2) + \bar{{\mathcal{I}}}_{j,i+1}(\bar{d}^{i+1}_1) \right), \label{new:boundary:merging:1}\\
\ltheta^{\partial  \circledast e}_{j,i+1} := &4 (f(\tau_{i+1}-\tau_j))^{-1} \frac{a'(L)-b(L)}{ a_{j}^{3/2}\sigma(L)}(\bar{X}_{j}-L)  \left(\bar{{\mathcal{I}}}_{j,i+1}^2(\hat{d}^{i+1}_2) + \bar{{\mathcal{I}}}_{j,i+1}(\hat{d}^{i+1}_1)\right)\label{new:boundary:merging:2}
\end{align}
\noindent for $j= i -1 , i$ with coefficients given by
		\begin{align*}
		\hat{d}^{i+1}_k:=&\bar{d}^{i+1}_k\times(\bar{\Phi}g^{-1})(a(L)(\tau_{i+1}-\tau_{j}),Z_{j,i+1}),\quad k=1,2,
		\end{align*}
\noindent  and ${\ltheta}^{\partial*e}_{j,n+1} := 4 (1-F(T-\tau_{j}))^{-1} \frac{a'(L)-b(L)}{a_{j}}$, $  {\ltheta}^{\partial  \circledast e}_{j,n+1 } :=4  (1-F(T-\tau_{j}))^{-1} \frac{2a'(L)-b(L)} { a_{j}^{3/2}\sigma(L)}(\bar{X}_{j}-L) \hat{d}^{n+1}$ for $j=n-1, n$ with $\hat{d}^{n+1}:=(\bar{\Phi}g^{-1})(a(L)(\tau_{n+1}-\tau_{j}),Z_{j,n+1} )$. 

At this stage, it is important to remark that the above new weights satisfy a time degeneracy estimate, with a slight modification of Definition \ref{def:td}. The notation $\E_{i, n}[X]$ now is used for the expectation of $X$ conditional on $\left\{ \mathcal{G}_i, T^{n+1}, \rho^{n+1}, J_T= n\right\}$ and one considers the corresponding norm $\|.\|_{p, i,n}$. We now say that a weight $ H\in \mathbb{S}_{i, n} $ satisfies the time degeneracy estimate if for all $ p\geq 1 $
\begin{equation}
\label{modification:degeneracy:estimate}
 \I_{D_{i-1,n}}\left\|\I_{D_{i,n}}H\right\|_{p,i-1,n}\leq Cf(\tau_{i}-\tau_{i-1})^{-1}(\tau_i-\tau_{i-1})^{-\frac 12} 
\end{equation}

\noindent in the case that $ i \in \N_n$ and $ \I_{D_{n, n}}\left\|\I_{D_{n+1,n}}H\right \|_{p, n, n}\leq C $ in the case that $ i=n+1 $. In a completely analogous manner as done in Lemma \ref{lem:5.1}, the weights $\ltheta^{a}_i$ for $a\in \left\{ e, c, \partial \right\}$ satisfies the time degeneracy estimate \eqref{modification:degeneracy:estimate}. For the boundary merging weights, we replace \eqref{modification:degeneracy:estimate} for $ H\in \mathbb{S}_{j,i+1,n}(\bar{X}^\partial) $ by 
\begin{align}
\label{modification:time:degeneracy:estimate:boundary:process}
\forall p\geq1, \quad \I_{D_{j,n}}\E\left[\I_{D^\partial_{j,i+1,n}}|H|^p \Big| \mathcal{G}_{j-1},\tau_{i+1},J_T=n\right] \leq C (f(\tau_{i+1}-\tau_j))^{-p}(\zeta_{i+1}-\zeta_j)^{-\frac p2}, \, i\in\mathbb{N}_{n-1}, 
\end{align}

\noindent for $j=i-1,i$ and$ \I_{D_{j,n}}\E\left[\I_{D^\partial_{j,n+1,n}}|H|^p \Big| \mathcal{G}_{j-1}, J_T=n\right] \leq C $ for $j=n-1,\, n$. Doing so, the new boundary merging weights $\ltheta^{\partial *e}_{j,i+1}$ and $\ltheta^{\partial  \circledast e}_{j,i+1}$ defined respectively by \eqref{new:boundary:merging:1} and \eqref{new:boundary:merging:2} both satisfy the time degeneracy estimate \eqref{modification:time:degeneracy:estimate:boundary:process}.
	
With the above new definitions and properties, we finally redefine the corresponding weights $\ltheta^{\bold{s}}$ with the related Markov chain $\bar{X}^{\bold{s}}$ for each $ \bold{s}\in S_{n+1}$ or $ \bold{s}\in \dot{S}^k_{n+1}$ on the time partition $\pi:=\left\{0=: \tau_0 < \cdots < \tau_{n+1}:= T\right\}$ of the underlying renewal process on the set $\left\{ J_T  = n\right\}$. This is done in a completely analogous manner as presented in the subsection \ref{tree:structure:ibp}. 

We can now restate Theorem \ref{prop:5.2} as follows. For any function $ f\in \mathscr{C}^1_b(\mathbb{R})$ satisfying $ f(L)=0$,
	\begin{align*}
	{T} 	\E[f'(X_{T})\I_\seq{\tau\geq T}]  =&
	\E\left [\sum_{k=1}^{J_T+1} (\tau_k - \tau_{k-1})
	\left\{f(		
	\bar{X}_{J_T+1})
	\ltheta^{I^{J_T+1}_k}+\sum_{\mathbf{s}\in \bar{S}^k_{J_T+1}\cup \dot{S}^k_{J_T+1}}f(\bar{X}^{\mathbf{s}}_{J_T+1})
	\ltheta^{\mathbf{s}}\right\}
	\right].
	\end{align*}
	Note also that the corresponding other formulation of Remark \ref{other:formulation} also holds. Moreover, if $J$ is a renewal process with $[0,\bar{\tau}]$-valued $Beta(1-\alpha, 1)$ jump times, $\bar \tau >T$, with $\alpha$ satisfying $p(\frac12-\alpha) < 1-\alpha$ or if $J$ is a renewal process with $[0,\bar{\tau}]$-valued $Beta(1-\alpha, 1-\beta)$ jump times, $\bar \tau >T$, with $\alpha$ and $\beta$ such that $p(\frac12-\alpha) < 1-\alpha$ and $\beta \in (0,1)$, then the r.v. appearing inside the expectation of the right-hand side of the above equality belongs to $\mathbb{L}^p(\mathbb{P} ) $ for any $p\geq 1$. The proof follows similar lines of reasonings as those employed above in order to deal with the new probabilistic representation \eqref{eq:PR:modified} and is thus omitted.

We proceed similarly for the BEL formula of Theorem \ref{thm:forward:ibp}. Namely, we redefine the weights $\rtheta^{a}$, for $a\in \left\{e, c,\partial\right\}$ as follows
\begin{align*}
\nonumber
		 \rtheta^e_{i}:=&
		2 (f(\tau_{i}-\tau_{i-1}))^{-1} \left ({\mathcal{I}}_{i}^2(d_2^{i})
		+{\mathcal{I}}_{i}(d^{i}_1)
		\right ),
		\\
		{ \rtheta}_{i}^c:=& 
		{{\mathcal{I}}}_{i}\left(
		(2\rho_{i}-1)\bar \theta_{i} - \rtheta^e_{i}
		\right) 
		+ \partial_{\bar{X}_{i-1}}
		 \bar \theta_{i}
		+
		\sigma'_i{{\mathcal{I}}}_{i}\left(Z_{i}
		\bar \theta_{i}
		\right) ,\\
			\nonumber
		{{\rtheta}}_{i}^\partial:=&
		\rtheta^e_{i}
\end{align*}
	and also set ${\rtheta}^e_{n+1}:=2 (1-F(T-\zeta_n))^{-1}(1+(2\rho_{n+1}-1)\sigma'_n Z_{n+1})$, $\rtheta^{c}_{n+1} :=0$. These new weights satisfy the time degeneracy estimate \eqref{modification:degeneracy:estimate}. Then, the following BEL formula is satisfied for any $f\in \mathcal{C}^{1}_b(\rr)$ satisfying $f(L)=0$ and any initial point $x\in [L,\infty)$:
		\begin{align*}
		{T} \partial_x	\E[f(X_{T})\I_\seq{\tau\geq T}]  =&
		\E\left [f(\bar{X}_{J_T+1})\sum_{k=1}^{J_T+1} (\zeta_k - \zeta_{k-1})\left\{ \rtheta^{I^{J_T+1}_k}+\sum_{\mathbf{s}\in \widehat{S}^k_{J_T+1}} \rtheta^{\mathbf{s}}\right\}\right].
		\end{align*}
		Moreover, the r.v. appearing inside the expectation in the right-hand side of the above equality belongs to $\mathbb{L}^p(\mathbb{P} ) $, for any $ p\geq1$ in the case of $[0,\bar{\tau}]$-valued $Beta(1-\alpha, 1)$ jump times or $[0,\bar{\tau}]$-valued $Beta(1-\alpha, 1-\beta)$ jump times under the condition $p(\frac12-\alpha) < 1-\alpha$ and $\beta \in (0,1)$. In particular, choosing $\alpha=1/2$, the $L^{p}(\mathbb{P})$ moment is finite for any $p\geq1$.

\section{Numerical tests}
\label{sec:7}
In this section, we provide some numerical results for the unbiased Monte Carlo simulation method based on the probabilistic representation formula established in Theorem \ref{th:3.2a} for the marginal law of the killed process and the Bismut-Elworthy-Li (BEL for short) formula of Theorem \ref{thm:forward:ibp}. A similar numerical analysis could be done for the IBP formula established in Theorem \ref{prop:5.2} but we restrict to the two aforementioned case for sake of simplicity. For a one dimensional Brownian motion $W$, we thus consider the following one-dimensional SDE with dynamics
\begin{equation}
\label{sde:dynamics:test}
X_t = x_0 + \int_0^{t} b(X_s) ds + \int_{0}^t \sigma(X_s) dW_s, \, x\in \rr
\end{equation}
\noindent and we choose the coefficients $b$, $\sigma$ and the test function $f$ as follows
$$
\sigma(x)= \bar \sigma \times (\sin(\omega x) +2), \quad b(x) = -\frac{x}{x^2  + \frac{c_1}{3 c_3}} \sigma(x), \quad f(x) = c_3 x^3  + c_1 x + c_0  
$$
\noindent for some positive constant $\bar \sigma >0$. With this particular choice, we first observe that assumption \textbf{(H)} is clearly satisfied and that a direct computation yields $\mathcal{L}f(x) := b(x) f'(x) + \frac12 \sigma^2(x) f''(x) = 0$. The process $(f(X_t))_{t\geq0}$ is thus a martingale and by Doob's stopping theorem one gets $\E[f(X_{\tau \wedge T})] = f(x_0)$ where $\tau = \inf\left\{ t\geq0: X_t \leq L\right\}$. We now shift the function $f$ by considering $h(x)= f(x) -f(L)$ satisfying $h(L)=0$ instead of $f$. It is readily seen that $h$ satisfies $\mathcal{L} h(x) = 0$ so that $\E[h(X_{\tau \wedge T})] = \E[h(X_T) \I_\seq{\tau >T}] = h(x_0)$. Note that since $h$ is not bounded but of polynomial growth, only Theorem \ref{th:3.2a} directly applies. However, the extension of Theorem \ref{thm:forward:ibp} (and also of Theorem \ref{prop:5.2}) to polynomially growing function can be performed by a standard approximation argument noting that all moments of $\bar{X}_{N_T+1}$ and $X_T$ are bounded.

Having this extension in mind, by Theorem \ref{thm:forward:ibp}, one has 
\begin{align*}
\forall x \geq L, \quad T h'(x) = \partial_x \E[h(X_T) \I_\seq{\tau >T}] & =\E\left [h(\bar{X}_{N_T+1})\sum_{k=1}^{N_T+1} (\zeta_k - \zeta_{k-1})\left\{ \rtheta^{I^{N_T+1}_k}+\sum_{\mathbf{s}\in \widehat{S}^k_{N_T+1}} \rtheta^{\mathbf{s}}\right\}\right].
\end{align*}

We select the following parameters: $T=0.5$, $L=0$, $c_0 = 0$, $c_1 = 1$, $c_3 = 1$ and $x_0=1$ so that $h(x_0) = 2$ and $T h'(x_0) = 2$. We use three different parameters sets for $\bar \sigma = \omega  = 0.1, \, 0.2, \, 0.3$. We examine the performance of the proposed Monte Carlo estimator with respect to the previous sets of parameters when one uses the Exponential sampling (the distribution of the jump times is exponential with parameter $\lambda$) as it is written in Theorem \ref{th:3.2a} and Theorem \ref{thm:forward:ibp} and when one uses an importance sampling technique with Beta distribution with parameters $(\gamma, \bar \tau)$ for the jump times of the renewal process, as exposed in Section \ref{importance:sampling:section}, which allows to achieve finite variance for our estimators. We note that though the variance is not finite in the case of Exponential time sampling, we include it here in order to compare its performance with the Beta sampling scheme. In both cases, we first select the optimal parameters which minimize the variance of the estimator using few samples, that is, the optimal $\lambda$ in the case of Exponential sampling and the optimal $(\gamma, \bar \tau)$ in the case of Beta sampling. We then use $M=4 \times 10^6$ i.i.d. samples to estimate the considered quantities. The results are summarized in the two tables below. The first column of Table \ref{tab:table1} and Table \ref{tab:table2} provides the value of the two parameters $\bar \sigma = \omega$. The second column (resp. third column) of Table \ref{tab:table1} provides the estimated value of the quantity $\E[h(X_T) \I_\seq{\tau >T}]$ with its associated variance, $L^1(\P)$-error and $95\%$-confidence interval in the case of Exponential sampling  (resp. Beta sampling). The second column (resp. third column) of Table \ref{tab:table2} provides the estimated value of the quantity $\partial_x \E[h(X_T) \I_\seq{\tau >T}]$ with its associated variance, $L^1(\P)$-error and $95\%$-confidence interval in the case of Exponential sampling  (resp. Beta sampling).

\begin{table}[h!]
	\begin{center}
		\begin{tabular}{ | l | c | c | }
			\hline
			$\boldsymbol{\bar \sigma}\bold{=}\boldsymbol{\omega}$ & \textbf{Exponential sampling} & \textbf{Beta sampling}\\
			\hline
			0.1 & 2.0; 26.3; 3.2; (+/-) 0.005  &  2.0; 14.9; 2.9; (+/-) 0.004 \\ \hline
			0.2 & 1.99; 213.2; 4.7; (+/-) 0.014 &  1.99; 77.2; 4.5; (+/-) 0.009 \\ \hline
			0.3 & 2.0; 3064.1; 7.8; (+/-) 0.054 &  1.98; 681.2; 7.6 ;(+/-) 0.025 \\ \hline
		\end{tabular}
		\caption{Unbiased Monte Carlo estimation for the quantity $\E[h(X_T) \I_\seq{\tau >T}]$ based on Theorem \ref{th:3.2a} by Exponential and Beta sampling with its associated $95\%$-confidence interval.}
		\label{tab:table1}
	\end{center}
\end{table}


\begin{table}[h!]
	\begin{center}
		\begin{tabular}{ | l | c | c | }
			\hline
			$\boldsymbol{\bar \sigma}\bold{=}\boldsymbol{\omega}$ & \textbf{Exponential sampling} & \textbf{Beta sampling}\\
			\hline
			0.1 &  1.99; 379.7;  8.4 ; (+/-)  0.019 & 2.00; 295.0; 8.4 ;(+/-) 0.017 \\ \hline
			0.2 & 1.98; 1008.7; 7.1; (+/-) 0.035 &1.98; 467.7; 7.1 ;(+/-) 0.021\\ \hline
			0.3 & 1.97; 5411.8; 8.9; (+/-) 0.072 & 1.97; 2358.4;  8.7 ;(+/-) 0.047 \\ \hline
		\end{tabular}
		\caption{Unbiased Monte Carlo estimation for the quantity $\partial_x \E[h(X_T) \I_\seq{\tau >T}]$ based on Theorem \ref{thm:forward:ibp} by Exponential and Beta sampling with its associated variance and $95\%$-confidence interval.}
		\label{tab:table2}
	\end{center}
\end{table}

Most notably, we observe that in both tables the performance of our estimators quickly deteriorates as $\bar \sigma=\omega$ increases. Actually, for large values of $\bar \sigma$, $\omega$, say greater than $0.4$, the variance becomes difficult to estimate from the simulations and the obtained estimates become unreliable. We also see that the Beta time sampling method outperforms the Exponential time sampling for all values of the considered parameters especially for large values of $\bar \sigma$. This behavior was already observed in \cite{APKA} and is reminiscent of unbiased simulation methods for multidimensional diffusion processes. It was thus expected here since there is no hope that our estimator will overcome this problem. To circumvent this issue, one may resort to more sophisticated method such as the second order approximation method developed by \cite{AKY}.

\section{Some Conclusions}
In the present work, we presented a probabilistic representation formula for the marginal law of a killed process based on a basic Markov chain which is obtained using the reflection principle. From this representation, we established two IBP formulae, one being of BEL's type, from which directly stem an unbiased Monte Carlo method. The main element used in this construction is a suitable tailor-made Malliavin calculus for the  underlying Markov chain. For this reason, we do not need to use the full fledge power of Malliavin calculus by closing the derivative operator but just the concepts for simple discrete time Markov chain.

The methodology developed here seems to follow a general pattern that could be used to obtain IBP formulae for some other irregular functionals of the Wiener process for which boundary problems may appear such as the exit time, the local time, the running maximum or the occupation time of a multi-dimensional diffusion process. 
 Although the problem investigated here focuses in the one dimensional case, we believe that the approach developed here also extends to some multi-dimensional cases for which the reflection principle is well understood and densities for basic approximation processes are known, see e.g. \cite{AI}, \cite{Defo} and the references therein. 

On the other hand, it seems difficult at this moment to generalize the methods in \cite{Jenkins} and \cite{Herrmann} to the multi-dimensional case or even to obtain an amenable integration by parts formula based on a Markov chain using these formulations or a Lamperti like transform. We will discuss this extension in future works as well as their implementation for simulation purposes.





\bibliographystyle{abbrv}
\bibliography{biblio}

\section{Appendix}
\label{app:sec}

%
%
%
%
%
%

\subsection{Proof of the probabilistic representation in Theorem \ref{th:3.2a}}
Let $ X $ be the solution to
\eqref{sde:dynamics}. Let $ P $ denote the semigroup operator associated with the killed process. That is, for a measurable and bounded function $ f $, one defines $ P_t f(x)=\E\left[f(X_t)\I_\seq{\tau >t}\right] $. We remark that due to the indicator function, this semigroup is not conservative. 

The heuristic argument in order to obtain the probabilistic representation is to use It\^o's formula on an approximation process obtained from \eqref{sde:dynamics} by removing the drift and freezing the diffusion coefficient at the starting point. From It\^o's formula, one obtains a one step expansion of the law of $X$ around the law of the Markov chain $\bar X$. Then, an IBP formula based on the Markov chain $ \bar{X}$ has to be used to obtain the probabilistic representation of this first step expansion using one jump of the Poisson process $N$. Then, one just needs to iterate the first step expansion in order to obtain the full probabilistic representation.

In order to do this rigorously, one needs to use the regularity properties of the semigroup $ P $ which can be found in \cite{Men1}, {Chapter VI} and/or \cite{Men2}.
In particular, under assumption $ \mathbf{(H)} $ on the coefficients, one obtains that if $ f $ is a smooth function such that $ \lim_{x\downarrow L}f(x)=f(L)=0 $ then $ Pf \in \mathscr{C}^{1,2}((0,T]\times [L,\infty))$ and it satisfies $ \partial_t P_t f={\mathcal{L}} P_t f$ on $ [L,\infty) $, $ t>0 $, where $ {\mathcal{L}} =\frac 12 a\partial_x^2+b\partial_x$ is the infinitesimal generator of $ X $. Moreover, one has $\sup_{0\leq t \leq T} |\partial^{\ell}_x P_t f|_\infty \leq C$, for $\ell=1, 2$, for some positive constant $C:=C(T, a, b)$. 

Under the condition that the test function $ f $ vanishes at $L$, we obtain that $ P $ satisfies its corresponding Dirichlet boundary condition $ P_tf(L)=f(L)=0 $ together with $ P_0f(x)=f (x)$ for all $ x\geq L $.

The argument used to obtain the probabilistic representation starts by applying It\^o's formula to $(P_{T-t} f(\bar{Y}_{t\wedge\bar\tau}))_{t\in [0,T]}$ where the process $\bar Y $ is defined in Lemma \ref{lem:1} on the interval $[0,T]$ and $ \bar{\tau} $ is its associated exit time. We also recall the definition of the approximation process obtained from the reflection principle of Lemma \ref{lem:1}, namely $ \bar{X}^{s, x}_t = \rho x+(1-\rho)(2L-x)+\sigma(x)(W_t-W_s)$ (with the shorten notation $\bar{X}_t= \bar{X}^{0,x}_t$):
\begin{align*}
f(\bar{Y}_T)\I_\seq{\bar{\tau}>T} & \stackrel{\E}{=} P_{T}f(x) +\int_0^T \left(-\partial_uP_{u}f(\bar{Y}_s)\Big|_{u=T-s} +\frac{1}{2}a(x)\partial_x^2P_{T-s} f(\bar{Y}_s)\right)\I_\seq{\bar{\tau}>s}ds\\
& \stackrel{\E}{=} P_{T}f(x) +\int_0^T \left(\frac{1}{2}\left(a(x)-a(\bar{Y}_s)\right)\partial_x^2P_{T-s} f(\bar{Y}_s)-b(\bar{Y}_s)\partial_xP_{T-s}f(\bar{Y}_s)\right)\I_\seq{
	\bar{\tau}>s}ds \\
	& \stackrel{\E}{=} P_{T}f(x) + 2(2\rho-1)\int_0^T \left(\frac{1}{2}\left(a(x)-a(\bar{X}_s)\right)\partial_x^2P_{T-s} f(\bar{X}_s)-b(\bar{X}_s)\partial_xP_{T-s}f(\bar{X}_s)\right)\I_\seq{\bar{X}_s\geq L}ds.
\end{align*}

We now rewrite the previous representation using the Markov chain $(\bar{X}_i)_{0\leq i \leq N_T+1}$ defined by \eqref{eq:MCa} together with the Poisson process $N$. From the previous identity, we get
\begin{align}
P_T f(x) & \stackrel{\E}{=}  f(\bar{X}_T) 2(2\rho-1) \nonumber \\
& + 2(2\rho-1)\int_0^T \left(\frac{1}{2}\left(a(\bar{X}_s) -a(x)\right)\partial_x^2P_{T-s} f(\bar{X}_s) +b(\bar{X}_s)\partial_xP_{T-s}f(\bar{X}_s)\right)\I_\seq{\bar{X}_s\geq L}ds \label{integral:time:first:step} \\
& \stackrel{\E}{=} f(\bar{X}_{N_T+1}) \theta_{N_T+1} \I_\seq{N_T=0} \nonumber \\
&  + e^{\lambda T} 2 \lambda^{-1} (2\rho_{N_T}-1)\left\{\frac12(a(\bar{X}_1) - a(x)) \partial^2_x P_{T- \zeta_1}f(\bar{X}_1) + b(\bar{X}_1) \partial_x P_{T-\zeta_1}f(\bar{X}_1) \right\} \I_\seq{\bar{X}_1 \geq L} \I_\seq{N_T=1} \nonumber
\end{align}


Next, we apply the IBP formula \eqref{eq:IBP} with respect to the r.v. $ \bar{X}_1$ in the
above expression. The formula is applied once to the terms associated with the drift coefficient $ b $ and two times with respect to the terms related to the diffusion coefficient $ a$. In order to do that one first has to take the conditional expectation $\E_{0,1}[.]$ in the second term of the above equality. As these IBPs involve the indicator function $ \I_\seq{\bar{X}_1\geq L} $, the correct calculation has to be done using the theory in \cite{IW}, Chapter V.9.\footnote{An alternative but longer approach can also be achieved using smooth approximations for the indicator function.} This yields
\begin{align*}
P_{T}f(x)
\stackrel{\E}{=}& f(\bar{X}_{N_T+1}) \theta_{N_T+1} \I_\seq{N_T=0}
+ e^{\lambda T} 2\lambda^{-1}(2\rho_{N_T}-1)
\\&\times \left(\frac{1}{2}{{\mathcal{I}}}_{1}\left(\left(a(\bar{X}_1) - a(x)\right)
\I_\seq{ \bar{X}_1\geq L}\right)\partial_xP_{T-\zeta_1}f(\bar{X}_1) + {{\mathcal{I}}}_{1}\left(b(\bar{X}_{1})\I_\seq{ \bar{X}_1\geq L}\right) P_{T-\zeta_1}f(\bar{X}_{1})\right) \I_\seq{N_T= 1} \\
	 \stackrel{\E}{=}& f(\bar{X}_{N_T+1}) \theta_{N_T+1} \I_\seq{N_T=0}+ e^{\lambda T} 2\lambda^{-1}(2\rho_{N_T}-1)\\
	& \times \left(\frac{1}{2}{{\mathcal{I}}}_{1}\left(\left(a(\bar{X}_1) - a(x)\right)
\right)\partial_xP_{T-\zeta_1}f(\bar{X}_1) + {{\mathcal{I}}}_{1}\left(b(\bar{X}_{1} ) \right) P_{T-\zeta_1}f(\bar{X}_{1})\right) \I_\seq{ \bar{X}_1\geq L} \I_\seq{N_T= 1}
\end{align*}

\noindent where we used the extraction formula \eqref{eq:exta} applied to the r.v. $ \I_\seq{\bar{X}_s\geq L} $ and Lemma \ref{lem:2.1} (taking first the conditional expectation w.r.t $\zeta_1$) for $ \ell=1,\ k=0 $ for the last equality. 

Now we carry out again the same procedure for the second term in the above integrand using again the extraction formula  \eqref{eq:exta} and the fact that $ P_{t}f(L)=f(L)=0 $. We obtain
\begin{align}
    & P_{T-t}f(\bar{Y}_t)\I_\seq{\bar{\tau}>t}\nonumber \\
	\stackrel{\E}{=}& f(\bar{X}_{N_T+1}) \theta_{N_T+1} \I_\seq{N_T=0} + e^{\lambda T} 2\lambda^{-1}(2\rho_{N_T}-1)\nonumber\\
	& \times \left(\frac{1}{2}\mathcal{I}^2_{1} \left(a(\bar{X}_1) - a(x)\right) P_{T-\zeta_1}f(\bar{X}_1) + {{\mathcal{I}}}_{1}\left(b(\bar{X}_{1} ) \right) P_{T-\zeta_1}f(\bar{X}_{1})\right) \I_\seq{ \bar{X}_1\geq L} \I_\seq{N_T= 1} \label{prob:representation:step1}
\end{align}

We emphasize that, under assumption \A{H}, the following estimates hold: for all $p\geq1$, there exists $C:=C(a, b, T, p)$ such that
\begin{align}
\label{eq:est2}
\|\I_\seq{\bar{X}_1\geq L}\mathcal{I}_{1}(a(\bar{X}_1)-a(x))\|_{p, 0, 1} & \leq C,\\
\|\I_\seq{\bar{X}_1\geq L}\mathcal{I}^2_{1}(a(\bar{X}_1)-a(x))\|_{p, 0, 1} +\|\I_\seq{\bar{X}_1\geq L}\mathcal{I}_{1}(b(\bar{X}_{1}))\|_{p, 0, 1}& \leq \zeta_1^{-1/2} \nonumber
\end{align}

\noindent which in turn, by using the fact that on $\P(N_T=1, \zeta_1 \in dt) = \lambda e^{-\lambda T} dt$ on $[0,T]$, lead to the integrability of the second term appearing in \eqref{prob:representation:step1}. 

We now prove the three estimates \eqref{eq:est2}. Since they are obtained using the same technique, we only briefly explain how to obtain the first inequality. Using the definition of $ \mathcal{I}_{1}$, one has that 
\begin{align*}
\mathcal{I}_{1}(a(\bar{X}_1) -a (x))= (a(\bar{X}_1) -a(x)) \frac{Z_1}{\sigma(x)\zeta_1} + a'(\bar{X}_1).
\end{align*}
Therefore in order to bound the above expression, we need to find an upper bound for $ |a(\bar{X}_1)-a(x)| $ and then 
use classical estimates for the moments of Gaussian r.v.'s. In order to bound $ |a(\bar{X}_1)-a(x)| $, one uses the Lipschitz property of $ a $ as follows
\begin{align*}
\I_\seq{\bar{X}_1\geq L}|a(\bar{X}_1) -a(x)|\leq C\I_\seq{\bar{X}_1\geq L} |\bar{X}_1-x|\leq C|\bar{X}_1-x\rho_1-(1-\rho_1)(2L-x)|\leq C|Z_1|.
\end{align*}
Here we have used the fact that $ \bar{X}_1\geq L $ and $ x\geq L $ implies $|\bar{X}_1-x|\leq  |\bar{X}_1-x\rho-(1-\rho)(2L-x)|$. Finally the estimate \eqref{eq:est2} follows from the application of Lemma \ref{lem:6} with $H_1= a(\bar{X}_1) - a(x)$.

With the notations introduced in Section \ref{sec:3a}, \eqref{set:Din:X} and \eqref{eq:13aa}, the identity \eqref{prob:representation:step1} can be rewritten as
 \begin{align}
 P_Tf(x) & = \E\left[ f(\bar{X}_{N_T+1}) \I_{D_{N_T+1,N_T}} \bar{\theta}_{N_T+1} \I_\seq{N_T=0}\right]+ e^{\lambda T} \E\left[P_{T-\zeta_1}f(\bar{X}_{1}) \I_{D_{N_T,N_T}} \bar{\theta}_1\I_\seq{N_T=1}\right]. \label{first:step:probabilistic:representation}
 \end{align}

Using again \eqref{integral:time:first:step}, applying Lemma \ref{lem:cj} and by finally performing IBPs as before, we obtain by induction the following formula, $\forall n \in \N, \quad$ .
\begin{align}
P_Tf(x)= \sum_{j=0}^{n-1} \E\left[ f(\bar{X}_{N_T+1})\prod_{i=1}^{N_T+1}\I_{D_{i, N_T}} \bar {\theta}_i  \I_\seq{N_T=j}\right ] + e^{\lambda T} \E\left[ P_{T-\zeta_{n}}f(\bar{X}_{n})\prod_{i=1}^{n} \I_{D_{i, N_T}}\bar {\theta}_i \I_\seq{N_T=n}\right ].\label{induction:step:probabilistic:representation}
 \end{align} 
 To conclude it remains to prove the absolute convergence of the first sum and the convergence to zero of the last term\footnote{An analytical proof of this fact can be found in \cite{FKL1}, Lemma 5.2.}. These two results follow directly from the boundedness of $f$ and the following general estimates on the product of weights.

\begin{lem} 
	\label{lem:7}
	The r.v.'s $\bar{\theta}_i \in \mathbb{S}_{i,n}$, $ i\in\mathbb{N}_{n+1} $ satisfy the time degeneracy estimates in Definition \ref{def:td}. More precisely, for any $p\geq1$, there exists a (deterministic) constant $ C>0 $ (which may depend on $ \lambda $) such that for $ i\in\mathbb{N}_n $
	\begin{align}
	\label{eq:14}
	 \I_{D_{i-1,n}}\E_{i-1,n}\left[\I_{D_{i,n}}|\bar{\theta}_i |^p\right]\leq C(\zeta_{i}-\zeta_{i-1})^{-\frac{p}{2}}.
	\end{align}	
	Furthermore, one has $ \I_{D_{n-1,n}}\E_{n,n}\left[\I_{D_{n,n}}|\bar{\theta}_{n+1}|^p\right ]\leq C $.
\end{lem}
\begin{proof}
	The proof for $ i\leq N_T $ follows from Lemma \ref{lem:6} after noting that $c^{i}_2 = \frac12 (a(\bar{X}_i) - a(\bar{X}_{i-1})) \in \mathbb{M}_i(\bar{X},1/2) \cap \mathbb{S}_i(\bar{X})$ and is based on the same arguments as in the proof of \eqref{eq:est2}. In the case $ i=N_T+1 $ the conclusion follows directly from the boundedness of the r.v. $\bar\theta_{N_T+1}$. This argument will be used repeatedly in order to obtain the so-called time degeneracy estimates.	
\end{proof}

\begin{lem}\label{lem:4.1} Assume that the weights $ (\bar{\theta}_i,\mathbb{S}_{i,n}) $, $ i\in\mathbb{N}_{n+1} $ satisfy the time degeneracy estimates in Definition \ref{def:td}. Then for any $ p\in [0,2)  $, we have the following moment estimate:
	\begin{align*}
	\E\left[\Big|\prod_{i=1}^{N_T+1}\I_{D_{i,N_T}}{\bar\theta}_i \Big|^p\right]\leq
	E_{1-\frac{p}2,1}(CT^{-\frac{p}2+1})<\infty.
	\end{align*} 
	Here $ E_{1-\frac{p}2,1} $ stands for the Mittag-Leffler function $E_{\alpha, \beta}(z) := \sum_{n\geq 0} \frac{z^k}{\Gamma(\beta+ k\alpha)}$ with parameters $ \alpha= 1-\frac{p}2, \, \beta=1$. 
\end{lem}
\begin{proof}
	For the proof, it is enough to use the Markov property of the Markov chain $ \bar{X} $ together with the time degeneracy estimates and the fact that given $ N_T=n $, the jump times of the Poisson process are distributed as the order statistics of $n$ i.i.d. uniform $ [0,T]$-valued r.v.'s satisfying $\P(N_T=n, \zeta_1\in dt_1, \cdots, \zeta_n \in dt_n)= \lambda^n e^{-\lambda T} dt_1, \cdots, dt_n$ on the set $A_n=\left\{ (t_1, \cdots, t_n) \in [0,T]^n: 0 < t_1 < \cdots <t_n <T \right\}$. This gives
\begin{align*}
\E\left[\Big|\prod_{i=1}^{N_T+1}\I_{D_{i,N_T}}\bar{\theta}_i\, \Big|^p \I_\seq{N_T=n}\right] &  \leq C^{n+1} \int_{A_n} \prod_{i=1}^{n} (t_{i}-t_{i-1})^{-\frac{p}{2}} \, dt_1 \cdots dt_n \\
& = C^{n+1} T^{n (1-\frac{p}{2})} \frac{\Gamma^{n}(1-\frac{p}{2})}{\Gamma(1+ n(1-\frac{p}{2}))}
\end{align*}

\noindent for some positive constant $C:=C(T, a, b, p, \lambda)$. One concludes the proof by adding the previous inequality from $n=0$ to infinity. 

\end{proof}

 As a consequence, \eqref{eq:PR} holds for bounded smooth functions. The extension to bounded continuous functions follows from an approximation argument while the extension to bounded measurable function follows from a monotone class argument. The extension to polynomially growing functions is performed by a limit argument noting that all moments of $ \bar{X}_{N_T+1}$ and $ X_T $ are bounded.
 
\begin{remark}
(i) From the above proof, it should be clear that inequality \eqref{eq:14} (or equivalently \eqref{eq:td}) is strongly tied with the restriction $ p\in [0,2) $ in Lemma \ref{lem:7}. This is the reason why it is important to always have weights which satisfy a time degeneracy estimate as stated in Definition \ref{def:td}.\hfill\break 
(ii) On the other hand, the estimate in \eqref{eq:14}, as remarked in \cite{BK}, also suggests that the variance of the Monte Carlo estimator may be infinite. In order to achieve finite variance, one can resort to an importance sampling technique on the jump times as proposed by Anderson and Kohatsu-Higa \cite{APKA}.
\end{remark}

\subsection{Proof of the transfer of derivatives of Lemma \ref{lem:5.1}}
\label{app:tl}

	First, we remark that from
	the explanation given after the definition of the operators $ \mathcal{D} $ and $ \mathcal{I} $ in \eqref{eq:I1}, the r.v.'s $(\ltheta^e_{i+1},\ltheta^c_{i+1},\ltheta^\partial_{i+1})$ defined by \eqref{eq:defDc} are elements of ${\mathbb{S}}_{i+1}(\bar{X})$. Then, it is clear that since the coefficients $b$ and $\sigma$ are smooth as stated in $ \mathbf{(H)} $ and the transition of the Markov chain $\bar{X}$ has a smooth Gaussian transition density, the map $x \mapsto \mathbb{E}_{i,n}[f(\bar X_{i+1}) \I_{D_{i+1,n}} \ltheta^e_{i+1}]\big|_{\bar{X}_i=x} $ is in $\mathscr{C}^{1}_p(\mathbb{R})$ a.s. 
	
	 In order to prove \eqref{eq:5.1a}, we consider the difference between the term appearing on the left-hand side and the first term appearing on the right-hand side of equation \eqref{eq:5.1a}. Using the integration by parts formula \eqref{eq:IBP} and then the extraction formula \eqref{eq:exta}, we get 
	$$
	\mathbb{E}_{i,n}[\partial_{\bar X_{i+1}}f(\bar X_{i+1}) \I_{D_{i+1,n}} \bar{\theta}_{i+1}]   = \E_{i,n}[f(\bar X_{i+1}) \I_{D_{i+1,n}} {\mathcal{I}}_{i+1}(\bar{\theta}_{i+1})] - \E_{i,n}[f(\bar{X}_{i+1}) \delta_L(\bar{X}_{i+1}) \bar{\theta}_{i+1}].
	$$
	From \eqref{eq:flow} and the integration by parts formula \eqref{eq:IBP}, we also obtain 
	\begin{align*}
	\partial_{\bar{X}_i}\E_{i,n}[f(\bar X_{i+1}) \I_{D_{i+1,n}}\ltheta^e_{i+1}]  = & \E_{i,n}[\partial_{\bar{X}_i} (f(\bar X_{i+1}) \I_{D_{i+1,n}}) \ltheta^e_{i+1}] + \E_{i,n}[f(\bar X_{i+1}) \I_{D_{i+1,n}}\partial_{\bar{X}_i} \ltheta^e_{i+1}] ,\\
	= &\E_{i,n}[f(\bar{X}_{i+1}) \I_{D_{i+1,n}} {\mathcal{I}}_{i+1}([(2\rho_{i+1}-1) + \sigma'_i Z_{i+1}] \ltheta^{e}_{i+1}) ] \\
	&+ \E_{i,n}[f(\bar X_{i+1}) \I_{D_{i+1,n}}\partial_{\bar{X}_i} \ltheta^e_{i+1}].
	\end{align*}
	By combining the two previous computations, we see that the difference
	$ \mathbb{E}_{i,n}[\partial_{\bar X_{i+1}}f(\bar X_{i+1})
	\I_{D_{i+1,n}}
	\bar{\theta}_{i+1}]   - \partial_{\bar{X}_i}\mathbb{E}_{i,n}[f(\bar X_{i+1}) 
	\I_{D_{i+1,n}}
	\ltheta^e_{i+1}]  $ can be expressed as
	\begin{align}
	\nonumber
	&	\mathbb{E}_{i,n}[f(\bar X_{i+1})\I_{D_{i+1,n}} {{\mathcal{I}}}_{i+1}(\bar \theta_{i+1}-(2\rho_{i+1}-1)\ltheta^e_{i+1})]   -\mathbb{E}_{i,n}[f(\bar X_{i+1}) \I_{D_{i+1,n}}\partial_{\bar{X}_i}\ltheta^e_{i+1}]\\
	&
	-
	\mathbb{E}_{i,n}[f(\bar X_{i+1}) \I_{D_{i+1,n}}{{\mathcal{I}}}_{i+1}(
	\sigma'_iZ_{i+1}
	\ltheta^e_{i+1}
	)]
	-\mathbb{E}_{i,n}[f(\bar X_{i+1})\delta_L(\bar X_{i+1}) 
	\bar{\theta}_{i+1}] 		
	.\label{eq:firb}
	\end{align}
	
	Note that by using the relation $ {d_1^{i+1}}= c^{i+1}_1-(2\rho_{i+1}-1)\partial_{\bar{X}_i}c^{i+1}_2$ in \eqref{d1} and Lemma \ref{chain:rule} we obtain
	\begin{align*}
	\partial_{\bar{X}_i}\ltheta^e_{i+1}{=}&
	2\lambda^{-1}\left ({\mathcal{I}}_{i+1}^2(\partial_{\bar{X}_i}c^{i+1}_2)+{\mathcal{I}}_{i+1}(\partial_{\bar{X}_i}{d_1^{i+1}})
	- (\sigma' \sigma^{-1})_i\left({\mathcal{I}}_{i+1}({d_1^{i+1}})
	+2{\mathcal{I}}_{i+1}^2(c_2^{i+1})\right)\right ),\\
	\partial_{\bar{X}_i}c^{i+1}_2=&a'_{i+1}\partial_{\bar{X}_i}\bar{X}_{i+1}-a'_i,\\
	\partial_{\bar{X}_i}^2c_2^{i+1}=&a''_{i+1}(\partial_{\bar{X}_i}\bar{X}_{i+1})^2+a'_{i+1}\partial^2_{\bar{X}_i}\bar{X}_{i+1}-a''_i,\\
	\partial_{\bar{X}_i}{d_1^{i+1}}=&b'_{i+1}\partial_{\bar{X}_i}\bar{X}_{i+1}-(2\rho_{i+1}-1)\partial_{\bar{X}_i}^2c_2^{i+1},
	\end{align*}
	
	\noindent which in turn, after some algebraic simplifications, yields the formula for $\ltheta^{c}_{i+1}$ appearing in \eqref{eq:defDc}. In order to conclude the proof of \eqref{eq:5.1a}, it remains to notice that for  $ f\in \mathscr{C}_p^1([L,\infty))$, one has
	\begin{align}
	-\E_{i,n}\left[f(\bar{X}_{i+1})\delta_L(\bar{X}_{i+1})\bar{\theta}_{i+1}\right]=&
	\E_{i,n}\left[f(\bar{X}_{i+1})\delta_L(\bar{X}_{i+1})	\ltheta^\partial_{i+1}\right]
	\label{eq:coeff}
	\end{align}
	which is a direct consequence of Corollary \ref{cor:2.1}. 
	
	The time degeneracy estimates of the r.v.'s $(\ltheta^e_{i+1},\ltheta^c_{i+1},\ltheta^\partial_{i+1}) $ are straightforward applications of Lemma \ref{lem:6} using the same arguments as in \eqref{eq:est2}. The proof of the transfer of derivative formula \eqref{eq:sdf} as well as the time degeneracy estimates of the r.v.'s for $\ltheta^{a}_{n+1}$, $a \in \left\{ e, c\right\}$ on the set $ \{N_T=n\} $ follow from similar arguments and we omit the remaining technical details. \qed

\subsection{Proof of the boundary merging lemmas}

\subsubsection{A time convolution result for  Gaussians}

In this section, we describe a series of explicit calculations of convolutions of Gaussian densities with respect to the time variable. These are used when merging weights and building the boundary process. 
\begin{lem}	\label{lem:7.1} 
	For 
 $x,y,\alpha,\beta\in\mathbb{R}_+$ and $\ell\in\{0,1,2\}$, 
	\begin{align}\label{eq:4.6}
	\int_{0}^{t}
	\partial_{x}g(\alpha^2s,x)
	\partial_{y}^\ell{( g(\beta^2(t-s),y)) } 
	ds
	=-\alpha^{-2}\partial_{y}^{\ell}(g(\beta^2t,
	y+\frac{\beta}{\alpha} x))
	.
	\end{align}
	Furthermore
	\begin{align}\label{eq:4.6a}
	\int_{0}^{t}
	s\partial_x g(\alpha^2s,x)
	\partial_{y}^\ell{(g(\beta^2(t-s),y)) } 
	ds
	=-\alpha^{-3}\beta^{-1}x\partial_{y}^{\ell}(\bar{\Phi}(\beta^2t,
	y+\frac{\beta}{\alpha} x))
	.
	\end{align}
where $\bar \Phi(t,z) := \int^\infty_{|z|} g(t,y)dy$.
\end{lem}
\begin{proof}
	
	The results follows from simple manipulations using the following
	 Laplace transform for $ x, \, \eta>0 $:
	\begin{align}
	\label{eq:LTg}\mathcal{L}(g_{-1})(\eta):= \int_0^{\infty} e^{-\eta s} g_{-1}(s) \, ds =&x^{-1}e^{-{\sqrt{2\eta x^2}}} \mbox{ with } g_{-1}(s):=s^{-1}g(s,x).
	\end{align}
	
	 In particular, one uses the fact that $s\partial_x g(a^2s,x) = -\frac{x}{a^2}g(a^2s,x)$ in proving \eqref{eq:4.6a}.
\end{proof}


\subsubsection{Proof of Lemma \ref{lem:4.3}}\label{proof:lemma:merging}
The key ingredient to prove \eqref{eq:faim}, that is, the boundary merging of $ \ltheta^e_{i+1} $ and  $ \ltheta^\partial_{i} $, is  Lemma \ref{lem:7.1}. We divide the proof into several steps.

{\it Step 1: Simplify the statement.}
First, one observes that it is easier to verify \eqref{eq:faim} for each term appearing in $ \ltheta^e_{i+1} $ and $ \ltheta^\partial_i $ in \eqref{eq:defDc}, after proper extraction and localization as in Corollary \ref{cor:2.1}, and then add them together. Therefore proving the following general statement will suffice: For any $ j=0,1,2 $ and real-valued $ f\in \mathscr{C}^0_p(\mathbb{R}^2)$,
\begin{align}
\label{eq:opmerg} 
&\E\big[f(\bar{X}_i,\bar{X}_{{i+1}})\I_{D_{i+1,n}} {\mathcal{I}}_{{{i+1}}}^j(1)\delta_L(\bar{X}_{{i}})(2\rho_i-1)
{\mathcal{I}}_i(1)
\,\big|\,\mathcal{G}_{{i-1}}, \zeta_{i+1},N_T= n\big]\\
 & =\frac{\lambda^{-1}}{\zeta_{i+1}-\zeta_{i-1}}
\E\big[f(L,\bar{X}^{\partial}_{i-1,i+1})
\I_{D^\partial_{i-1,i+1,n}}
a_{i-1}^{-1}
\bar{{\mathcal{I}}}_{i-1,i+1}^{j}(1)
\,\big|\,\mathcal{G}_{{i-1}},\zeta_{i+1},N_T= n\big].\nonumber
\end{align}

%

{\it Step 2: Proof of \eqref{eq:opmerg}. Case $ j=0 $.}
In order to foster the understanding of the proof, let us first consider the case $ j=0 $. The following arguments also gives rise to the definition of the merged boundary process.  

First, observe that as $(2\rho_i-1) \sigma_{i-1} Z_i = L-\bar{X}_{i-1}:=Y_{i-1} $ on the set $\left\{\bar{X}_{i}=L\right\}$, one has 
\begin{align*}
(2\rho_i -1)\delta_{L}(\bar{X}_{i}){\mathcal{I}}_i(1) 
& = \delta_{L}(\bar{X}_{i}) \frac{Y_{i-1}}{a_{i-1} (\zeta_{i}-\zeta_{i-1})}.
\end{align*}
Next, we rewrite the right hand side of \eqref{eq:opmerg}, namely, using the independence of $W$ and $N$, we have
\begin{align*}
&\E\big[f(\bar{X}_i,
\bar{X}_{i+1})
\I_{D_{i+1,n}}
\delta_L(\bar{X}_i)
(2\rho_i-1)
{\mathcal{I}}_i(1)
\,\big\vert\,\mathcal{G}_{i-1},\zeta_{i+1}, \zeta_i,\rho_i,N_T =  n\big]\\
& \quad = \int_L^{\infty}  f(L,z)\frac{Y_{i-1}}{a_{i-1}(\zeta_i-\zeta_{i-1}) }g(a_{i-1}(\zeta_i-\zeta_{i-1}),Y_{i-1}) \, g(a(L)(\zeta_{i+1}-\zeta_{i}),{z-L}) \, dz.
\end{align*}

We now use the following key property: {for a real-valued bounded $\mathcal{G}_{i-1}\times \mathcal{B}(\mathbb{R}^2_+)$-measurable function $h:\Omega\times \rr^2_+\rightarrow \mathbb{R}$,} one has  
\begin{gather}
\E[h(\zeta_{i} , \zeta_{i+1})\,\vert\, \mathcal{G}_{i-1}, \zeta_{i+1},N_T= n] = \E[h(s+U, t)\vert\, \mathcal{G}_{i-1}, \zeta_{i+1},N_T= n]\Big|_{s= \zeta_{i-1},\, t= \zeta_{i+1}}.\label{eq:U}
\end{gather}
\noindent where $U \sim \mathcal{U}(0,t-s)$ is independent of $\mathcal{G}_{i-1}$ and $N$. More precisely, by using the tower property for conditional expectation and \eqref{eq:U}, we obtain
\begin{align*}
&\E\big[ f(\bar{X}_i,
\bar{X}_{i+1})
\I_{D_{i+1,n}}\delta_L(
\bar{X}_i)
(2\rho_i-1)
{\mathcal{I}}_i(1)
\,\big\vert \,\mathcal{G}_{i-1},\zeta_{i+1},\rho_i,N_T = n\big]
\\
& =   \frac{1}{\zeta_{i+1}-\zeta_{i-1}}\int_L^{\infty} dz\int_0^{\zeta_{i+1}-\zeta_{i-1}}f(L,z)\frac{Y_{i-1}}{a_{i-1}s}g(a_{i-1}s,Y_{i-1})g(a(L)((\zeta_{i+1}- \zeta_{i-1})-s),{z-L}) \, ds.
\end{align*}

In order to compute the above time convolution and show that the above conditional expectation can be rewritten using the boundary process $ \bar{X}^\partial $, we apply Lemma \ref{lem:7.1}, in particular \eqref{eq:4.60}, noticing that $\frac{Y_{i-1}}{a_{i-1}s}g(a_{i-1}s,Y_{i-1}) =  \partial_2 g(a_{i-1}s,  |Y_{i-1}|)$, to obtain
\begin{align*}
& \E\big[f(\bar{X}_i,\bar{X}_{{i+1}})\I_{D_{i+1,n}} \delta_L(\bar{X}_{{i}})
(2\rho_i-1)
{\mathcal{I}}_i(1)
\,\big\vert \,\mathcal{G}_{{i-1}}, \zeta_{i+1},\rho_i,N_T = n\big] \\
& =	 \frac{1}{ a_{i-1}(\zeta_{i+1} - \zeta_{i-1})} \int_L^{\infty} f(L,z) g(a(L)(\zeta_{i+1} - \zeta_{i-1}),{z-(L(1-\mu(\bar{X}_{{i-1}})) + \bar{X}_{{i-1}} \mu(
	\bar{X}_{{i-1}}))}) \,dz,
\\
& =   \frac{1 }{ a_{i-1}(\zeta_{i+1} - \zeta_{i-1})} \E\big[f(L,\bar{X}^\partial_{i-1,i+1})
\I_{D^\partial_{i-1,i+1,n}}
 \,\big\vert \,\mathcal{G}_{i-1},\zeta_{i+1},N_T =  n\big].
\end{align*}
This finishes the proof for the case $ j=0$.

{\it Step 3: The general case: $ j=1,2 $.} Using equality \eqref{eq:4.6}, one can similarly prove \eqref{eq:opmerg} in the case  $j = 1,2$. In fact, a useful formula to prove this general case directly from \eqref{eq:4.6} is the following property for $ i\in\mathbb{N}_n $ which was stated in \eqref{eq:link:integral:hermite:pol}.

\begin{align*}
&{{\mathcal{I}}}_{i+1}^j(1)=(-1)^j
\left(g^{-1}\partial_y^jg\right )
(a_i(\zeta_{i+1} - \zeta_{i}),y-\rho_{i+1} \bar{X}_i-(1-\rho_{i+1})(2L-\bar{X}_i))
\Big |_{y=\bar{X}_{i+1}}.
\end{align*}

{\it Step 4: The general case follows by linearity.}
Note that in the general case of \eqref{eq:faim}, we have that $ \ltheta_{i+1}^e $ is given by \eqref{eq:defDc}. This expression can be rewritten using the extraction formula \eqref{eq:exta} as linear combinations of terms of the type $ f_j(\bar{X}_i,\bar{X}_{i+1})\mathcal{I}^j_{i+1}(1) $ for some particular functions $ f_j\in\mathscr{C}(\mathbb{R}^2) $. Therefore the result in \eqref{eq:opmerg} applies. 

One uses again \eqref{eq:exta} for $ \bar{\mathcal{I}} $ in order to rewrite the resulting expressions in a compact form. For this, note that we have conveniently defined $ \bar{\mathcal{D}} $ as the adjoint of $ \bar{\mathcal{I}} $ for which the extraction formulae are satisfied.

Therefore towards obtaining the formulae which appear in the definition of the weights $ \ltheta^{\partial*e} $ one combines linearly \eqref{eq:opmerg} to obtain: 
:
\begin{align}
\label{eq:c3}
&\E\big[f(\bar{X}_{{i+1}})\I_{D_{i+1,n}}\ltheta^e_{{i+1}}\delta_L(\bar{X}_{i})
(2\rho_i-1)
{\mathcal{I}}_i(1)
\big\vert \,\mathcal{G}_{{i-1}}, \zeta_{i+1},N_T =  n \big] = \\
& 2\lambda^{-1}
\E\left[\left.f(\bar{X}^{\partial}_{i-1,{i+1}})
\I_{D^\partial_{i-1,i+1,n}}
\left\{ \frac{1}{ a_{i-1}(\zeta_{i+1} - \zeta_{i-1})} \left(\bar{{\mathcal{I}}}^2_{i-1,i+1 }(\bar{d}^{i+1}_2) +\bar{{\mathcal{I}}}_{i-1,i+1 }
(\bar  c^{i+1}_3)\right)\right\}\right|\mathcal{G}_{{i-1}}, \zeta_{i+1},N_T = n\right].
\nonumber
\end{align}
Here, $\bar c^{i+1}_3$ takes the form
\begin{align*}
\bar  c^{i+1}_3:= b(\bar{X}^\partial_{i-1,i+1})-(2\rho_{i+1}-1)
(a'(\bar{X}^\partial_{i-1,i+1})((2\rho_{i+1}-1)+\sigma'(\bar{X}^\partial_{i-1,i+1})(Z_{i+1}+\Delta Z_i))-a'(L))
\end{align*}
and belongs to $\mathbb{S}_{i-1,i+1}(\bar{X}^\partial)$.

To continue simplifying the above expression and to obtain $\bar d^{i+1}_1$ from $\bar c^{i+1}_3$, note that except for the factor $ (2\rho_{i+1}-1) $ in the coefficient $ \bar{c}_3^{i+1}$ all expressions on the right hand side of the above equality do not depend on $ \rho_{i+1} $. In particular, note that the r.v. $ \bar{X}^\partial_{i-1,i+1} $ does not depend on $ \rho_{i+1} $. Therefore by the symmetry of the Bernoulli r.v. $ 2\rho_{i+1}-1 $, the term $(a'\sigma')(\bar X^\partial_{i-1,i+1})(Z_{i+1} + Z_i) - a'(L)$ in $\bar c_3^{i+1}$ vanishes, leading to the definition of $ \bar{d}_1^{i+1}$.  From here, 
\eqref{eq:faim} follows by multiplying \eqref{eq:c3} by  $ 2\lambda^{-1}(a'(L)-b(L)) $ in order to obtain $ \ltheta^{\partial} $ on the left side of the equation.

 Finally, the equality \eqref{eq:faim:last:interval1} follows by arguments similar to those employed in \eqref{eq:opmerg} for $ j=0 $ and we omit its proof.

{\it Step 5: The time degeneracy estimates.} Using the basic properties for Skorokhod integral, obtained mutatis mutandis from Lemma \ref{lem:6}, one obtains the time degeneracy inequalities as it was done in Lemma \ref{lem:7} using the estimate \eqref{eq:est2}. On the set $\left\{\bar{X}_{i-1}\geq L\right\}$, one obtains that $|\bar{d}_2^{i+1}(L,\bar{X}_{i-1,i+1})|\leq C|Z_{i+1}+Z_{i}|$ so that
\begin{align*}
\I_{D_{i-1,n}}\|\I_{D^\partial_{i-1,i+1,n}}\ltheta^{\partial *j}_{i-1,i+1}\|_{p, i-1}\leq& C(\zeta_{i+1}-\zeta_{i-1})^{-1/2}.
\end{align*} 
This ends the proof of Lemma \ref{lem:4.3}. \qed

One may heuristically understand why the above merged structure appears. Actually, the Dirac delta function $ \delta_L(\bar{X}_i) $ appearing inside the conditional expectation \eqref{eq:opmerg} imposes the approximation process to go from $\bar{X}_{i-1}  $ to $ L $ and then from $ L $ to $ \bar{X}_i $. The two corresponding Gaussian laws on each interval are convolved in time in the conditional expectation through the random time $\zeta_i$. Therefore, using the branching property of Gaussian kernels, stated in Lemma \ref{lem:7.1} in its analytic form\footnote{Clearly, this property is connected with a similar one for Bessel processes as stated in Chapter XI of \cite{revuz}.}, one obtains the results after renormalization of the variances. This is the argument used in the above proof.

\subsubsection{Proof of Lemma \ref{lem:4.3a}}\label{proof:lemma:merging:ibp}
As this proof has many intersections with the proof of Lemma \ref{lem:4.3}, we only indicate the main points.
	Essentially we need to perform two steps. First, one {simplifies (using the reduction formula of Corollary \ref{cor:2.1}) the expression of} $\bar{\theta}_i  $ using the fact that one has $ \delta_L(\bar{X}_i) $ in the expression \eqref{eq:faima} as was done in the proof of \eqref{eq:coeff}. Then one follows the lines of reasoning as in the proof of Lemma \ref{lem:4.3}, except that one needs to use \eqref{eq:4.6a} instead of \eqref{eq:4.6}.	
	
	We remark here that the time degeneration estimate in Lemma \ref{lem:4.3a} improves in comparison with its corresponding estimate in Lemma \ref{lem:4.3} because of the inequalities 
	\begin{align*}
	|\bar{X}_{i-1}-L|\vee |\bar{X}^\partial_{i-1,i+1}-L|\leq& {C}|\bar{X}^\partial_{i-1,i+1}-L(1-\mu(\bar{X}_{i-1}))-\bar{X}_{i-1}\mu(\bar{X}_{i-1})|\leq C|Z_{i+1}+Z_{i}|,\\
	(\bar{\Phi}g^{-1})(a(L)(\zeta_{i+1}-\zeta_{i-1}),Z_{i+1}+Z_{i})\leq& C{ (\zeta_{i+1}-\zeta_{i-1})^{1/2}  },
	\end{align*}	
\noindent where the last inequality follows from { change of variables and} Komatsu's inequality and in the first we assume that $\bar{X}_{i-1},\bar{X}^\partial_{i-1,i+1}\geq L  $. The proofs corresponding to the last interval are easier and follow similar arguments. This concludes the proof of Lemma \ref{lem:4.3a}.\qed

\subsection{Localization and reduction lemmas}
\label{app:ext}
%

The following lemma is the basic result which explains that the expectation of a weight can be simplified thanks to the symmetry of the law of $\bar{X}_i$, for $1\leq i \leq N_T$. For this reason, we call the following result a reduction lemma. This section only uses the results and setting of Section \ref{sec:3a}.
\begin{lem}
	\label{lem:2.1}
	Let $ f \in \mathscr{C}^1_p(\mathbb{R})$ and $(\ell, k, n ) \in \bar{\N} \times \N^2 $. Then, for any $i \in \N_n$, one has 
	\begin{align*}
	& \E\left[f(\bar{X}_{i})\delta_L(\bar{X}_{i})(2\rho_i-1)^\ell{{\mathcal{I}}}_{i}^{k}(1) | \mathcal{G}_{i-1}, T^{n+1}, N_T=n\right] \\
	&\quad  =
	\begin{cases}
	0,\ \text{ if } \ell +k\text{ is odd,}\\
	\E\left[f(\bar{X}_i)\delta_L(\bar{X}_i)(2\rho_i-1)^{k}{{\mathcal{I}}}_{i}^{k}(1)\,|\, \mathcal{G}_{i-1}, T^{n+1}, N_T=n \right],
	\ \text{if } \ell+k\text{ is even.}
	\end{cases}
	\end{align*}
\end{lem}
\begin{proof}
The proof follows by using \eqref{eq:link:integral:hermite:pol} and noting that the Hermite polynomials of even degree are even functions (as functions of the variable $Z_i$) and the fact that, conditional on $\mathcal{G}_{i-1}$, $T^{n+1}$ and $\bar{X}_i = L$, the law of the random vector $(Z_i, 2\rho_{i}-1) $ is a Bernoulli(1/2) r.v. with symmetric values $ \pm (\sigma_{i-1}^{-1}(L-\bar X_{i-1}),1) $ for $i\in \N_n$.
\end{proof}

The above lemma states that any weight for which its leading term satisfies that $\ell+k $ is odd (in the sense of highest order of time degeneration) will have a reduction in its time degeneration. Combining the extraction formula \eqref{eq:exta} with Lemma \ref{eq:exta}, one obtains a result which encompasses extraction, localization (at $L$) and reduction in time degeneration.
\begin{corol}
	\label{cor:2.1} Let $ f \in \mathscr{C}^1_p(\mathbb{R})$ and $(\ell, k,n) \in\N_0 \times \N^2 $. Let $c \in \mathscr{C}^k_b(\mathbb{R})$. Then, for any $i \in \N_n$, one has 
	\begin{align*}
	&\E\left[f(\bar{X}_i)\delta_L(\bar{X}_i)(2\rho_i-1)^{\ell}{{\mathcal{I}}}_{i}^{\ell}(c(\bar{X}_i)) | \mathcal{G}_{i-1}, T^{n+1}, N_T=n\right] \\
	&\quad \quad  =
	\sum_{\substack{{j}=0\\ \ell+k-{j}=even}}^k(-1)^{j}{k\choose {j}}c^{({j})}(L)\E\left[f(\bar{X}_t)\delta_L(\bar{X}_t)(2\rho-1)^{k-j}
	{{\mathcal{I}}}_{i}^{k-{j}}(1) | \mathcal{G}_{i-1}, T^{n+1}, N_T=n\right].
	\end{align*}
\end{corol}

\subsection{Appendix: Results about emergence and reduction of  jumps}
\label{app:jumps}
The first result is used in the proof of the probabilistic representation in Theorem \ref{th:3.2a} and is used to express that time integrals add jumps to the Poisson process. 
\begin{lem}
	\label{lem:cj}
	Let $ n \in \N $ and $ G:\{(t_1,\dots,t_{n+2}): 0<t_1<\dots<t_{n+1}<t_{n+2}:=T\}\rightarrow \mathbb{R}_+ $ be a measurable function such that
	$\E[\int_{\zeta_n}^TG(\zeta_1,\dots,\zeta_{n},s,T)\I_\seq{N_T=n}ds]<\infty  $. Then 
	\begin{align*}
	\E[\int_{\zeta_n}^TG(\zeta_1,\dots,\zeta_{n},s,T)\I_\seq{N_T=n}ds]=\lambda^{-1}
	\E[G(\zeta_1,\dots,\zeta_{n},\zeta_{n+1},T)\I_\seq{N_T=n+1}]
	\end{align*}
\end{lem}
The proof is straightforward and follows by rewriting the above expectations using that the conditional law of the jump times of a Poisson process given the number of jumps follows the same law as the order statistics for i.i.d. uniformly distributed r.v.'s. 

The next result is used for the reduction of jumps after the boundary merging procedure. For this reason we called it the time merging lemma.
\begin{lem}
	\label{lem:Unif}
	Let $n\in \N$. Let $ G:\{(t_1,...,t_{n}): 0<t_1<\dots<t_{n}<T\}\rightarrow \mathbb{R} $ be a measurable function such that for all $i \in \N_n$, $\E[(\zeta_{i+1}-\zeta_{i-1})^{-1}G(\zeta_1,\dots,\zeta_{i-1},\zeta_{i+1},\dots,\zeta_{N_T+1})\I_\seq{N_T=n}]<\infty  $. Then, for any $i \in \N_n$, one has
	\begin{align*}
	\E[(\zeta_{i+1}-\zeta_{i-1})^{-1}G(\zeta_1,\cdots,\zeta_{i-1},\zeta_{i+1},\dots,\zeta_{N_T+1})\I_\seq{N_T=n}]=\lambda\E[G(\zeta_1,\dots,\zeta_{N_T+1})\I_\seq{N_T=n-1}].
	\end{align*}
\end{lem}
\begin{proof}
	The proof follows from standard computations having at hand the two following important facts: conditionally on the event $\left\{ N_T=n , \zeta_1,\dots,\zeta_{i-1},\zeta_{i+1},\dots.,\zeta_{N_T} \right\}$, the distribution of $ \zeta_i $ is uniform on $ [\zeta_{i-1},\zeta_{i+1}] $ and $ n\P(N_T=n)=\lambda T\P(N_T=n-1)$. 
\end{proof}

\begin{remark}\label{remark12}
The above result will be used in the proof of Theorem \ref{prop:5.2} for $ i\in\mathbb{N}_n $ on the set $ \{N_T=n\} $ for some measurable functions $G$ of the following form
\begin{align*}
G(\zeta_1,\dots,\zeta_{i-1},\zeta_{i+1},\dots,\zeta_{N_T})= F(\bar{X}^{\partial}_{i-1, i+1}, \zeta_{i+1}, \cdots, \zeta_{n+1}) \times \I_{D^{\partial}_{i-1,i+1, n}}\ltheta_{i-1,i+1}^{a_i} \times \prod_{j=1}^{i-1} \I_{D_{j, n}} \bar \theta_j
\end{align*}
\noindent for some symbols $ a_i \in \left\{ \partial * e , \partial \circledast  e\right\} $ and for some measurable function $F(x, s_{i+1}, \cdots, T)$ 
\end{remark}

\end{document}